\documentclass[preprint,12pt,3p]{elsarticle}
\usepackage{microtype}
\usepackage[utf8]{inputenc}
\usepackage{amsmath, amsthm, amsfonts, amssymb}
\usepackage{appendix}
\usepackage{hyperref}
\usepackage{mathtools}
\usepackage{lipsum}
\usepackage[shortlabels]{enumitem}
\usepackage{tikz}
\usepackage{tikz-cd}
\usetikzlibrary{arrows, calc, cd, fit, decorations.pathmorphing}
\usepackage[bb=dsserif]{mathalpha}
\usepackage{alltt}
\usepackage{algorithm}
\usepackage[noend]{algpseudocode}
\usepackage{nicematrix}
\usepackage{tabularx}
\usepackage{subcaption}
\usepackage{arydshln}
\usepackage{pifont}
\usepackage{cleveref}
\usepackage{float} 

\makeatletter 
\renewcommand{\theHALG@line}{\thealgorithm.\arabic{ALG@line}}
\makeatother

\newcommand{\R}{\mathbb{R}}
\newcommand{\Z}{\mathbb{Z}}

\newcommand{\blank}{{-}}
\newcommand{\Hom}{\mathrm{Hom}}
\newcommand{\Mod}{\mathrm{Mod}}

\newcommand{\fin}{\mathrm{fin}}

\newcommand{\Ob}{\mathrm{Ob}}

\newcommand{\Fun}{\mathrm{Fun}}
\newcommand{\id}{\mathrm{id}}
\newcommand{\op}{\mathrm{op}}
\newcommand{\Dnk}[2]{\Delta_{\hspace{1pt}#1}^{\hspace{-2pt}(#2)}}
\DeclareMathOperator{\im}{im}

\DeclareMathOperator{\supp}{supp }
\DeclareMathOperator{\coker}{coker}
\DeclareMathOperator{\Proj}{proj}
\DeclareMathOperator{\Incl}{incl}

\DeclareMathOperator{\Span}{span}
\DeclareMathOperator{\St}{St}

\DeclareMathOperator{\Cl}{Cl}

\DeclareMathOperator{\stcplx}{SC}

\newcommand{\Rfunc}{\mathcal{I}} 

\newcommand{\Vect}{{\textbf{Vect}}}
\newcommand{\vect}{{\textbf{vect}}}

\definecolor{mypurple}{rgb}{0.44, 0.16, 0.39}
\definecolor{mygreen}{rgb}{0,0.55,0.5}
\definecolor{myred}{rgb}{0.9, 0.1, 0.1}
\definecolor{myorange}{rgb}{1.0, 0.5, 0.31}
\definecolor{myblue}{rgb}{0.1, 0.3, 0.7}

\newcommand\quotient[2]{
        \mathchoice
            {
                \text{\raise1ex\hbox{$#1$}\Big/\lower1ex\hbox{$#2$}}%
            }
            {
                #1\,/\,#2
            }
            {
                #1\,/\,#2
            }
            {
                #1\,/\,#2
            }
    }
    
\newcommand{\setdef}[2]{
		\left\{ #1 \,\middle|\,#2 \right\}
}

\makeatletter
\setbox0\hbox{$\xdef\scriptratio{\strip@pt\dimexpr
    \numexpr(\sf@size*65536)/\f@size sp}$}
\newcommand{\scriptveryshortarrow}[1][4pt]{{%
    \vcenter{\hbox{\rule[\scriptratio\dimexpr-.2pt\relax]
              {\scriptratio\dimexpr#1\relax}{\scriptratio\dimexpr.4pt\relax}}}%
  \mkern-4mu\hbox{\let\f@size\sf@size\usefont{U}{lasy}{m}{n}\symbol{41}}}}
\newcommand{\smap}[3]{{#1}(#2\leq #3)}

\algnewcommand\algorithmicforeach{\textbf{for each}}
\algdef{S}[FOR]{ForEach}[1]{\algorithmicforeach\ #1\ \algorithmicdo} 

\newcommand{\makeexact}{{\sc MakeExact}}

\theoremstyle{definition}

\newtheorem{theorem}{Theorem}
\numberwithin{theorem}{section}

\newtheorem{proposition}[theorem]{Proposition}

\newtheorem{lemma}[theorem]{Lemma}
\newtheorem{corollary}[theorem]{Corollary}

\newtheorem*{remark}{Remark}
\newtheorem{definition}[theorem]{Definition}
\newtheorem{example}[theorem]{Example}

\numberwithin{equation}{section}

\journal{Journal of Pure and Applied Algebra}
\begin{document}
\begin{frontmatter}
\tnotetext[t1]{This project has received funding from the European Research Council (ERC) under the European Union's Horizon 2020 research and innovation programme, grant no.\ 788183, from the Wittgenstein Prize, Austrian Science Fund (FWF), grant no.\ Z 342-N31, and from the DFG Collaborative Research Center TRR 109, `Discretization in Geometry and Dynamics', Austrian Science Fund (FWF), grant no.\ I 02979-N35.}
\title{Discrete Microlocal Morse Theory\tnoteref{t1}}

\author[inst1]{Adam Brown\corref{cor1}}
\ead{abrown@ist.ac.at}
\affiliation[inst1]{organization={Institute of Science and Technology Austria},
            addressline={Am Campus 1}, 
            city={Klosterneuburg},
            postcode={3400}, 
            country={Austria}}
\cortext[cor1]{Corresponding author}
\author[inst1]{Ond\v{r}ej Draganov}
\ead{ondrej.draganov@ist.ac.at}

\begin{abstract}
    We establish several results combining discrete Morse theory and microlocal sheaf theory in the setting of finite posets and simplicial complexes. Our primary tool is a computationally tractable description of the bounded derived category of sheaves on a poset with the Alexandrov topology.
    We prove that each bounded complex of sheaves on a finite poset admits a unique (up to isomorphism of complexes) minimal injective resolution, and we provide algorithms for computing minimal injective resolution of an injective complex, as well as several useful functors between derived categories of sheaves. For the constant sheaf on a simplicial complex, we give asymptotically tight bounds on the complexity of computing the minimal injective resolution using those algorithms. 
    Our main result is a novel definition of the discrete microsupport of a bounded complex of sheaves on a finite poset. We detail several foundational properties of the discrete microsupport, as well as a microlocal generalization of the discrete homological Morse theorem and Morse inequalities. 
\end{abstract}

\begin{keyword}
derived sheaf theory \sep discrete Morse theory \sep computational topology \sep microsupport \sep 55N30\sep  55N31\sep  55-08



\end{keyword}

\end{frontmatter}

\setcounter{tocdepth}{3}
\tableofcontents

\newpage

\section{Introduction}
A motif of twentieth-century mathematics is the investigation of global structures via local inquiry. Here, we build on two instances of this global-local theme. 

The first instance is Morse theory. Classically, Morse theory studies the global topological structure of a manifold by (locally) analyzing critical points of a real-valued smooth function on the manifold~\cite{Milnor,Morse}. Morse theory has since been generalized to stratified spaces~\cite{GoreskyMacPherson1988} and applied to many areas of mathematics. Forman applied these ideas to study cell complexes, and in doing so initiated the study of discrete Morse theory ~\cite{Forman2002,Forman1998}. Forman's reformulation has since led to many advances in applied and computational topology~\cite{Nanda2019,NandaTamakiTanaka,DotkoWagner,HarkerMischaikowMrozekNanda,MischaikowNanda,Skoldberg,Freij2009,BauerEdelsbrunner2016}. 

The second instance is sheaf theory, the poster child of the global-local motif. Introduced by Leray~\cite{Leray} and modernized by Grothendieck, with numerous contributions from Cartan, Serre, Verdier, Deligne, and many others, sheaf theory is a general framework combining homological algebra and topology (see `A Short History: Les d\'ebuts de la th\'eorie des faisceaux' by Houzel in \cite{KashiwaraSchapira1994} for a beautiful synopsis of the historical development of sheaf theory). Analogous to the recent expansion of discrete Morse theory, a growing body of research studies sheaves on cell complexes with applications in applied and computational topology~\cite{ MacPhersonPatel,Nanda,CurryPatel, CurryMukherjeeTurner,Curry2018,CurryGhristNanda2016,Curry2014}, although sheaves on finite posets were already studied earlier, e.g., in \cite{Deheuvels1962}. However, the utility of sheaf theory is most fully realized when working in the setting of derived categories. Computational aspects of derived cellular sheaf theory are explored in~\cite{BerkoukGinot,BerkoukPetit2021,BerkoukGinotOudot, Curry2014,Shepard1985}. Despite these advances, algorithms for many useful computations within derived cellular sheaf theory remain (to the best of our knowledge) underdeveloped.  

At the intersection of Morse theory and derived sheaf theory lies microlocal geometry. The `microlocal' perspective of sheaf theory originates from Sato's study of singularities within systems of linear differential equations~\cite{SatoKawaiKashiwara,Sato}. This perspective considers the propagation of various local structures (such as local solutions of differential equations or homological vanishing properties of sheaves) along different directions within (the cotangent bundle of) a manifold. Microlocal geometry (and its relatives) have far-reaching applications throughout modern mathematics, including intersection cohomology and perverse sheaf theory~\cite{Goresky2020,GoreskyMacPherson1981}, enumerative geometry~\cite{Behrend}, symplectic geometry~\cite{NadlerZaslow}, and representation theory~\cite{ABV} (to name a few). 

This paper introduces a discrete analog of two important concepts of microlocal geometry, originally introduced by Kashiwara--Schapira for manifolds. The first is the notion of microsupport, which quantifies the cohomological vanishing properties of sheaves on `local half-spaces'~\cite{KashiwaraSchapira1985,KashiwaraSchapira1982}. The second is a microlocal generalization of the Morse theorem and inequalities, which exchanges singular homology and Euler characteristics with hypercohomology and the Euler--Poincare index~\cite{SchapiraTose,Kashiwara1983,GoreskyMacPherson1981}. Before concluding this introduction with a summary of these results, we first give a brief sketch of an indispensable prerequisite: a systematic treatment of computational derived sheaf theory.

\paragraph{Computational Derived Sheaf Theory} Sheaves use algebra to model relationships between local and global properties of a topological space. When the topological space is a poset with the Alexandrov topology, a \emph{sheaf} (of finite-dimensional vector spaces), $F$, is defined by associating a finite-dimensional vector space, $F(\sigma)$, to each element, $\sigma$, and a linear map, $F(\sigma\le\tau):F(\sigma)\rightarrow F(\tau)$, to each relation, $\sigma\le \tau$ (subject to commutativity requirements, see Definition \ref{def:sheaf}). The utility of this definition is also its foil: the high level of generality encompasses many pathologies. 

A common strategy for analyzing such a complicated mathematical structure is to approximate or represent it with a collection of simpler, or at least more familiar, objects; the goal is to reframe questions concerning the complex structure as questions about the building blocks that represent it.
The first tool of this paper is a particular instance of this phenomenon: \emph{injective resolutions}. An \emph{injective resolution} represents a given sheaf (much like a Fourier series represents a periodic function) with an exact sequence, $0\rightarrow F\rightarrow I^0\rightarrow I^1\rightarrow I^2\rightarrow  \cdots$, of \emph{injective sheaves} $I^d$, which admit many desirable properties (see, for example, Lemma \ref{lem:coker-injective}, Proposition \ref{prop:decomposition_of_injective_sheaves}, and Lemma \ref{lem:maps_between_injective_sheaves}). Efficient algorithms for computing these sequences are a first step toward applying well-established and powerful theoretical results from derived sheaf theory to computational topology. In this paper, we aim to present this theory in an explicit and computationally amenable framework. 


Injective resolutions are used to study sheaves from the ‘derived’ perspective, i.e.\ as objects in a derived category (Definition \ref{def:derived-category}). 
These derived categories unify and generalize many variants of (co)homology, such as simplicial cohomology, Borel--Moore homology, intersection cohomology, etc. For example, simplicial cohomology (and level-set persistent cohomology, see~\cite{BerkoukGinotOudot}) can be computed from an injective resolution of the constant sheaf (see Example \ref{ex:constant-sheaf} and Section \ref{sec:examples}), illustrating that even an injective resolution of the constant sheaf contains subtle topological information. Several recent works point to the potential benefits of applying derived sheaf theory to the study of persistent homology~ \cite{BerkoukPetit2021,BerkoukGinotOudot,BerkoukGinot,Curry2014,Schapira,KS2021}. We approach this subject from a computational perspective to help bridge gaps between applied topology and derived sheaf theory. With this goal in mind, we aim to limit the mathematical prerequisites of our approach whenever possible (a choice that often comes at the cost of brevity).


\paragraph*{Main Results}
Broadly, this paper develops computational methods for the bounded derived category sheaves of finite-dimensional vector spaces on finite posets with the Alexandrov topology. Our main contributions are:
\begin{enumerate}
     \item We establish the existence and uniqueness of a \emph{minimal} injective resolution of a given complex of sheaves (Theorem \ref{thm:minimal_injective_complex}). We formulate an inductive algorithm for computing minimal injective resolutions (Algorithm~\ref{algo:injective_resolution_tail}) using elementary methods from linear algebra. For the constant sheaf on a simplicial complex, we give an asymptotically tight bound on the complexity of Algorithm \ref{algo:injective_resolution_tail} (Proposition \ref{prop:upper_bound_on_SC} and Corollary \ref{cor:complexity}). 
     \item We give an explicit description of the set of objects and morphisms for (a skeleton of) the bounded derived category of sheaves on a finite poset in terms of `labeled matrices' (Definition \ref{def:derived-category}), and provide algorithms for computing right derived (proper) pushforward and (proper) pullback functors along order preserving maps of posets (Section \ref{sec:derived-functors}). See~\ref{app:section:correspondence} for the correspondence between Definition~\ref{def:derived-category} and the standard definition of a derived category.
     \item We give a novel definition of the discrete microsupport for a complex of sheaves in the derived category of sheaves on a finite poset (Definition \ref{def:microsupport}), and prove a microlocal generalization of the discrete Morse theorem and inequalities (Theorem \ref{thm:microlocal-morse} and \ref{thm:micro-morse-ineq}).
     \item A Python software package implementing the algorithms developed in this paper is freely available at \url{https://github.com/OnDraganov/desc}~\cite{BrownDraganov}. It demonstrates the utility of our approach---for a simplicial complex with a few thousands of simplices, the computation of the minimal injective resolution of the constant sheaf as well as (proper) pushforwards/pullbacks of said resolution take just several seconds on a regular laptop.
     

\end{enumerate}
One pedagogical benefit of our approach is that while sheaf theory is an indispensable component of our proofs, the computations we describe make no fundamental use of sheaves or category theory, hopefully making these calculations more accessible to anyone with knowledge of linear algebra and elementary poset combinatorics.

\paragraph{Comparison to Prior Work} Both derived sheaf theory and Morse theory are rich subjects that have been thoroughly studied for several decades. There are many textbooks on sheaf theory \cite{Bredon1997,KashiwaraSchapira1994,Iversen} and several publications which study sheaves on finite topological spaces. In \cite{Shepard1985}, Shepard relates sheaves on finite cell complexes (viewed as posets) to the classical setting of constructible sheaves on stratified topological spaces. In \cite{Ladkani2008, Ladkani}, Ladkani studies the homological properties of finite posets and introduces combinatorial criteria guaranteeing derived equivalences between categories of sheaves. In \cite{Curry2014}, Curry establishes a connection between sheaf theory and persistent homology. More recently, several publications expand on the work initiated by Curry on applications of derived sheaf theory to persistent homology~\cite{BerkoukGinot,BerkoukPetit2021,KS2021,BerkoukGinotOudot,Schapira}. In a set of lecture notes, Goresky beautifully explains derived and perverse sheaf theory in both the Whitney stratified and cellular settings~\cite{Goresky2021}. 
There are also several closely related works on discrete Morse theory. Notably, Nanda~\cite{Nanda2019} describes a categorical localization procedure for Morse-theoretic cellular simplification which relates to sheaf propagation and our definition of microsupport (Definition ~\ref{def:microsupport}). Additionally, Sk\"oldberg gives an algebraic approach to discrete Morse theory by studying reductions of chain complexes that preserve the homotopy type of the complex~\cite{Skoldberg}. This work is reminiscent of the derived category approach we take in Section \ref{sec:derived_categories}. Motivated by these advances and the potential to develop new techniques for computational topology, we aim to establish preliminary results on computational aspects of derived sheaf theory for finite topological spaces. The contributions of this paper are the first of our knowledge to use discrete Morse theory to study the microsupport of derived sheaves from a computational perspective.



\paragraph{Acknowledgements} We thank François Petit for the useful discussions and helpful feedback on an earlier version of the manuscript. We thank Herbert Edelsbrunner and Sebastiano Cultrera di Montesano for teaching a course on discrete Morse theory at the Institute of Science and Technology Austria in the spring of 2022, during which much of this work was completed.  


\section{Background and terminology}

In this paper we study finite-dimensional vector space valued sheaves on finite posets, with particular emphasis on face posets of simplicial complexes. To clarify the objects of study and terminology, we begin by recalling definitions and well-known results, drawing heavily from \cite{Curry2014,Curry2018,Shepard1985,Ladkani2008,Goresky2021}. Throughout the paper we fix a field $k$. In the examples we often choose $k=\mathbb{Z}_2$, the two-element field.

A \emph{poset} is a partially ordered set. We usually denote it by $\Pi$ with relation $\geq$. We only consider finite posets. For elements $\pi,\tau$ in a poset $\Pi$, we write $\pi<_1\tau$ if $\pi\lneq\tau$ and there is no other element between $\pi$ and $\tau$.
We use some of the standard terminology from simplicial complexes for general posets: the \emph{star} of an element $\sigma$ is $\St\sigma=\setdef{\tau\in\Pi}{\sigma\leq\tau}$, the \emph{boundary} is $\textrm{bnd}(\sigma)=\setdef{\tau\in\Pi}{\tau<_1\sigma}$, and the \emph{coboundary} is $\textrm{cobnd}(\sigma)=\setdef{\tau\in\Pi}{\sigma<_1\tau}$. A chain $\pi_0<\dots<\pi_n$ has \emph{length} $n$, and the \emph{height} of a poset is the greatest length over all its chains. This corresponds to dimension if $\Pi=(\Sigma,\subseteq)$ is the face poset of a simplicial complex. An \emph{abstract simplicial complex} $\Sigma$ on a vertex set $V$ is a system of subsets of $V$ such that $A\subseteq B\in \Sigma$ implies $A\in\Sigma$. The combinatorial aspects of posets are tied with geometrical concepts via Alexandrov topology.

\begin{definition}\label{def:alexandrov}
A finite poset $\Pi$ is a $T_0$-topological space with the \emph{Alexandrov topology}, where a set is \emph{open} when it is upwards closed: \[
U\subset\Pi \text{ is open if and only if } \tau\in U\text{ implies }\gamma\in U\text{ for each }\gamma\ge \tau.
\] 
The closure of a set $Z\subseteq \Pi$ is $\Cl Z = \setdef{\pi\in\Pi}{\pi\leq\sigma\text{ for some $\sigma\in Z$}}$.

\end{definition}

We study vector space representations of posets from the perspective of sheaf theory. Since we only consider finite dimensional vector spaces, what we call a \emph{sheaf} is equivalently also a quiver representation with relations, a module over a bound quiver algebra, a persistence module or a functor from the poset seen as a category to the category of vector spaces (see \cite[Theorem 4.2.10]{Curry2014}).

\begin{definition}
\label{def:sheaf}
A \emph{sheaf} $F$ on a finite poset $\Pi$ is an assignment of a finite-dimensional $k$-vector space $F(\pi)$ to each element $\pi\in\Pi$, and an assignment of a linear map
\[
F(\tau\le\gamma):F(\tau)\rightarrow F(\gamma)
\]
to each face relation $(\tau\le\gamma)\in\Pi$, such that 
\begin{enumerate}
    \item $F(\tau\le\tau)$ is the identity map, $\id_{F(\tau)}$, and
    \item $F(\tau\le\gamma)\circ F(\sigma\le\tau) = F(\sigma\le\gamma)$, for each triple $\sigma\le\tau\le\gamma\in\Pi$. 
\end{enumerate}
The \emph{support} of $F$ is the set of elements $\pi\in\Pi$ for which $F(\pi)\neq 0$.
\end{definition}

\begin{example}\label{ex:constant-sheaf}
The \emph{constant sheaf}, denoted $k_\Pi$, on a poset $\Pi$, assigns to each element $\pi\in\Pi$ the one-dimensional vector space, $k$, and to each relation, $(\pi\le \tau)\in\Pi$, the identity map $\id_k$. 
\end{example}

\begin{definition}
\label{def:natural-transformation}
A \emph{natural transformation}, $\eta:F\rightarrow G$, between two sheaves on $\Pi$, is a collection of linear maps 
$
\eta(\pi):F(\pi)\rightarrow G(\pi)
$
for each $\pi\in\Pi$, such that 
\[
G(\tau\le\gamma)\circ \eta(\tau) = \eta(\gamma)\circ F(\tau\le\gamma),
\]
for each $(\tau\le\gamma)\in \Pi$. For a natural transformation $\eta:F\rightarrow G$, the kernel, cokernel, image, and coimage are taken point-wise, defining sheaves on $\Pi$: 
\[(\ker\eta)(\pi):=\ker (\eta(\pi)),\qquad (\ker\eta)(\pi\le\tau):=F(\pi\le\tau)\vert_{\ker\eta(\pi)}.\]
Moreover, if $\ker\eta(\pi)=0$ for each $\pi\in\Pi$, we say that $\eta$ is \emph{injective}. We write $G/F \coloneqq\coker\eta$ if $\eta$ is an injection clear from the context.
\end{definition}

As kernels, images and cokernels, other concepts can also be easily defined point-wise on sheaves. Importantly the \emph{direct sum} of sheaves is defined as the direct sum at each $\pi\in\Pi$.

\begin{definition}\label{def:injective_sheaf_2}
A sheaf $I$ is called \emph{injective} if for each injective natural transformation $f: A\hookrightarrow B$, any given natural transformation $\alpha: A\rightarrow I$ can be extended to $\beta: B\rightarrow I$ so that $\alpha = \beta \circ f$.
\end{definition}

This condition is always satisfied for sheaves over a single point space (i.e., the assignment of a single vector space to a point): we can extend any linear map on a subspace to the whole space by mapping a complement space to 0. See appendix for examples of non-injective sheaves: Example~\ref{ex:non-injective_sheaf}, Example~\ref{ex:non-injective_sheaf_2}.

In the case of finite-dimensional sheaves over finite posets, each injective sheaf is created from simple building blocks---as a direct sum of \emph{indecomposable} injective sheaves. This allows us to store them efficiently, as discussed in Section~\ref{sec:labeled_matrices}. We recall standard definitions and results that can be found in \cite{Curry2014} and \cite{Shepard1985}. Although phrased for sheaves on cell complexes in those references, the generalization to any finite poset is straightforward. We revisit proofs of the standard claims below in \ref{app:sec:proofs_injective_sheaves}.

\begin{definition}[{cf.\,\cite[Definition 7.1.3]{Curry2014}, \cite[Chapter 23]{Goresky2021}}]
\label{def:indecomposable_injective_sheaf}
    For each $\pi\in\Pi$, we define an \emph{indecomposable injective sheaf} $[\pi]$ as
    
    \begin{minipage}{.4\textwidth}
        \begin{align*}
            [\pi](\sigma) \coloneqq
            \begin{cases}
                k &\text{ if $\sigma\leq \pi$,} \\
                0 &\text{ otherwise,}
            \end{cases}
        \end{align*}
    \end{minipage}
    \begin{minipage}{.1\textwidth}
        \vspace{.4cm}
        with
    \end{minipage}
    \begin{minipage}{.4\textwidth}
        \begin{align*}
            [\pi](\sigma\le \tau) \coloneqq
            \begin{cases}
                \id &\text{ if $\sigma\le\tau\leq \pi$,} \\
                0 &\text{ otherwise.}
            \end{cases}
        \end{align*}
    \end{minipage}
  
\end{definition}
For $n\in\mathbb{Z}_{\ge 0}$, we denote by $[\pi]^n$, the direct sum $\bigoplus_{j=1}^n[\pi]$. For a vector space $V$, we denote by $[\pi]^V$, the sheaf $[\pi]^{\dim V}$, with an implicitly fixed isomorphism between $k^{\dim V}$ and $V$.

\begin{lemma}[{cf.\,\cite[Lemma 7.1.5]{Curry2014}}]\label{lem:elementary_injective_sheaf}
    Indecomposable injective sheaves are injective.
\end{lemma}

\begin{lemma}[{cf.\,\cite[Lemma 1.3.1]{Shepard1985}}]
\label{lem:coker-injective}
    A direct sum of injective sheaves is injective. Additionally, if $I\xhookrightarrow{\alpha} J$ is an injective natural transformation with $I,J$ injective sheaves, then $J\cong I\oplus \coker\alpha  $, and $\coker\alpha$ is an injective sheaf. 
\end{lemma}

\begin{proposition}[{cf.\,\cite[Lemma 7.1.6]{Curry2014}, \cite[Theorem 1.3.2]{Shepard1985}}] 
\label{prop:decomposition_of_injective_sheaves}
  Each injective sheaf on a finite poset is isomorphic to a direct sum of indecomposable injective sheaves.
\end{proposition}

The central objects studied in this paper are (cochain) complexes of sheaves. Unless otherwise stated, all complexes we consider are bounded.

\begin{definition}\label{def:complex_of_sheaves}
A (bounded) \emph{complex of sheaves}, denoted $(A^\bullet, \mu^\bullet)$, is a sequence of sheaves $A^d$ and natural transformations $\mu^d$
\[
\cdots\rightarrow A^d\xrightarrow{\mu^d}A^{d+1}\xrightarrow{\mu^{d+1}}A^{d+2}\xrightarrow{\mu^{d+2}}\cdots
\]  
such that $\mu^{d+1}\circ\mu^d=0$ for each $d$, and $A^d=0$ for $|d|$ sufficiently large. A complex of sheaves is \emph{exact} if $\im\mu^d=\ker\mu^{d+1}$ for each $d$. A morphism $\alpha^\bullet:A^\bullet\rightarrow B^\bullet$ between complexes of sheaves is a collection of natural transformations $\alpha^d:A^d\rightarrow B^d$ such that the diagrams commute: 
\begin{center}
   \begin{tikzcd}
        A^d\ar{r}{\mu^d} \ar{d}{\alpha^d} & A^{d+1} \ar{d}{\alpha^{d+1}} \\
        B^d \ar{r}{\nu^d} & B^{d+1}
    \end{tikzcd}
\end{center}
A complex of sheaves is \emph{injective} if each sheaf in it is injective.
\end{definition}

We often abbreviate complex of sheaves just to \emph{complex}. Sometimes we abuse the notation and write $A^\bullet = (A^\bullet, \mu^\bullet)$ when the symbols for natural transformations are clear from context or not used.

The tools discussed in this paper are developed to study (co)homology. One version of it, defined for a complex, is the cohomology sheaf in degree $d$.

\begin{definition} \label{def:cohomology_sheaf}
  For a complex of sheaves $(A^\bullet,\mu^\bullet)$ on $\Pi$, define, for each $d\in\mathbb{Z}$, the \emph{cohomology sheaf} $H^d(A^\bullet)$ by
  \[
  H^d(A^\bullet)(\sigma) \coloneqq \ker\mu^d(\sigma)/\im\mu^{d-1}(\sigma).  
  \]
   Because $\mu^\bullet$ is a morphism of complexes, the linear map $A^d(\sigma\le\tau)$ maps $\ker\mu^d(\sigma)$ to $\ker\mu^d(\tau)$, and $\im\mu^{d-1}(\sigma)$ to $\im\mu^{d-1}(\tau)$. Therefore, $A^d(\sigma\le\tau)$ induce maps \[H^d(A^\bullet)(\sigma\le\tau): H^d(A^\bullet)(\sigma)\rightarrow H^d(A^\bullet)(\tau).\]
   These linear maps give $H^d(A^\bullet)$ the structure of a sheaf on $\Pi$.
\end{definition}

We finish the section by describing the structure of natural transformations between injective sheaves---this simple observation is crucial for the algorithmic treatment of injective complexes. There are two ingredients: a natural transformation between two direct sums uniquely decomposes to natural transformations between pairs of summands, and a natural transformation between powers of indecomposable injective sheaves is fully described by a single linear map, which can only be non-zero when the support of the codomain is contained in the support of the domain.

\begin{lemma}\label{lem:maps_between_injective_sheaves}
Let $\eta: I\rightarrow J$ be a natural transformation between two injective sheaves on a poset $\Pi$, and let
\[
I \cong \bigoplus_{\pi\in\Pi} [\pi]^{p_\pi}\hspace{15pt}\text{ and }\hspace{15pt}
J \cong \bigoplus_{\sigma\in\Pi} [\sigma]^{s_\sigma}
\]
be the decompositions of $I$ and $J$ into indecomposable injective sheaves with multiplicities $p_\pi$, $s_\sigma$. Then $\eta$ can be uniquely described by a collection of linear maps $f_{\pi\sigma}: k^{p_\pi}\rightarrow k^{s_\sigma}$, one for each pair $\pi\leq\sigma$ in $\Pi$. On the other hand, each such collection of linear maps defines a natural transformation. In other words,
\[\Hom(I,J) \ \ \cong\ \  \bigoplus_{\sigma\leq\pi} \Hom(k^{p_\pi},k^{s_\sigma})\cong \bigoplus_{\sigma\leq\pi}k^{p_\pi s_\sigma}, \]
where $\Hom(I,J)$ denotes the set of natural transformations from $I$ to $J$ and $\Hom(k^{p_\pi},k^{s_\sigma})$ denotes the set of linear transformations from $k^{p_\pi}$ to $k^{s_\sigma}$. 
\end{lemma}

\begin{proof}
Denote by $\Proj_A$ and $\Incl_A$ the projection onto and inclusion of a direct summand $A$, respectively. Then the map $\eta: I\rightarrow J$ decomposes into a sum of maps between the powers of indecomposable injective sheaves: \[
\eta = \sum_{\pi,\sigma\in\Pi}\eta_{\pi\sigma}, \hspace{10mm}
\eta_{\pi\sigma}=\Proj_{[\sigma]^{s_\sigma}}\circ\, \eta \circ \Incl_{[\pi]^{p_\pi}}.
\]
Consider $\tau\in\Pi$. If $\tau\not\leq\sigma$, then $\eta_{\pi\sigma}(\tau)=0$.
Otherwise, $
\eta_{\pi\sigma}(\tau)
= [\sigma]^{s_\sigma}(\tau\leq\sigma) \circ \eta_{\pi\sigma}(\tau)
= \eta_{\pi\sigma}(\sigma) \circ [\pi]^{p_\pi}(\tau\leq\sigma)$, which is equal to $\eta_{\pi\sigma}(\sigma)$ if $\sigma\leq\pi$, and $0$ otherwise. This shows that the natural transformation $\eta_{\pi\sigma}$ is determined by the single linear map $f_{\pi\sigma}\coloneqq\eta_{\pi\sigma}(\sigma)$. Moreover, this map is necessarily $0$ whenever $\sigma\not\leq \pi$, and it can be any linear map otherwise.
\end{proof}

\begin{example}\label{ex:natural_transformation_decomposition}
    Consider the three element ``meet'' poset, $\Pi$, given by two relations---$\sigma < \pi$ and $\tau < \pi$---and two injective sheaves on $\Pi$: $I \cong [\pi]^2 \oplus [\sigma]$ and $J \cong [\pi] \oplus [\sigma] \oplus [\tau]^2$. Then any natural transformation $\eta:I\rightarrow J$ is described by four linear maps: $f_{\pi\pi}: k^2\rightarrow k$, $f_{\pi\sigma}: k^2\rightarrow k$, $f_{\pi\tau}: k^2\rightarrow k^2$, and $f_{\sigma\sigma}: k\rightarrow k$. Note that we do not consider, e.g., $f_{\sigma\pi}$, because $\pi\not\leq \sigma$, and so the only natural transformation $[\sigma] \rightarrow [\pi]$ is $0$.
\end{example}

\section{Fixing a Basis of a Complex of Injective Sheaves}\label{sec:labeled_matrices}

The central concept for our algorithmic treatment of injective complexes is a \emph{labeled matrix} representing a natural transformation between injective sheaves. The same way we can represent a linear map by a matrix once we fix bases of the vector spaces, we can represent natural transformations between injective sheaves by matrices, once we fix specific bases. By definition, a natural transformation between sheaves, $\eta: I\rightarrow J$, is described by many matrices: one $\eta(\pi)$ for each $\pi\in\Pi$. By Lemma~\ref{lem:maps_between_injective_sheaves}, much less information is necessary in case of injective sheaves. In this section we organize this information in a single matrix and show that multiplying the defined matrices corresponds to composition of natural transformations. First, we need to make sure to fix well behaved systems of bases for a given injective sheaf.

\begin{definition}\label{def:compatible_basis}
    Let $I = \bigoplus_{i=1}^m [\pi_i]$ be a fixed decomposition of an injective sheaf on a poset $\Pi$, with possible repetitions of $\pi_i$. Let $U_{\pi}$ be a basis of the vector space $I(\pi)$ for each $\pi\in\Pi$. We call the system of bases $(U_{\pi})_{\pi\in\Pi}$ a \emph{decomposition-compatible basis} of $I$, if each map $I(\sigma \leq \pi)$ with respect to bases $U_{\sigma}$, $U_{\pi}$ is a projection to a subset of coordinates, respecting the order of the decomposition. As all systems of bases we consider are decomposition-compatible, we also say just \emph{basis} of $I$.
\end{definition}

\begin{example}
    If $\sigma < \pi$ and $I = [\pi] \oplus [\sigma] \oplus [\sigma]$, then the matrix representing $I(\sigma \leq \pi)$ with respect to any decomposition-compatible basis is $(1\  0\  0)$. Note that this still leaves some freedom. For example, if we multiply the basis vector for the ``$[\pi]$-coordinate'' by the same non-zero constant at each $\tau\leq\pi$, we get a formally different system of bases, but the matrices stay the same.
\end{example}

\subsection{Poset Labeled Matrices}\label{sec:poset_labeled_matrices}

We describe how to represent a natural transformation between two injective sheaves with fixed decomposition-compatible bases as a single matrix with labeled columns and rows. The idea is to simply record the information about the individual maps $f_{ij}$ from Lemma~\ref{lem:maps_between_injective_sheaves} in a single matrix. Recall that such a map $f_{\pi\sigma}: [\pi]^{p_\pi}\rightarrow [\sigma]^{s_\sigma}$ can be non-zero only if $\sigma\leq \pi$.

\begin{definition}\label{def:labeled-matrix}
    For a poset $\Pi$, a \emph{$\Pi$-labeled matrix} is a matrix $M$ together with a labeling of each column and row by elements of $\Pi$, such that the entry on an intersection of a column labeled by $\pi$ and a row labeled by $\sigma$ is non-zero only if $\sigma\leq\pi$.
    
    Columns and rows are labeled independently, the labels can be repeated, and not all labels need to be used. If $\Pi$ is clear from the context, we also say just \emph{labeled matrix}.

    For $\pi,\sigma\in\Pi$, we denote by $M[\sigma,\pi]$ the submatrix of $M$ obtained by deleting all rows not labeled by $\sigma$ and all columns not labeled by $\pi$. The same way for $S,P\subseteq\Pi$ we denote by $M[S,P]$ the submatrix obtained by deleting rows and columns labeled by simplices not in $S$ and $P$, respectively.

    We also allow a labeled matrix that has no rows---corresponding to a morphism into a trivial sheaf---or no columns---corresponding to a morphism from a trivial sheaf. Both are special cases of zero matrices.
\end{definition}

Note that the difference between a $0\times 0$ and a $0\times n$ matrix is more than just a formality---their kernels are $0$- and $n$-dimensional, respectively. This will be important when we discuss exactness later.

\begin{example}\label{ex:labeled_matrix}
    Let $\Pi$ be a three element poset with $\sigma, \tau < \pi$ as in Example~\ref{ex:natural_transformation_decomposition}. Then
    \begin{center}
    \includegraphics[width=40mm]{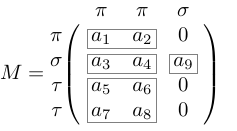}
    \end{center}
    is a $\Pi$-labeled matrix. Each of the outlined blocks is a matrix representing one of the four linear maps from Example~\ref{ex:natural_transformation_decomposition}, e.g., $M[\sigma, \pi] = (a_3 \ a_4)$ describes $f_{\pi\sigma}$. Note that the zeros are all forced, as non-zero values are only allowed for entries with `row-label $\leq$ column-label`.
\end{example}

\begin{example}
    The labeled matrices are, in a way, ``lower-triangular'': if $\Pi$ is linearly ordered, $\pi_n > \dots > \pi_1$, then a labeled matrix with both columns and rows labels exactly $(\pi_n, \pi_{n-1}, \dots, \pi_1)$ is a lower-triangular matrix.
\end{example}

\begin{example}
    The elements of an \emph{incidence algebra} of a poset $\Pi$ over a field $k$ are special cases of labeled matrices. By definition, such an element is a function, $f$, giving each interval $[\sigma, \pi]$ in $\Pi$ a value in $k$. This is equivalently described by a labeled matrix, $M$, with both columns and rows labeled by each element of $\Pi$ once---the value $f([\sigma, \pi])$ is then the single entry of $M[\sigma, \pi]$. The operations on the incidence algebra agree with the operations we define on labeled matrices. This relation shows that the incidence algebra with field coefficients can be viewed as the space of endomorpishms of the injective sheaf $\bigoplus_{\pi\in\Pi}[\pi]$.
\end{example}

See also Example~\ref{ex:submatrices-of-labeled-matrix} below. For even more examples see matrices in Section~\ref{sec:examples} labeled by simplices ordered by inclusion, e.g., $13<123$, with entries in the field $k=\Z_2$.

\begin{definition}
    A $\Pi$-labeled matrix $M$ \emph{represents} a natural transformation $\eta:I\rightarrow J$ between two injective sheaves if the following conditions hold:
    \begin{itemize}
        \item $I \cong \bigoplus_{i=1}^m [\pi_i]$, where $\pi_1,\dots,\pi_m$ are the labels of columns of $M$ (in that order),
        \item $J \cong \bigoplus_{j=1}^n [\sigma_j]$, where $\sigma_1,\dots,\sigma_n$ are the labels of rows of $M$ (in that order),
        \item For every $\sigma \leq \pi$ from $\Pi$, the matrix $M[\sigma,\pi]$ represents the map $f_{\sigma\pi}=\eta_{\sigma\pi}(\sigma)$ described in Lemma~\ref{lem:maps_between_injective_sheaves} with respect to fixed decomposition-compatible bases of $I$ and $J$.
    \end{itemize}
\end{definition}

Once a decomposition-compatible basis is chosen for $I$ and $J$, there is a one-to-one correspondence between natural transformations $\eta: I\rightarrow J$, and labeled matrices with the labeling from the decompositions of $I$ and $J$. This follows from Lemma~\ref{lem:maps_between_injective_sheaves}.

The information can be condensed in this way, because the matrices $\eta(\tau)$ overlap. If $M$ represents $\eta$, then the matrix representing a linear map $\eta(\tau)$ is a submatrix of $M$. Recall that $\St\tau$ is the set of elements greater or equal to $\tau$.

\begin{lemma}
    If a labeled matrix $M$ represents a natural transformation $\eta$, and $\tau\in\Pi$, the submatrix $M[\St\tau,\St\tau]$ is the matrix of the linear map $\eta(\tau)$ with respect to the fixed decomposition-compatible bases.
\end{lemma}
\begin{proof}
    As in Lemma~\ref{lem:maps_between_injective_sheaves}, we have $\eta = \bigoplus_{\sigma\leq\pi} \eta_{\pi\sigma}$. If $p_\pi$ is the multiplicity of $[\pi]$ in the decomposition of $I$, then \[
        \eta(\tau) = \bigoplus_{\sigma\leq\pi} \eta_{\pi\sigma}(\tau) = \bigoplus_{\tau\leq\sigma\leq\pi} \eta_{\pi\sigma}(\sigma) \circ [\pi]^{p_{\pi}}(\tau \leq \sigma).
    \]
    By definition, $M[\sigma, \pi]$ is the matrix of $\eta_{\pi\sigma}(\sigma)$ with respect to fixed decomposition-compatible bases for $I$ and $J$. Moreover, restriction maps $[\pi]^{p_{\pi}}[\tau\leq\sigma]$ are projections to coordinates consistent with the decomposition. Therefore, $M[\St\tau, \St\tau]$ is the matrix of $\eta(\tau)$ with respect to those fixed bases.
\end{proof}

\begin{example}\label{ex:submatrix_of_labeled_matrix}
    The matrix $M$ from the previous Example~\ref{ex:labeled_matrix} represents some natural transformation $\eta: I\rightarrow J$. The linear map $\eta(\tau)$ with respect to the fixed bases of $I$ and $J$ is {
    \footnotesize \[
        M[\St\tau, \St\tau] = M[\{\tau, \pi\}, \{\tau, \pi\}] =
        \begin{pNiceArray}{cc}
            a_1 & a_2 \\
            a_5 & a_6 \\
            a_7 & a_8
        \end{pNiceArray}.
    \]}
\end{example}

\begin{example}\label{ex:submatrices-of-labeled-matrix}
    A labeled matrix $M$ representing a natural transformation $\eta$ on a poset $\Pi$, a face poset of a simplicial complex on vertices $\{1,2,3,4,5,6\}$ ordered by inclusion---for example $134<1234$. The meaning of $\eta$ is revealed later in Example~\ref{ex:4-simplex_with_extra_edges}. Here we only show how matrices of linear maps $\eta(\sigma)$ overlap in a labeled matrix $M$.
    \[
        \includegraphics[width=120mm]{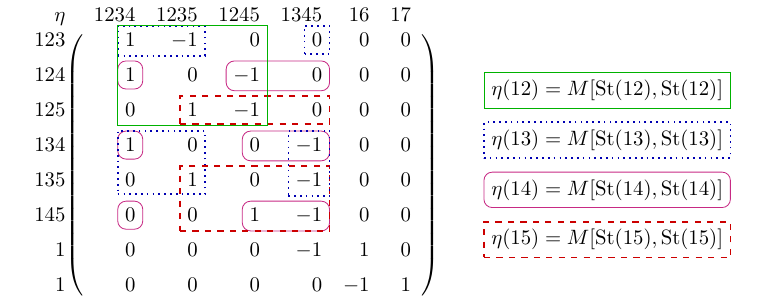}
    \]
\end{example}

Multiplying labeled matrices $N$ and $M$ makes sense whenever the column labels of $N$ exactly match the row labels of $M$. In this case, $N\!\cdot\! M$ is the product of the unlabeled matrices with columns labeled as columns of $M$, and rows labeled as row labels of $N$. the operation corresponds to composition of natural transformations.

\begin{lemma}\label{lem:multiplying_labeled_matrices}
    Let $I,J,K$ be injective sheaves, each with a fixed basis. If $M$ represents $\eta:I\rightarrow J$ and $N$ represents $\lambda: J\rightarrow K$ with respect to the fixed bases, then $N\!\cdot\! M$ represents $\lambda \circ \eta: I\rightarrow K$.
\end{lemma}
\begin{proof}
    Note first that $NM$ is a valid labeled matrix, i.e., it respects the non-zero elements condition: if $0 \neq NM[\sigma, \pi] = N[\sigma, \Pi]\!\cdot\! M[\Pi, \pi]$, then there is some $\tau$ such that $N[\sigma, \tau] \neq 0$ and $M[\tau, \pi] \neq 0$, which implies $\sigma \leq \tau \leq \pi$.

    For any $\tau\in\Pi$, note that $N[\St\tau, \Pi\setminus\St\tau] = 0$ by the non-zero elements condition for $N$. Then $(NM)[\St\tau,\St\tau] = N[\St\tau,\Pi]\cdot M[\Pi,\St\tau] = N[\St\tau,\St\tau]\cdot M[\St\tau,\St\tau]$, which implies that $N\!\cdot\!M$ represents $\lambda\circ\eta$ with respect to the fixed bases on $I$ and $K$.
\end{proof}

\begin{example}
    Consider $M$ from Example~\ref{ex:labeled_matrix}, and $N$ a labeled matrix representing a natural transformation $\lambda: J\rightarrow K\cong [\pi] \oplus [\sigma]$. The multiplication $N\!\cdot\! M$ is as follows:
    \begin{center}
        \includegraphics[width=.85\textwidth]{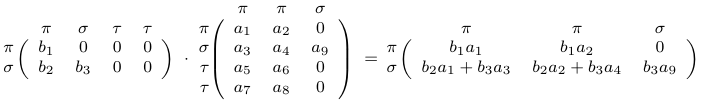}
    \end{center}
\end{example}

\subsection{Change of Bases for Labeled Matrix Representations}\label{sec:change_of_basis_for_labeled_matrices}

Just as for matrices between vector spaces, changing the basis of domain or codomain of a natural transformation $\eta: I\rightarrow J$ corresponds to multiplication of the representing labeled matrix, $M$, by another labeled matrix from the right or the left side, respectively. Consider $U$, $U'$ two different bases of $I$. Both can be decomposition-compatible with the same decomposition or with different decompositions---that is, the order of summands can differ. Either way, we can represent the identity $id: I\rightarrow I$ by a labeled matrix, $R$, with respect to those two bases $U$ and $U'$. Then $R$ is a change of basis labeled matrix. That is, if $M$ and $M'$ are labeled matrix representations of $\eta$ with respect to $U$ and $U'$, respectively, then $M' \!\cdot\! R = M$. Analogously, if we change bases in the codomain, $J$, it corresponds to matrix multiplication from the left side.

The above observation implies that choosing a different basis for the domain or codomain corresponds to column or row operations in the labeled matrix representations, respectively. Since $R$ is itself a labeled matrix, $R[\sigma,\pi]$ is non-zero only if $\sigma\leq\pi$. This restricts the set of operations that are allowed.

\begin{definition}\label{def:allowed-operations}
    For a $\Pi$-labeled matrix, the \emph{allowed elementary column operations} are
    \begin{itemize}
        \item adding a multiple of a column labeled by $\sigma$ to a column labeled by $\pi$ for $\sigma\leq\pi$,
        \item multiplying a column by a non-zero constant,
        \item swapping two columns,
    \end{itemize}
    and the \emph{allowed elementary row operations} are
    \begin{itemize}
        \item adding a multiple of a row labeled by $\pi$ to a row labeled by $\sigma$ for $\sigma\leq\pi$,
        \item multiplying a row by a non-zero constant,
        \item swapping two rows.
    \end{itemize}
    An allowed elementary row (column) operation corresponds to multiplication from the left (right) side by an \emph{elementary labeled matrix}, which is a special case of an elementary matrix in the standard sense.
\end{definition}

\begin{example}\label{ex:change_of_basis}
    Let $\Pi$ be a poset with $\sigma < \pi$. Consider $[\pi] \overset{\eta^0}{\rightarrow} [\pi]\oplus [\sigma] \overset{\eta^1}{\rightarrow} [\sigma]$ with $\eta^i$ represented by $M^i$. Change the basis for the middle sheaf so that the identity is represented by $R$. Then the new matrices $\hat{M}^0 = R\cdot M^0$, $\hat{M}^1 = M^1\cdot R^{-1}$ represent $\eta^0$, $\eta^1$ with respect to the new basis, respectively. We performed one allowed  elementary operation on each $M^i$.
    
    \noindent\begin{minipage}{\textwidth}
        \centering
        \includegraphics[width=160mm]{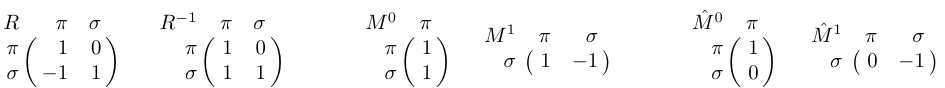}
    \end{minipage}
\end{example}

Any change of basis matrix, $R$, is a composition of allowed elementary row operations or column operations. Indeed, we can reduce any such $R$ to the identity matrix using allowed row or column operations. With either option, we first reduce the square ``diagonal'' blocks $R[\pi,\pi]$ to identity matrices---this is always possible since we impose no restrictions on operations on those blocks, and $R[\pi, \pi]$ represents an isomorphism of vector spaces. Then we can clear the rest of the matrix, since $R[\sigma,\pi]$ can be non-zero only if $\sigma\leq\pi$, in which case we can clear this block by adding rows labeled by $\pi$ or by adding columns labeled by $\sigma$. In particular, this also implies that the change of bases labeled matrices are invertible, and the inverse is, again, a labeled matrix. 

\subsection{Complexes of Labeled Matrices}

One labeled matrix represents a single natural transformation between injective sheaves. A series of labeled matrices represents an injective complex.

\begin{definition}
    A sequence $M^\bullet=(\dots, M^d, M^{d+1}, \dots)$ of labeled matrices is a \emph{complex of labeled matrices} if 
    \begin{itemize}
        \item columns of $M^{d+1}$ have the same labeling as the rows of $M^d$ (including order) for all $d$,
        \item $M^{d+1}\cdot M^d=0$ for all $d$.
    \end{itemize}
    The sequences we consider are finite, but formally extend indefinitely by empty matrices. Recall that an empty labeled matrix can be a $0\times 0$ matrix, but also an $n\times 0$ matrix, representing a natural transformation from a trivial to a non-trivial sheaf, or a $0\times n$ matrix.
\end{definition}

Note that $M^{d+1}\cdot M^d=0$ iff $M^{d+1}[\St\tau, \St\tau] \cdot M^d[\St\tau, \St\tau]=0$ for all $\tau\in\Pi$. This is because any row in either matrix is labeled by some $\tau$, and is non-zero only at columns labeled by $\St\tau$.

\begin{example}\label{ex:complex_of_labeled_matrices}
    A complex of labeled matrices on $\Dnk{3}{2}$, the face poset of the 2-skeleton of a tetrahedron $1234$ (for example $14\leq 124$, as $\{1,4\}\subseteq\{1, 2, 4\}$):

    \begin{center}
        \includegraphics[trim={6mm 0 6mm 0}, clip, width=\textwidth]{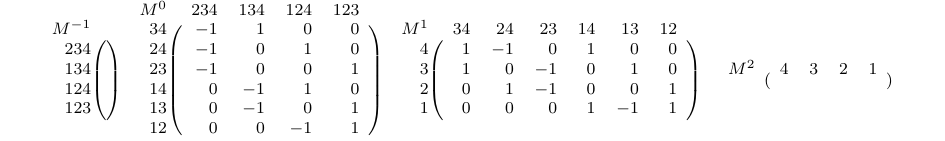}
    \end{center}
\end{example}

\begin{definition}\label{def:complex_of_labeled_matrices}
    An injective complex $(I^\bullet, \eta^\bullet)$ is \emph{represented} by a complex of labeled matrices $M^\bullet$, if the following hold:
    
    \begin{itemize}
        \item $I^d$ is isomorphic to the sum of indecomposables given by the labeling of the columns of $\eta^d$ for every~$d$,
        \item $M^{d-1}$ and $M^{d}$ are representations of the natural transformations $\eta^{d-1}$ and $\eta^{d}$ with respect to the same decomposition-compatible basis of $I^d$, for every~$d$.
    \end{itemize}
\end{definition}
\noindent Note that if $0$ is the smallest degree with a non-zero entry, $I^0$, then we formally require $M^\bullet$ to start with a $n \times 0$ matrix $M^{-1}$ as shown in Example~\ref{ex:complex_of_labeled_matrices}. In most other examples we do not show the empty matrices at the start or the end---their labeling is determined by the columns of the first or rows of the last non-empty matrix, respectively.

\begin{lemma}\label{lem:change_of_basis_for_complex}
Changing the basis of $I^d$ in a complex corresponds to performing column operations on $M^d$, and at the same time corresponding inverse row operations on $M^{d-1}$. If $R$ is the labeled matrix representing the change of basis of $I^d$, then the new matrices with codomain and domain $I^d$ are $R\cdot M^{d-1}$ and $M^{d}\cdot R^{-1}$, respectively. For example, if we add the $i$-th column to the $j$-th in $M^d$, we need to subtract the $j$-th row from the $i$-th in $M^{d-1}$.
\end{lemma}

\begin{remark}
    To simplify notation, in the rest of the paper, we use the same symbol both for a natural transformation and for a labeled matrix representation of it. We always assume that a particular decomposition-compatible basis of each injective sheaf is fixed. When we say ``$(I^\bullet,\eta^\bullet)$ is a complex represented by labeled matrices'', we mean that matrices denoted by $\eta^d$ form a complex of labeled matrices that represents the complex of injective sheaves $(I^\bullet, \eta^\bullet)$.
\end{remark}

\section{Derived Categories}\label{sec:derived_categories}

We often use different notions and degrees of homology separately when studying geometrical objects, losing important relationships between the different choices. Derived categories are a tool to take a step back while still abstracting to ``just homological information''. The formal definition is technically involved, but it can be understood in simpler terms for our setting of finite dimensional bounded cochain complexes over finite posets---in this section we define essentially the skeleton of the derived category. The objects are minimal injective complexes, and the morphisms are a quotient of morphisms of complexes.

We recall definitions necessary to describe the quotient of morphisms in \ref{sec:derived_categories:morphisms}. In \ref{sec:derived_categories:minimal_complexes} we show why we can restrict our attention to minimal injective complexes and give the definition of (a skeleton of) the derived category. In \ref{sec:derived_categories:algorithm_minimize} we give an explicit algorithmic way to go from any injective complex to the minimal one. Finally, in \ref{sec:derived_categories:injective_resolutions} we recall how to think of a single sheaf as a member in the derived category of minimal injective complexes, and describe a concrete algorithm to compute the minimal injective resolution of a sheaf. See also \ref{app:basis_for_morphisms} for explicit description of the vector space of morphisms in the derived category with a system of linear equations.

\subsection{Morphisms, Null-homotopy, Quasi-isomorphism} \label{sec:derived_categories:morphisms}

\begin{definition}\label{def:null_homotopic}
A morphism $\alpha^\bullet:F^\bullet\rightarrow G^\bullet$ between two complexes, $(F^\bullet,\eta^\bullet)$ and $(G^\bullet,\delta^\bullet)$, is \emph{null-homotopic} if there exists morphisms of sheaves $h^d:F^d\rightarrow G^{d-1}$:
    \begin{center}
    \includegraphics[width=150mm]{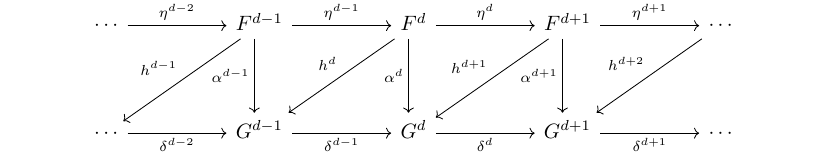}
\end{center}
 such that $\alpha^d=\delta^{d-1}h^d+h^{d+1}\eta^d$, for each $d$. 
\end{definition}

To get a better intuition about null-homotopic maps, we state an alternate definition. We need to first introduce the mapping cone of a morphism of complexes, which will also be useful in later sections.

\begin{definition}\label{def:mapping_cone_comlex}
    The \emph{mapping cone}, of a morphism $\alpha^\bullet:F^\bullet\rightarrow G^\bullet$ between two complexes, $(F^\bullet,\eta^\bullet)$ and $(G^\bullet,\delta^\bullet)$, is a complex $C(\alpha^\bullet) = (C^\bullet, \gamma^\bullet)$ given by \[
    C^d \coloneqq F^{d+1}\oplus G^{d},\hspace{20mm} \gamma^d = \begin{pmatrix} -\eta^{d+1} & 0 \\ \alpha^{d+1} & \delta^d \end{pmatrix}.
    \]
    
    \begin{minipage}{.95\textwidth}
    \centering
    \vspace{3mm}
    \includegraphics[width=150mm]{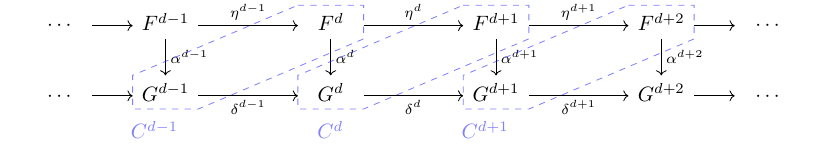}
    \vspace{3mm}
    \end{minipage}
\end{definition}
The following is a useful fact which we need in later sections.
\begin{lemma}[{\cite[Def.\ 1.5.7 and the following discussion]{KashiwaraSchapira1994}}]\label{lem:mapping_cone_of_quasi-isomorphism}
A morphism $\alpha^\bullet:F^\bullet\rightarrow G^\bullet$ between two complexes is a quasi-isomorphism if and only if its mapping cone, $C(\alpha^\bullet)$, is exact.
\end{lemma}

The mapping cone construction provides a condition equivalent to Definition~\ref{def:null_homotopic}. The proof is straightforward if we write $\tilde{\alpha}^\bullet$ in the statement as $(\alpha^\bullet, h^\bullet)$.
\begin{proposition}\label{prop:explicit_null_homotopic}
A morphism $\alpha^\bullet:F^\bullet\rightarrow G^\bullet$ between two complexes, $(F^\bullet,\eta^\bullet)$ and $(G^\bullet,\delta^\bullet)$,
is null-homotopic if and only if $\alpha^\bullet$ factors through the mapping cone of the identity map on the complex $(G^{\bullet-1},-\delta^{\bullet-1})$ (with indices shifted by $-1$):
\begin{center}
\begin{tikzcd}
    F^\bullet \ar[swap]{r}{\tilde{\alpha}^\bullet} \ar[bend left=30]{rr}{\alpha^\bullet} &
    C\left(\id_{(G^{\bullet-1},-\delta^{\bullet-1})}\right) \ar[swap]{r}{\mathrm{proj}} &
    G^\bullet
\end{tikzcd}
\end{center}

\end{proposition}

Note that, by Lemma \ref{lem:mapping_cone_of_quasi-isomorphism}, the mapping cone of an identity map is an exact complex. Hence, the above proposition immediately implies that all null-homotopic maps induce trivial morphisms on cohomology sheaves. This gives some intuition for the following definition, which, on a high level, states that null-homotopic maps should be treated as zero maps for the purposes of homological algebra. Recall the notion of cohomology sheaf from Definition~\ref{def:cohomology_sheaf}.

\begin{definition}
  A morphism $\alpha^\bullet:F^\bullet\rightarrow G^\bullet$ of complexes of sheaves is a \emph{quasi-isomorphism} if the induced map 
  \[
  H^d(\alpha^\bullet): H^d(F^\bullet)\rightarrow H^d(G^\bullet)
  \] 
  between cohomology sheaves is an isomorphism for each $n\in\mathbb{Z}$. 
\end{definition}

\begin{definition}
    An \emph{injective resolution} of a complex of sheaves $F^\bullet$ is a quasi-isomorphism 
    \[
    \alpha^\bullet:F^\bullet\rightarrow I^\bullet
    \]
    from $F^\bullet$ to a complex $I^\bullet$ of injective sheaves. 
\end{definition}

Each complex of sheaves can be identified with a complex of injective sheaves up to quasi-isomorphism---see, e.g., \cite[Thm~1.3.4]{Shepard1985}. In the next section we show that it can be identified with a minimal injective complex, unique up to isomorphism of complexes, and in the following section we give an explicit algorithm for minimizing an injective complex.

\subsection{Minimal complexes of Injective Sheaves} \label{sec:derived_categories:minimal_complexes}

We will now define minimal complexes of injective sheaves, and show that every complex of injective sheaves has a minimal complex of injective sheaves quasi-isomorphic to it, and this minimal complex is unique up to isomorphism of complexes.

\begin{definition}\label{def:minimal_injective_complex}
    A complex of injective sheaves $I^\bullet$ is \emph{minimal} if it minimizes the number of indecomposables among all quasi-isomorphic complexes. That is, if for every complex of injective sheaves $J^\bullet$ which is quasi-isomorphic to $I^\bullet$ and for every $d$, the number of indecomposable summands of $I^d$ is less than or equal to the number of indecomposable summands of $J^d$.
\end{definition}

To formulate one of the equivalent conditions of minimality, we define maximal vectors. In the module perspective, they are exactly the elements of the socle.

\begin{definition}\label{def:maximal-vectors}
    A vector $s\in F(\pi)$ is \emph{maximal} if $F(\pi\le\tau)(s)=0$ for all $\tau > \pi$. We denote by $M_F(\pi)$ the subspace of maximal vectors in $F(\pi)$, i.e., \[M_F(\pi)\coloneqq\bigcap_{\pi< \sigma}\ker F(\pi\leq\sigma).\]
    Note that it is sufficient to take only the intersection of $\ker F(\pi\le\sigma)$ for all $\pi<_1\sigma$. 
\end{definition}

\begin{theorem} \label{thm:minimal_injective_complex}
Let $(I^\bullet, \eta^\bullet)$ be a complex\footnote{Recall that the complexes considered in this paper are bounded.} of injective sheaves on $\Pi$. The following are equivalent: 
\begin{enumerate}
    \item $I^\bullet$ is minimal.
    \item For any quasi-isomorphic complex $(J^\bullet, \lambda^\bullet)$ of injective sheaves, there exists a morphism of complexes $r^\bullet:J^\bullet\rightarrow I^\bullet $ such that $r^d$ is surjective for each $d$. 
 \item For any quasi-isomorphic complex $(J^\bullet, \lambda^\bullet)$ of injective sheaves, there exists a morphism of complexes $q^\bullet:I^\bullet\rightarrow J^\bullet $ such that $q^d$ is injective for each $d$. 
 \item \label{thm:minimal_injective_complex_maximal_vectors} For each $d$, $\pi\in\Pi$, and maximal vector $s\in I^d(\pi)$, $\eta^d(\pi)(s)=0$.
\end{enumerate}
Moreover, for a given complex of sheaves there exists a unique (up to isomorphism of complexes) injective resolution by a minimal complex of injective sheaves.
\end{theorem}
\begin{proof}

The implications $2\Rightarrow 1$ and $3\Rightarrow 1$ are immediate as both imply that $\dim I^d(\pi) \leq \dim J^d(\pi)$ for every complex $J^\bullet$ quasi-isomorphic to $I^\bullet$, and every $d\in\Z$, $\pi\in\Pi$.

We show implication $1\Rightarrow 4$, proof of which also yields an argument for existence of minimal injective complexes, and then we show $4\Rightarrow 3$ and $4\Rightarrow 2$, which will also yield the uniqueness.

\underline{$1\Rightarrow 4$}: We show that if $4$ does not hold, then $1$ does not hold either. If $s$ is a maximal vector in $I^d(\pi)$ which does not vanish when we apply $\eta^d$, we can define $\hat{I}^d:= I^d/[\pi]_s$, where $[\pi]_s$ is the indecomposable injective such that $s\in[\pi]_s(\pi)$, and similarly $\hat{I}^{d+1}:= I^{d+1}/[\pi]_{\eta^d(\pi)(s)}$. By Lemma \ref{lem:coker-injective}, both are injective sheaves. Denoting by $q^d$ the quotient maps and setting $r=\eta^d(\pi)(s)$, we get:

\begin{center}
    \includegraphics[width=.65\textwidth]{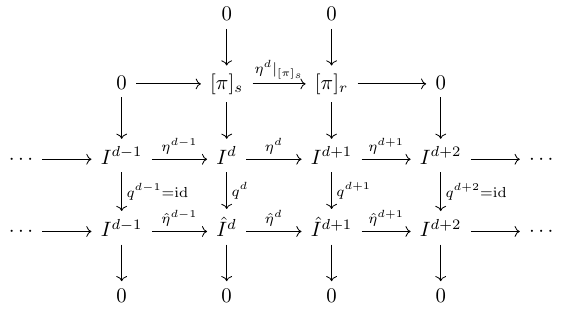}
\end{center}
Since the top row has trivial cohomology, and a short exact sequence of complexes induces a long exact sequence in cohomology, we see that \[0\rightarrow H^j(I^\bullet) \overset{q^j}{\rightarrow} H^j(\hat I^\bullet) \rightarrow 0\] is exact, that is, $q^\bullet$ induces an isomorphism on cohomology. We have found a complex $\hat I^\bullet$ quasi-isomorphic to $I^\bullet$ with fewer indecomposable summands in places $d$ and $d+1$, so $I^\bullet$ is not minimal.

By iterating this process, we can produce, from any bounded complex of injective sheaves $J^\bullet$, a minimal complex of injective sheaves $I^\bullet$, which proves the existence of minimal injective complexes.

\underline{$4\Rightarrow 2, 3$}: 
Since $I^\bullet$, $J^\bullet$ are quasi-isomorphic, there is a quasi-isomorphism $q^\bullet: I^\bullet\rightarrow J^\bullet$. Because both complexes consist of injective sheaves, $q^\bullet$ is a homotopy equivalence. Therefore, let $q^\bullet: I^\bullet \rightarrow J^\bullet$ and $r^\bullet: J^\bullet \rightarrow I^\bullet$ form a homotopy equivalence. We show that if $I^\bullet$ satisfies the condition~$4$, then $q^d$ is injective and $r^d$ is surjective for each $d$.

We inductively prove that $r^d \circ q^d$ is injective and that $r^d \circ q^d (\ker\eta^d) = \ker\eta^d$. As the base case, choose any $d$ such that $I^j=0$ for each $j\le d$, where the induction hypothesis holds trivially. We show the induction step.

By assumption, $r^d \circ q^d=\id + h^{d+1}\circ \eta^d + \eta^{d-1}\circ h^d$ for a collection of morphisms
\[
h^d:I^d\rightarrow I^{d-1}.\]
We begin by showing that the restriction of $r^d \circ q^d$ to $\ker\eta^d$ is injective. 
\begin{center}
    \includegraphics[width=.8\textwidth]{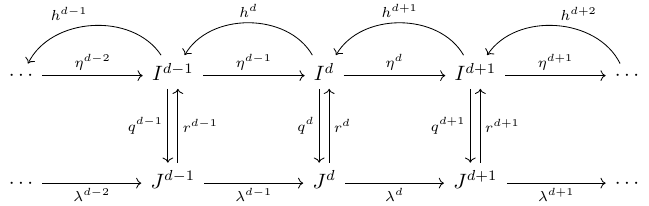}
\end{center}
Assume $\pi\in\Pi$, $x\in\ker\eta^d(\pi)$, and $x\in\ker (r^d\circ q^d)(\pi)$. Then 
\[0 = (r^d\circ q^d)(\pi) (x) = x + 0 + (\eta^{d-1} \circ h^d)(\pi) (x).\]
That is, $x=(\eta^{d-1}\circ h^d)(\pi)(-x)$. Set $y\coloneqq h^d(\pi)(-x)$. 
Because $q^\bullet$ is a morphism of complexes,  $\lambda^{d-1}\circ q^{d-1} = q^d\circ \eta^{d-1}$. Because $\eta^{d-1}(\pi)(y)=x$, we have that $(\lambda^{d-1}\circ q^{d-1})(\pi)(y) = q^d(\pi)(x)$. Because $(r^d\circ q^d)(\pi)(x)=0$, we have that
\[
(r^d\circ \lambda^{d-1}\circ q^{d-1})(\pi) (y)  =0.
\]
Because $r^\bullet$ is a morphism of complexes, $r^d\circ \lambda^{d-1}=  \eta^{d-1}\circ r^{d-1}$.  Therefore, 
\[
(\eta^{d-1}\circ r^{d-1}\circ q^{d-1})(\pi) (y)  =0.
\]
We have therefore shown that $ (r^{d-1}\circ q^{d-1})(\pi)(y)\in \ker \eta^{d-1}(\pi)$.

By the induction assumption, we have that \[(r^{d-1}\circ q^{d-1})(\ker\eta^{d-1})=\ker\eta^{d-1}.\]
Therefore, there exists $z \in \ker\eta^{d-1}(\pi)$ such that $(r^{d-1}\circ q^{d-1})(\pi)(z)=(r^{d-1}\circ q^{d-1})(\pi)(y)$. By the induction assumption that $ r^{d-1}\circ q^{d-1}$ is injective, we have that $z=y$. Therefore, $y\in\ker \eta^{d-1}(\pi)$, and $x=\eta^{d-1}(\pi) (y) = 0$. Altogether we showed that the restriction of $r^d\circ q^d$ to $\ker\eta^d$ is injective. 

For each non-zero vector $u\in I^d(\tau)$, there exists $(\tau\le\pi)\in\Pi$ and a non-zero maximal vector $s\in I^d(\pi)$ such that $I^d(\tau\le\pi)(u)=s$. By the condition~4 for $I^\bullet$, all maximal vectors in $I^d(\pi)$ are elements of $\ker\eta^d(\pi)$. Because the restriction of $r^d\circ q^d$ to $\ker\eta^d$ is injective, $(r^d\circ q^d)(\pi)(s)\neq 0$. By the commutativity assumed in the definition of a natural transformation (Definition \ref{def:natural-transformation}), \[
I^{d+1}(\tau\le\pi)\big((r^d\circ q^d)(\tau)(u)\big) = (r^d\circ q^d)(\pi)\big( I^{d}(\tau\le\pi)(u)\big) = (r^d\circ q^d)(\pi)(s)\neq 0.
\]
Therefore, $(r^d\circ q^d)(\tau)(u) \neq 0$. That is, $r^d \circ q^d$ is injective. Moreover, because $r^\bullet$ and $q^\bullet$ are morphisms of complexes, $(r^d \circ q^d) (\ker\eta^d) \subseteq \ker\eta^d$. Because the restriction of $r^d\circ q^d$ to $\ker\eta^d$ is injective, we have the desired equality: $(r^d \circ q^d) (\ker\eta^d) = \ker\eta^d$.

Since $r^d\circ q^d(\pi)$, for $\pi\in\Pi$, is an endomorphism of finite-dimensional vector spaces, it is injective iff it is surjective. Therefore, $q^d$ is injective and $r^d$ is surjective, as claimed.

Finally, if both complexes satisfy condition $4$, then by a symmetric argument, also $r^d$ is injective, and $q^d$ is surjective. In particular, $q^d$ is an isomorphism of sheaves, i.e., $q^\bullet$ is an isomorphism of complexes.

\end{proof}

\begin{definition}\label{def:derived-category}
    The \emph{bounded derived category of sheaves} on $\Pi$, $D^b(\Pi)$, consists of a set of objects 
    \[
    \Ob D^b(\Pi) = \{(I^\bullet,\eta^\bullet)\text{ a minimal complex of injective sheaves}\}
    \]
    and a vector space of morphisms between any two objects
    \[
    \Hom_{D^b(\Pi)}(I^\bullet, J^\bullet)
    \]
    which is defined to be the quotient space of the (finite-dimensional) vector space of all morphisms (of complexes) from $I^\bullet$ to $J^\bullet$ modulo the subspace of null-homotopic morphisms. 
\end{definition}

\begin{remark}
It follows from Theorem \ref{thm:minimal_injective_complex} that the above definition of the bounded derived category of sheaves is in fact a skeleton of the category which is traditionally referred to as the bounded derived category of sheaves~\cite{KashiwaraSchapira1994,Bredon1997}---see \ref{app:section:correspondence} for more details. We choose this definition in order to streamline and simplify the computational perspective of derived categories considered below. Any concepts we define for the elements of $D^b(\Pi)$, in the sense above, apply for any (bounded) complex of sheaves on $\Pi$. An element $I^\bullet\in D^b(\Pi)$ is a chosen representative of a class of quasi-isomorphic complexes---the unique minimal injective member of it.
\end{remark}

\subsection{Recognizing Minimality and Minimizing an Injective Complex}\label{sec:derived_categories:algorithm_minimize}

With labeled matrix representations we can easily recognise whether a complex of injective sheaves is minimal, and obtain a minimal complex in case it is not. 

\begin{proposition}\label{prop:recognising_minimality_labeled_matrices}
    A complex of injective sheaves $(I^\bullet, \eta^\bullet)$ represented by labeled matrices is minimal iff $\eta^d[\pi, \pi]$ is a zero (or empty) matrix for each $\pi\in\Pi$ and degree $d$.
\end{proposition}
\begin{proof}
    The claim is a reformulation of the condition~(\ref{thm:minimal_injective_complex_maximal_vectors}) in Theorem~\ref{thm:minimal_injective_complex}: the complex is minimal iff all maximal vectors are sent to zero by the complex maps $\eta^d$. If we write $I^d = \bigoplus_{\pi\in\Pi} [\pi]^{p_\pi}$, then the maximal vectors in $I^d(\pi)$ are exactly vectors in $[\pi]^{p_\pi}(\pi)$ (as a subspace of $I^d[\pi]$).
    In other words, written with respect to the bases fixed for the complex of labeled matrices, a vector $s\in I^d(\pi)$ is maximal iff it is zero at the coordinates labeled by $\St\pi \setminus\{\pi\}$. Then $\eta^d(\pi)(s) = \eta^d[\St\pi, \St\pi]\cdot s$ is zero for all maximal $s\in I^d(\pi)$ iff $\eta^d[\pi,\pi]=0$.
\end{proof}

\paragraph{The peeling procedure} If $(I^\bullet, \eta^\bullet)$ is not minimal, we perform allowed row and column operations so that the non-zero entries in $\eta^d[\pi, \pi]$ are on the main diagonal, and, for each, the corresponding row and column in the whole labeled matrix $\eta^d$ only contain that one single non-zero element. This row and column then represents the sequence $0 \rightarrow [\pi]_s \rightarrow [\pi]_{\eta^d(\pi)(s)} \rightarrow 0$ from the proof of Theorem~\ref{thm:minimal_injective_complex}, and can be safely removed from the complex---unlike in the proof of the theorem, here we see that the sequence actually splits off the complex as a direct summand. We repeat this process until $\eta^d[\pi, \pi]=0$ for each $\pi\in\Pi$ and degree $d$, which then verifies minimality by Proposition~\ref{prop:recognising_minimality_labeled_matrices}. The procedure is described in Algorithm~\ref{algo:peeling}.

\begin{algorithm}[ht]
    \caption{Computing minimal complex: Peeling}\label{algo:peeling}
    \textbf{Input:} $\eta^0,\dots,\eta^n$ labeled matrices representing an injective complex  $(I^\bullet,\eta^\bullet)$ on $\Pi$ \\
    \textbf{Output:} $\hat\eta^0,\dots,\hat\eta^n$ labeled matrices representing the minimal injective complex quasi-isomorphic to $(I^\bullet,\eta^\bullet)$ \\
    \begin{algorithmic}[1]
        \While{$\exists d\in\{0,\dots,n\} \ \exists \pi\in \Pi$ such that $\eta^d[\pi,\pi]$ is a non-zero matrix}
            \State find labeled matrices $L,R$ such that:
                \begin{itemize}
                    \item $(L\!\cdot\! \eta^d\!\cdot\! R)[\pi,\pi]$ is a diagonal matrix with entries in $\{0, 1\}$
                    \item every $\pi$-labeled column of $L\!\cdot\! \eta^d\!\cdot\! R$ is either zero on $\pi$ or on $\Pi\setminus\pi$
                    \item every $\pi$-labeled row of $L\!\cdot\! \eta^d\!\cdot\! R$  is either zero on $\pi$ or on $\Pi\setminus\pi$
                \end{itemize}
                \Comment{extended Gaussian elimination on $\eta^d[\pi,\pi]$ and further reduction on $\eta^d$}
            \State $\eta^d\gets L \cdot\!\eta^{d}\cdot\! R$,\hspace{3mm} $\eta^{d-1}\gets R^{-1}\cdot\eta^{d-1}$,\hspace{3mm} $\eta^{d+1}\gets \eta^{d+1} \cdot L^{-1}$
            \Comment{Lemma~\ref{lem:change_of_basis_for_complex}}
            
            \While{$\exists$ $i$-th row, $j$-th column in $\eta^d$ labeled by $\pi$ such that $\eta^d[i,j]=1$}
                \State delete the $i$-th row and the $j$-th column from $\eta^d$
                \State delete the $j$-th row from $\eta^{d-1}$
                \State delete the $i$-th column from $\eta^{d+1}$
            \EndWhile
        \EndWhile
        \State \Return $\eta^0,\dots,\eta^n$
    \end{algorithmic}
\end{algorithm}

\begin{proposition}\label{prop:correctness_peeling}
    Given an injective complex $(I^\bullet, \eta^\bullet)$ represented by labeled matrices, Algorithm~\ref{algo:peeling} outputs a labeled matrix representation of a quasi-isomorphic minimal injective complex.
\end{proposition}
\begin{proof}
    We first need to argue that it is possible to find the desired labeled matrices $L$, $R$ whenever $\eta^d[\pi,\pi]\neq 0$. All elementary operations are allowed for rows and columns labeled by the same element $\pi$. Therefore, following standard Gaussian elimination and further reduction, we reduce $\eta^d[\pi,\pi]$ so that it is a diagonal matrix with entries in $\{0,1\}$. Then we use the $1$ on the diagonal to remove any non-zeros in the corresponding row and column in the whole matrix $\eta^d$---indeed, there can be a non-zero exactly when removing it by subtracting our fixed column or row, respectively, is allowed (see Definition~\ref{def:allowed-operations}). The operations are performed as a change of basis of the whole complex, as described in Lemma~\ref{lem:change_of_basis_for_complex}. The newly produced complex is (quasi-)isomorphic to the original.

    For the new matrices $\eta^\bullet$, we argue that the rows and columns with non-zeros in $\eta^d[\pi,\pi]$ can be peeled off the complex. Let $(i,j)$ be an index of $\eta^d$ with entry $1$ such that both the $i$-th row and the $j$-th column are indexed by $\pi$. By the above construction, both have $0$ in all entries other than their intersection $(i,j)$. Then the complex condition $\eta^d\!\cdot\!\eta^{d-1}=0=\eta^{d+1}\!\cdot\!\eta^d$ implies that the corresponding $j$-th row in $\eta^{d-1}$ and $i$-th column in $\eta^{d+1}$ are both zero. Therefore, we have identified an exact sequence $0\rightarrow [\pi] \rightarrow [\pi] \rightarrow 0$ that splits from the complex, and removing it produces a quasi-isomorphic complex, as argued in the proof of Theorem~\ref{thm:minimal_injective_complex}. The deletions from the labeled matrices described in the algorithm correspond to this removal.

    Finally, the algorithm clearly terminates, because each iteration of the main loop removes some non-zero number of the finitely many rows and columns of the complex of labeled matrices. The outputted complex of labeled matrices is minimal, since it satisfies the condition of Proposition~\ref{prop:recognising_minimality_labeled_matrices}.
\end{proof}

\begin{example}\label{ex:peeling}
Let $\Pi$ be such that $\sigma < \pi < \mu < \lambda$, and consider an injective complex \[
    0 \rightarrow [\lambda] \rightarrow [\mu] \oplus [\pi]^2 \rightarrow [\pi]^3 \oplus [\sigma] \rightarrow [\sigma] \rightarrow 0
\]
given by labeled matrices $\eta^d$. It is not minimal, because $\eta^1[\pi,\pi]\neq 0$. We minimize it using Algorithm~\ref{algo:peeling}. First we reduce the highlighted block $\eta^1[\pi,\pi]$ to a matrix with $1$ only on the main diagonal and $0$ everywhere else. Row operations suffice, and for each we perform the inverse column operation on $\eta^2$. We get matrices $\zeta^d$. The next step is to clear out the rest of the rows and columns that contain a $1$ within the $\zeta^1[\pi, \pi]$ block. We use the diagonal elements to clear out the highlighted non-zeros with two row and two column allowed operations. Finally, the resulting matrices $\delta^d$ contain two $0\rightarrow [\pi] \rightarrow [\pi] \rightarrow 0$ sequences that split off. We can delete all the highlighted rows and columns, obtaining matrices $\gamma^d$ representing a quasi-isomorphic complex \[
    0 \rightarrow [\lambda] \rightarrow [\mu] \rightarrow [\pi] \oplus [\sigma] \rightarrow [\sigma] \rightarrow 0,
\]
which is already minimal, because all $\delta^d[x,x]$ blocks are zero.

\noindent\begin{minipage}{\textwidth}
    \centering
    \includegraphics[width=160mm]{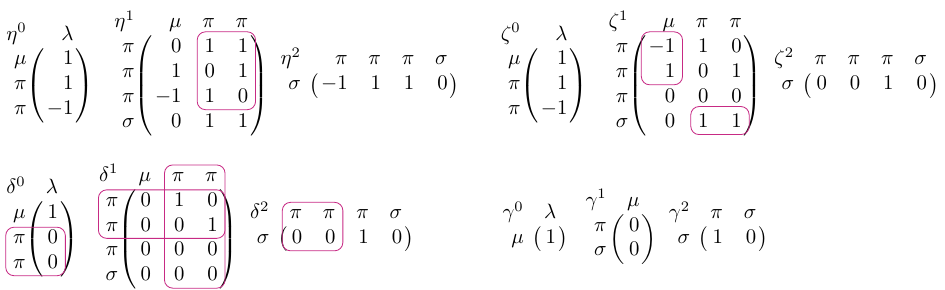}
\end{minipage}
\end{example}

\medskip
As noted above the proposition, this explicit peeling construction also serves as a proof that the minimal injective resolution splits off any injective complex:
\begin{corollary}\label{cor:minimal-exact_decomposition}
    Let $J^\bullet$ be any complex of injective sheaves, and $I^\bullet$ its minimal injective resolution. Then there exists an exact complex of injective sheaves $E^\bullet$, and an isomorphism of complexes
    $
        J^\bullet\cong I^\bullet\oplus E^\bullet. 
    $
    In other words, the short exact sequence 
    $
        0\rightarrow I^\bullet\xrightarrow{q^\bullet} J^\bullet\rightarrow E^\bullet\rightarrow 0
    $
    splits (in the category of complexes of sheaves). 
\end{corollary}

Apart from the peeling, we also provide an alternative computation to minimize a complex in Section~\ref{sec:pullback_algo}---our computation of $\Rfunc\,\id^*I^\bullet$ (defined there) yields a minimal injective complex quasi-isomorphic to $I^\bullet$. This can be convenient for implementation, as it requires no column operations.

\subsection{Single sheaf in Derived Category: Computing Injective Resolutions}\label{sec:derived_categories:injective_resolutions}

A single sheaf $F$ can be viewed as a complex $\cdots \rightarrow 0 \rightarrow F \rightarrow 0 \rightarrow \cdots$ with $F$ in degree~0. As for any complex, there exists a quasi-isomorphic complex of injective sheaves, $I^\bullet$. Replacing $F$ by $I^\bullet$ allows us to study it within the framework of derived categories. The mapping cone (Definition~\ref{def:mapping_cone_comlex}) of any such quasi-isomorphism is called \emph{injective resolution} of $F$---it is an exact complex by Lemma~\ref{lem:mapping_cone_of_quasi-isomorphism}.

\begin{definition}
\label{def:injective_resolution}
An \emph{injective resolution} of a sheaf $F$
is an exact sequence 
\[
    0\rightarrow F\xrightarrow{\alpha} I^{0}\xrightarrow{\eta^0} I^{1}\xrightarrow{\eta^{1}} I^{2} \xrightarrow{\eta^{2}} \cdots
\]
where $I^j$ is an injective sheaf for each $j$. We denote by $I^\bullet$ the complex 
\[
    \cdots\rightarrow 0\rightarrow I^0\xrightarrow{\eta^0}I^1\xrightarrow{\eta^1}\cdots.
\]
Slightly abusing the naming, we later also say that $I^\bullet$ is an injective resolution of $F$.
\end{definition}

\begin{definition}\label{def:injective_hull}
    The inclusion $F\hookrightarrow I^0$ into an injective sheaf is called \emph{injective hull} of $F$.
\end{definition}
Any such inclusion can be extended to an injective resolution, and we give an explicit algorithm to compute such extension below.

By Theorem~\ref{thm:minimal_injective_complex}, there exists a unique minimal injective resolution. It is straightforward to bound its length in terms of the height of the poset $\Pi$. First, note that the condition (\ref{thm:minimal_injective_complex_maximal_vectors}) of the theorem can be reformulated in terms of images, since injective resolutions are exact.

\begin{proposition}\label{lem:minimal_injective_resolution_im_condition}
    An injective resolution $0\rightarrow F\xrightarrow{\alpha} I^{0}\xrightarrow{\eta^0} I^{1}\xrightarrow{\eta^{1}} I^{2} \xrightarrow{\eta^{2}} \dots$ is minimal iff the following holds:
    \begin{itemize}
        \item Each maximal vector $s\in I^d(\pi)$ is in $\im \eta^{d-1}(\pi)$, for each $d>0$ and $\pi\in\Pi$. Each maximal vector $s\in I^0(\pi)$ is in $\im \alpha(\pi)$ for each $\pi\in\Pi$.
    \end{itemize}
\end{proposition}

\begin{corollary}\label{cor:length_of_resolution}
    The minimal injective resolution of a sheaf $F$ on a finite poset $\Pi$ of height~$h$ consists of at most $h+1$ non-zero injective sheaves.
\end{corollary}
\begin{proof}
    The length of the longest chain of non-zero vector spaces in $I^d$ is strictly decreasing in $d$ in the minimal injective resolution. This is implied by Proposition~\ref{lem:minimal_injective_resolution_im_condition} and Theorem~\ref{thm:minimal_injective_complex}~(\ref{thm:minimal_injective_complex_maximal_vectors}): If $I^d(\pi)=0$, then the former implies that there are no maximal vectors in $I^{d+1}(\pi)$. Therefore, if $I^d(\tau)=0$ for all $\tau\geq\pi$, then also $I^{d+1}(\tau)=0$ for all $\tau\geq\pi$. Moreover, if $I^d(\pi)\neq 0$ and $I^d(\tau)=0$ for all $\tau>\pi$, then all vectors in $I^d(\pi)$ are maximal, and by property (\ref{thm:minimal_injective_complex_maximal_vectors}) and the argument above, $I^{j}(\pi)=0$ for all $j>d$.
\end{proof}

We now turn our attention to the computation of the minimal injective resolution. For a generic sheaf, $F$, one needs to first compute an injective hull, and then continue the sequence to ensure exactness. There are well known algorithms for both steps found in literature on module theory and representation theory. We present our own version of it for two reasons. Firstly, we formulate the algorithms purely in the language of linear algebra so that they can be understood without prior knowledge of module or sheaf theory. Second, we generalize the main step of the construction to novel settings in Section~\ref{sec:derived-functors}.

We discuss the construction of an injective hull for a generic sheaf in \ref{app:injective_hull}. We want to point out that it is feasible to construct it using only basic linear algebra, but one of the strengths of our approach is that we can study interesting geometrical objects without ever explicitly working with a generic non-injective sheaf.

\paragraph*{Minimal injective resolution from injective hull}
We describe how to inductively compute the minimal injective resolution, given the minimal injective hull. That is, given the prefix of the resolution represented by labeled matrices $\eta^0, \dots, \eta^{d-1}$, the goal is to construct its continuation $\eta^d$. We need to construct $\eta^d$ such that \[
    I^{d-1} \xrightarrow{\eta^{d-1}} I^d \xrightarrow{\eta^d} I^{d+1}
\]
is exact at $I^d$, and moreover this extension of the complex is minimal in the sense that all maximal vectors in $I^{d+1}$ are in the image of $\eta^d$. We construct $\eta^d$ again by induction, this time over $\pi\in\Pi$ in non-increasing order.

The procedure will only need access to images of $\eta^{d-1}(\pi)$ for $\pi\in\Pi$. In particular, it can also be used to construct $\eta^0$ if images of $\alpha(\pi)$ are given---this can be done by running Algorithm~\ref{algo:injective_hull_alpha} in the case of a generic sheaf (\ref{app:injective_hull}), but it can be done more efficiently in the case of the constant sheaf, as we describe later in Section~\ref{sec:derived_categories:constant_sheaf}.

Recall that we denote labeled matrices by the same symbols as the maps they represent; the linear map $\eta^d(\pi)$ is described by the submatrix $\eta^d(\pi)=\eta^d[\St\pi,\St\pi]$

\paragraph*{Procedure \makeexact} The construction consists of repeated call of a core procedure \makeexact, described as Algorithm~\ref{algo:makeeact}. On the input, it takes $\eta^{d-1}$, $\eta^d$ and $\pi\in\Pi$, where $\eta^{d-1}$, $\eta^d$ are two labeled matrices representing a complex\[
    I^{d-1} \xrightarrow{\eta^{d-1}} I^d \xrightarrow{\eta^d} I^{d+1},
\]
On the output, it returns a labeled matrix representing a map $\hat\eta^d: I^d\rightarrow \hat I^{d+1}$ such that $\im{\eta^{d-1}(\pi)} = \ker{\hat\eta^{d}(\pi)}$, and, importantly, $\hat\eta^{d}(\tau) = \eta^{d}(\tau)$ for all $\tau\not\leq\pi$.

\begin{algorithm}[htb]
    \caption{\makeexact$(\eta^{d-1},\eta^d,\pi)$}
    \label{algo:makeeact}
    \hspace*{\algorithmicindent} \textbf{Input:} $\eta^{d-1}, \eta^d$ labeled matrices over a poset $\Pi$, s.t. $\eta^d\cdot\eta^{d-1}=0$, and $\pi\in\Pi$\\
    \hspace*{\algorithmicindent} \textbf{Output:} $\hat\eta^{d}$ a labeled matrix
\begin{algorithmic}[1]
        \State $B \gets$ basis of $\big(\im(\eta^{d-1}(\pi))\big)^{\perp}$ 
        \ForAll{$b\in B$}
            \If{$b$ is linearly independent from the rows of $\eta^d(\pi)$}
                \State add $b$ as a row of $\eta^{d}$, labeled by $\pi$
                \Comment{with zeros at $\Pi\setminus\St\pi$ entries}
            \EndIf
        \EndFor
        \State \Return $\eta^d$
\end{algorithmic}
\end{algorithm}

The goal is to make $I^\bullet(\pi)$ exact in degree $d$. We achieve this by adding new linearly independent rows to $\eta^d(\pi)$ from the space
\[
\big(\im\eta^{d-1}(\pi)\big)^{\perp}=\left\{v\in k^{\dim I^{d}(\pi)}\,\middle|\, \forall u\in\im\eta^{d-1}(\pi): v\cdot u=0\right\},
\]
until we cannot add any more. To insert the new rows in the full labeled matrix $\eta^d$, we fill them by zeros in all columns not labeled by $\St\pi$. Each new row is labeled by $\pi$, and represents a new copy of $[\pi]$ in $\hat I^{d+1}$.

\begin{lemma}\label{lem:makeexact}
    Let $\eta^{d-1}$, $\eta^d$ and $\pi\in\Pi$ are as above, and $\hat\eta^d\coloneqq$ \makeexact$(\eta^{d-1},\eta^d,\pi)$. Then
    \begin{enumerate}[(1)]
        \item\label{itm:makeexact_change} $\hat\eta^d$ differs from $\eta^d$ only by having extra rows labeled by $\pi$,
        \item\label{itm:makeexact_exact} $I^{d-1}(\pi) \xrightarrow{\eta^{d-1}(\pi)} I^{d}(\pi) \xrightarrow{\hat\eta^{d}(\pi)} \hat I^{d+1}(\pi)$ is exact at $I^d(\pi)$,
        \item\label{itm:makeexact_minimal} if $e_i,\dots,e_l\in\hat I^{d+1}(\pi)$ are the canonical vectors corresponding to the rows newly added to $\hat\eta^d(\pi)$, then $\Span(e_i,\dots,e_l)\subseteq\im\hat\eta^d(\pi)$.
    \end{enumerate}
\end{lemma}

\begin{proof}
    The point~\ref{itm:makeexact_change} is clear from the construction; we prove the other two. For~\ref{itm:makeexact_exact} we need to show that $\im \eta^{d-1}(\pi)=\ker \hat\eta^d(\pi)$. By assumption, $\im\eta^{d-1}(\pi)\subseteq\ker\eta^d(\pi)$. Observe that
        \begin{alignat*}{2}
            &\im\eta^{d-1}(\pi) &&= \ker\hat\eta^d(\pi) \\
            \Longleftrightarrow \ \ 
            &\ker\eta^{d-1}(\pi)^T &&= \im\hat\eta^d(\pi)^T \\
            \Longleftrightarrow \ \ 
            &\im\eta^{d-1}(\pi)^\perp &&= \mathrm{span}(\text{rows of }\hat\eta^d(\pi)),
        \end{alignat*}
    where the first equivalence comes from linear duality, and the second is just rewriting both sides of the equality. The same holds for inclusions, with the inclusion direction flipping at the first equivalence---hence, we start with $\im\eta^{d-1}(\pi)^\perp \supseteq \mathrm{span}(\text{rows of }\eta^d(\pi))$. In the algorithm, we add vectors from the left side as new rows in $\hat\eta^d(\pi)$, until the equality is achieved, securing the exactness.
    
    To show~\ref{itm:makeexact_minimal}, assume that rows $i,\dots,l$ were added to $\hat\eta^d(\pi)$, and let $e_j=(\dots,0,1,0,\dots)$ denote the $j$-th canonical vector. We show that $\Span(e_i,\dots,e_l)\subseteq \im\hat\eta^d(\pi)$. Let $A_j$ be the matrix consisting of the first $j$ rows of $\hat\eta^d(\pi)$. As we only add new rows when they are linearly independent from the previous, we have $\ker A_j\subsetneq \ker A_{j-1}$ for every $j\in\{i,\dots,l\}$. That is, there exists $u_j$ such that $A_{j-1} u_j = 0$, and $A_j u_j \neq 0$. This is only possible if $A_j u_j = \lambda e_j$ for some $\lambda\neq 0$, which means that $\eta^d(\pi) \cdot u_j = A_l u_j$ is a vector with zeros at positions $1,\dots,j-1$, and a non-zero at the $j$-th position. Therefore, $\Span(e_i,\dots,e_l) =
            \Span\left(\hat\eta^d(\pi)\cdot u_i, \dots,\hat\eta^d(\pi)\cdot u_l\right) \subseteq
            \im\hat\eta^d(\pi)$.
\end{proof}

\paragraph*{Extending the minimal injective resolution from $\eta^{d-1}$ to $\eta^{d}$}
The construction of $I^{d+1}$ and $\eta^d$ is described by Algorithm~\ref{algo:injective_resolution_tail}. We start with an empty matrix $\eta^d$ and repeatedly call \makeexact\  for all $\pi\in\Pi$ in a non-increasing order. In the next subsection, Example~\ref{ex:4-simplex_with_extra_edges} goes through the algorithm in detail on a concrete example.

\begin{algorithm}[ht]
    \caption{A step in the minimal injective resolution}
    \label{algo:injective_resolution_tail}
    \hspace*{\algorithmicindent} \textbf{Input:} $\eta^{d-1}$ a labeled matrix over a poset $\Pi$\\
    \hspace*{\algorithmicindent} \textbf{Output:} $\eta^{d}$ a labeled matrix
\begin{algorithmic}[1]
    \State $\eta^d \gets$ a labeled matrix with no rows and with columns labeled as rows of $\eta^{d-1}$
    \For{$\pi \in \Pi$ in a non-increasing order}
        \State $\eta^d \gets$ \makeexact$(\eta^{d-1},\eta^d,\pi)$
        \Comment{Algorithm~\ref{algo:makeeact}}
    \EndFor
    \State \Return $\eta^d$
\end{algorithmic}
\end{algorithm}

\begin{lemma}\label{lem:correctness_injective_resolution_sheaf}
    Starting with the minimal injective hull of $F$, iterative application of Algorithm~\ref{algo:injective_resolution_tail} yields the minimal injective resolution of $F$.
\end{lemma}
\begin{proof}
Each sheaf is injective by definition. By \ref{itm:makeexact_exact} from Lemma~\ref{lem:makeexact}, we get exactness of $I^\bullet(\pi)$ in degree $d$ right after the call of the procedure \makeexact$(\eta^{d-1},\eta^d,\pi)$. Since we go through the degrees in the increasing order, and through the poset elements in a non-increasing order, $\eta^{d-1}(\pi)$ and $\eta^d(\pi)$ both stay the same from this point on, and hence $I^\bullet(\pi)$ is still exact in degree $d$ at the end of the algorithm.

To show the minimality, we verify the condition on maximal vectors from Proposition~\ref{lem:minimal_injective_resolution_im_condition}. Maximal vectors in $I^{d+1}$ over $\pi$, $M_{I^{d+1}}(\pi)$, are exactly the new vectors added at step $\pi$. That is, if rows $i,\dots,l$ were added to $\eta^d(\pi)$ at step $\pi$, $M_{I^{d+1}}(\pi)=\Span(e_i,\dots,e_l)$. By \ref{itm:makeexact_minimal} from Lemma~\ref{lem:makeexact}, this space is contained in $\im\eta^d(\pi)$.
\end{proof}

\begin{remark}[Computing the orthogonal complement]
In Algorithm~\ref{algo:makeeact}, we purposefully leave out any particular way how to compute the basis of $(\im(\eta^{d-1}(\pi)))^{\perp}$. One way to compute it is via a standard row reduction algorithm: we start with $U\gets \text{identity matrix}$, $R\gets \eta^{d-1}(\pi)$, and we reduce rows from top to bottom, reducing each by adding the rows above it to push the left-most non-zero as much to the right as possible. Every row operation performed on $R$ is also performed on $U$, so that $R=U\cdot\eta^{d-1}(\pi)$. We end up with a lower-triangular matrix $U$ such that all its rows corresponding to the zero rows of $R$ form a basis of $(\im(\eta^{d-1}(\pi)))^{\perp}$.

An immediate advantage of this approach is that we only ever work with rows of $\eta^{d-1}(\pi)$. This means we can represent the matrices $\eta^d$ in a row-wise sparse representation, e.g., a list of rows, each represented as a ``\textit{column index} $\rightarrow$ \textit{value}'' dictionary. In this representation, when we go from $\eta^d$ to $\eta^d(\pi)$, we just choose all the rows labeled by $\St\pi$---there is no need to crop the rows themselves, as all entries not labeled by $\St\pi$ are $0$.
\end{remark}



\subsection{Injective Resolution of the Constant Sheaf}\label{sec:derived_categories:constant_sheaf}

We describe how to use Algorithm~\ref{algo:injective_resolution_tail} to compute the minimal injective resolution of a constant sheaf, give the geometrical meaning of multiplicities of indecomposables in the minimal resolution, showcase the introduced notions with and example, and finally compute the complexity of computing the minimal injective resolution of the constant sheaf.

\medskip
For the constant sheaf, the construction of the minimal injective hull $k_{\Pi}\xrightarrow{\alpha} I^0$ is very straightforward: $I^0=\bigoplus_{j=1}^n[\pi_j]$ where $\pi_1,\dots,\pi_n$ are the maximal elements of $\Pi$, each exactly once. The map $\alpha$ is given by the diagonal embedding
\begin{align*}
    \alpha(\sigma): k_{\Pi}(\sigma) &\longrightarrow I^0(\sigma) \\
    1 &\mapsto (1,\dots, 1)^T.
\end{align*}
Conveniently, we can represent this particular injection $\alpha$ as a labeled matrix: we define $\eta^{-1}$ as a column of ones with rows labeled by the maximal elements of $\Pi$. We label the column by a new ``virtual'' element greater than all elements of $\Pi$. To obtain the minimal injective resolution of $k_\Pi$, we now iteratively run Algorithm~\ref{algo:injective_resolution_tail} starting with the input $\eta^{-1}$.

Going one step further, we can even find the maximal elements of $\Pi$ as an iteration of Algorithm~\ref{algo:injective_resolution_tail}. Indeed, we start with an (empty) matrix $\eta^{-2}$ with no columns and one row labeled by the ``virtual'' element greater than all elements of $\Pi$. Running Algorithm~\ref{algo:injective_resolution_tail} on $\eta^{-2}$ then yields $\eta^{-1}$.

To summarise, we can compute the minimal injective resolution of the constant sheaf by iteratively running Algorithm~\ref{algo:injective_resolution_tail} starting with $\eta^{-2}$, and then keep matrices $\eta^d$ only for $d\geq 0$. Note that this is not just a lucky coincidence. This trick is a special case of computing a pullback, which we introduce later in Section~\ref{sec:pullback_algo}. For now, even though we have not defined a pullback yet, we only remark that the constant sheaf $k_{\Pi}$ is the pullback $f^\ast k_{\mathrm{pt}}$ for the constant map $f:\Pi\rightarrow\mathrm{pt}$.

\paragraph{Multiplicity of the Indecomposable Injective Sheaves}

The multiplicity of an indecomposable injective sheaf in the unique minimal injective resolution is a well-defined invariant of $F$. It is natural to ask what topological information is captured with these multiplicities. Below we answer this question for the constant sheaf on a finite simplicial complex. 
\begin{definition}\label{def:multiplicity_of_generators}
  Let $I^\bullet$ be the minimal injective resolution of $F$. By $m^d_F(\sigma)$ we denote the multiplicity of $[\sigma]$ in $I^d$:
\[
I^d \cong \bigoplus_{\sigma\in\Pi}[\sigma]^{m^d_F(\sigma)}.
\]
Equivalently, we can define 
$m^d_F(\sigma):=\dim M_{I^d}(\sigma)$, 
where $M_{I^d}(\sigma)$ is as in Definition \ref{def:maximal-vectors}.
\end{definition}
\begin{theorem}\label{thm:multiplicities}
Let $\Sigma$ be a finite simplicial complex, $k_\Sigma$ the constant sheaf on $\Sigma$ (viewed as a poset with the face relation), and $H_c^\bullet (|\St\sigma|;k)$ be the singular cohomology with compact support of the geometric realization of $\St\sigma$ (see also Definition~\ref{def:cohomology_functorial}). Then
\[
m^d_{k_\Sigma}(\sigma)= \dim H_c^{d+\dim\sigma}(|\St\sigma|;k).
\]
\end{theorem}
See Appendix \ref{app:thm:multiplicities} for a proof of the above theorem.

\paragraph{Examples}

We demonstrate how Algorithm~\ref{algo:injective_resolution_tail} works for constant sheaves with two examples.

\begin{example}\label{ex:4-simplex_with_extra_edges}
Consider the $3$-skeleton of the $4$-simplex, with two extra edges attached to vertex~$1$. We compute the minimal injective resolution of $\Pi\coloneqq\St(1)$. We describe simplices as lists of vertices, and for brevity omit the vertex $1$---e.g., $234=\{1,2,3,4\}$. See Figure~\ref{fig:4-simplex_with_extra_edges_poset}.

\begin{figure}[htb]
    \centering
    \includegraphics[width=140mm]{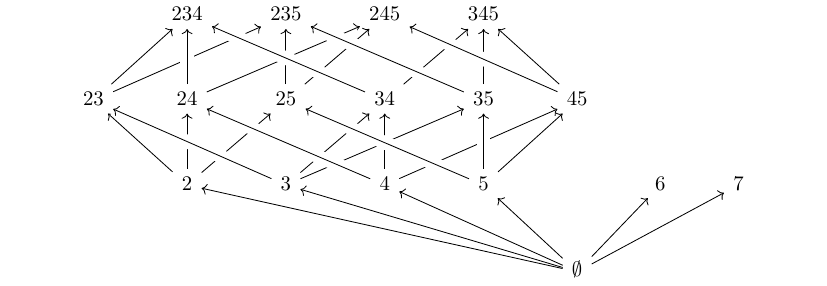}
    \caption{The poset considered in Example~\ref{ex:4-simplex_with_extra_edges}. We omit vertex $1$ from the labels.}
    \label{fig:4-simplex_with_extra_edges_poset}
\end{figure}

\begin{figure}[htb]
    \centering
    \includegraphics[width=160mm]{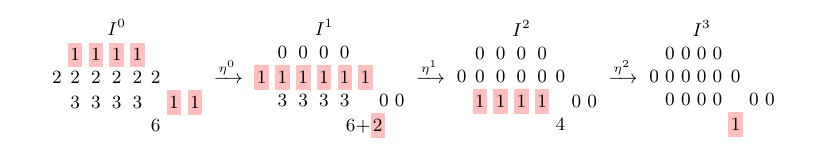}
    \caption{The resolution in Example~\ref{ex:4-simplex_with_extra_edges}. The numbers indicate the dimension at each simplex, positioned as in Figure~\ref{fig:4-simplex_with_extra_edges_poset}. The pink background indicates the generators. For example, in $I^1(\emptyset)$ we have 6 dimensions coming from generators above $\emptyset$, and we have 2 more generators at $\emptyset$.}
    \label{fig:4-simplex_with_extra_edges_resolution}
\end{figure}

The generators of $I^0$ are the maximal simplices $(234, 235, 245, 345, 6, 7)$, and $\eta^{-1}: k_{\Pi}\rightarrow I^0$ is the diagonal embedding for each $\pi\in\Pi$. We construct $I^1$ and $\eta^0$ as in Algorithm~\ref{algo:injective_resolution_tail}, with inputs $I^0,\eta^{-1}$. Initialize $I^1$ and $\eta^0$ empty, and go through the simplices row-by-row left-to-right as they are in Figure~\ref{fig:4-simplex_with_extra_edges_poset}. Starting with $234$, the space $I^0(234)$ is 1-dimensional and equal to $\im\eta^{-1}(234)$, so there is nothing to be added, and $\eta^0(234)=0$. The same happens for all the maximal simplices.

At triangle $23$, we have $I^0(23)=k^2$, since two generators are above $23$. At the moment, $\eta^0(23)$ is empty, so its kernel is $k^2$. We need $\ker\eta^0(23)=\im\eta^{-1}(23)=\Span\{(1,1)\}$. The orthogonal complement of $\im\eta^{-1}(23)$ is generated by the vector $(1,-1)$. We add it as a new row in $\eta^0(23)$. Therefore, we add $23$ to $I^1$, and add a first row to $\eta^0$; see Figure~\ref{fig:4-simplex_with_extra_edges_matrix}. Similarly, we add one row for each other triangle.

Now for the edges. We have $I^0(2)=k^3$, and $\eta^0(2)$ a $3\times 3$ matrix, highlighted as a green solid rectangle in Figure~\ref{fig:4-simplex_with_extra_edges_matrix}. 
We already have $\ker\eta^0(2)=\Span\{(1,1,1)\}=\im\eta^{-1}(2)$, so we do not add any new generators over $2$. The same goes for $\eta^0(3), \eta^0(4), \eta^0(5)$, each of which you can see highlighted in Figure~\ref{fig:4-simplex_with_extra_edges_matrix} with a different color and line style.

Finally, we get to the vertex $\emptyset$, with $\eta^0(\emptyset)$ starting as the part of the matrix in Figure~\ref{fig:4-simplex_with_extra_edges_matrix} above the horizontal line. Its rank is $3$, and its nullity is $3$. We need the kernel to be $1$-dimensional, so we need to add two additional rows from $(\Span\{(1,1,1,1,1,1)\})^\perp$. We also add $\emptyset$ to $I^1$ twice. This completes the construction of $I^1$ and $\eta^0$.

The resolution goes on for two more steps: $I^2$ is generated by $(2,3,4,5)$, $I^3$ by $(\emptyset)$. The matrices $\eta^k$ are in Figure~\ref{fig:4-simplex_with_extra_edges_matrix}, and the whole resolution is schematically shown in Figure~\ref{fig:4-simplex_with_extra_edges_resolution}.

\begin{figure}[htb]
    \centering
    \includegraphics[width=140mm]{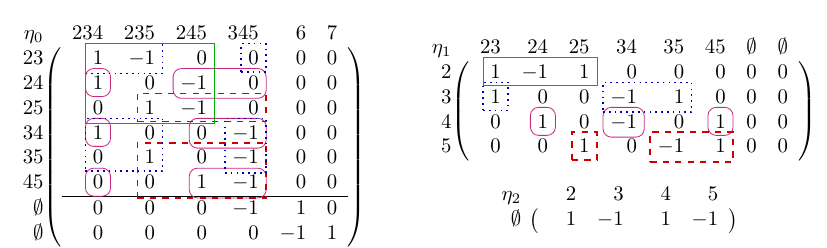}
    \caption{Matrices $\eta^0$, $\eta^1$, $\eta^2$ in Example~\ref{ex:4-simplex_with_extra_edges}, with highlighted submatrices $\eta^k(2)$ (solid green), $\eta^k(3)$ (dotted blue), $\eta^k(4)$ (rounded corners magenta), $\eta^k(5)$ (dashed red). Recall that $\eta^k(\sigma)=\eta^k[\St\sigma,\St\sigma]$, and note that if $\sigma\not\leq\tau$, then $\eta^k[\sigma,\tau]=0$. Recall that the labels omit the vertex $1$ that is present in all the simplices, e.g., $\emptyset$ is the vertex $1$ and $234$ is the tetrahedron $1234$.}
    \label{fig:4-simplex_with_extra_edges_matrix}
\end{figure}
\end{example}

\begin{example}\label{ex:tetraheron_skeleton}
    Let $\Sigma\coloneqq\Dnk{3}{2}$ be the $2$-skeleton of a tetrahedron (whose geometric realization is homeomorphic to the sphere). We give the minimal injective resolution of the constant sheaf $k_{\Sigma}$ in Figure~\ref{fig:tetraheron_skeleton}.
    
    \begin{figure}[htb]
        \centering
        \includegraphics[width=.49\textwidth]{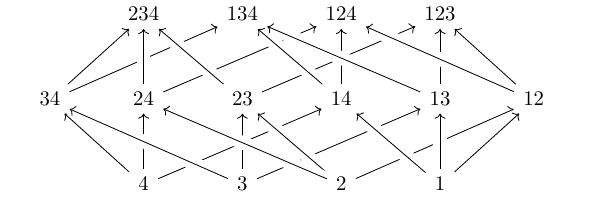}
        \raisebox{.4\height}{\includegraphics[width=.49\textwidth]{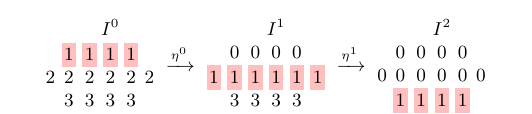}}
        \includegraphics[width=140mm]{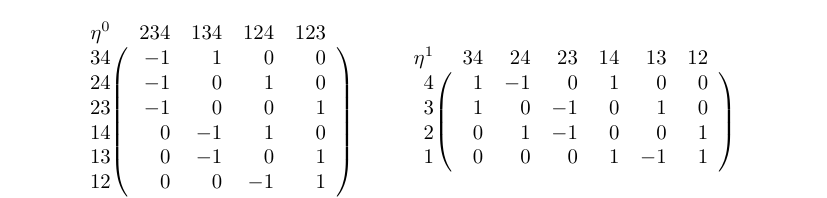}

        \caption{Example~\ref{ex:tetraheron_skeleton}: the minimal injective resolution of the constant sheaf on $\Dnk{3}{2}\cong S^2$. First is $\Dnk{3}{2}$ as a poset, second the dimensions in the injective resolution with highlighted generators, as in Figure~\ref{fig:4-simplex_with_extra_edges_resolution}, and third the matrices describing the natural transformations.}
        \label{fig:tetraheron_skeleton}
    \end{figure}
\end{example}

\paragraph{Complexity Analysis}

We analyze the complexity of finding the minimal injective resolution of the constant sheaf, $k_{\Pi}$, on a poset $\Pi$, with $n$ elements and height $h$, computed by an iterative application of Algorithm~\ref{algo:injective_resolution_tail}. That is, we start with $k_{\Pi}\xrightarrow{\eta^{-1}}I^0$ the minimal injective hull of the constant sheaf as described above, and then iteratively apply Algorithm~\ref{algo:injective_resolution_tail} until $I^d = 0$.

The call of \makeexact\ consists of finding a basis of $(\im(\eta^{d-1}(\sigma)))^{\perp}$, and checking for linear independence of rows of $\eta^d(\sigma)$. Both of those operations can be computed in time at most $\mathcal{O}(c^3)$ with $c$ the maximum of the number of rows of $\eta^{d-1}(\sigma)$ and $\eta^{d}(\sigma)$. In our analysis we ignore the complexity of finding $\St\sigma$ to extract the submatrices from $\eta^d$ in the first place, since it is less expensive than $\mathcal{O}(c^3)$ when we estimate $c$ by the size of $\St\sigma$.

By Corollary~\ref{cor:length_of_resolution}, the length of the minimal injective resolution is at most $h+1$. Therefore, we find it in time $\mathcal{O}(h\cdot n\cdot c^3)$, where \[c=\max_{j,\sigma} \sum_{\pi\in\St\sigma} m^j_{k_{\Pi}}(\pi)\]
is the maximal number of generators over any star throughout the resolution.
This analysis is output-sensitive. To give complexity bounds dependent only on the input, we compare $c$ to the maximal size of a star in $\Pi$. How well we can approximate $c$ this way depends on the structure of $\Pi$.

\begin{definition}\label{def:star_complexity}
    For $\sigma\in\Pi$ we define\[
        m^j(\St\sigma)\coloneqq \sum_{\pi\in\St\sigma}m_{k_{\Pi}}^j(\pi)
    \]
    to be the number of generators over $\St\sigma$ in the $j$-th step of the minimal injective resolution of the constant sheaf on $\Pi$ (recall Definition~\ref{def:multiplicity_of_generators}). Furthermore, we define the \emph{$j$-th star complexity} of $\sigma$ as \[
        \stcplx^j(\sigma)\coloneqq \frac{m^j(\St\sigma)}{\#\St\sigma}.
    \]
\end{definition}

For general posets, $\stcplx^j(\sigma)$ can be arbitrarily large even when lengths of chains are bounded, because sizes of boundaries and coboundaries can be arbitrarily large.
For simplicial complexes, we give an upper bound on $\stcplx^j(\sigma)$ depending on the dimension.

\begin{proposition}\label{prop:upper_bound_on_SC}
    Let $\Pi$ be a simplicial complex, $\sigma\in\Pi$ and $d\coloneqq \dim\St\sigma-\dim\sigma$, where by $\dim\St\sigma$ we mean the maximum dimension of any $\tau\in\St\sigma$. Then\[
        \stcplx^j(\sigma)\leq\binom{d}{j}.
    \]
    This bound is asymptotically tight. If $\Pi=\Dnk{n}{d}$ is the $d$-skeleton of the $n$-simplex, $v$ is a vertex in $\Dnk{n}{d}$, and $j$ is fixed, then \[
        \stcplx^j(v)\xrightarrow{n\rightarrow\infty}\binom{d}{j}.
   \]
\end{proposition}
\begin{proof}
We prove the upper bound using Theorem~\ref{thm:multiplicities} and bounding dimensions of homology groups by dimensions of chain groups:
\begin{multline*}
    m^j(\St\sigma)
    = \sum_{\tau\in\St\sigma} m^j(\tau)
    = \sum_{\tau\in\St\sigma} \dim H_c^{j+\dim\tau}(\St\tau)
    \leq \sum_{\tau\in\St\sigma} \dim C_c^{j+\dim\tau}(\St\tau)
    \\ = \sum_{\tau\in\St\sigma} \# \left\{ \pi \,\middle|\, \tau<_j\pi \right\}
    = \sum_{\pi\in\St\sigma} \# \left\{ \tau \in \St\sigma \,\middle|\, \tau<_j\pi \right\}
    = \sum_{\pi\in\St\sigma} \#\binom{\pi\setminus\sigma}{j} \leq \#\St\sigma\cdot\binom{d}{j},
\end{multline*}
where $d=\dim\St\sigma-\dim\sigma$, and $\tau <_j \pi$ signifies that $\dim\pi - \dim\tau = j$.

Now we analyse $\stcplx^j(\sigma)$ in the $d$-skeleton of the $n$-simplex, $\Dnk{n}{d}$. We use the fact that $\St\sigma$ in $\Dnk{n}{d}$ is combinatorially the same as $\Dnk{n'}{d'}\cup\{\emptyset\}$, with $n'=n-\dim\sigma-1$ and $d'=d-\dim\sigma-1$, using the correspondence $\St\sigma\ni\tau \mapsto \tau\setminus\sigma$. This map induces an isomorphism between the cochain complexes \[
C_c^{\bullet}(\St\sigma) \cong \tilde{C}^{\bullet-\dim\sigma-1}\left(\Dnk{n'}{d'}\right),
\] which, using Theorem~\ref{thm:multiplicities}, implies
\begin{equation*}
    m^j(\sigma) = \dim H_c^{j+\dim\sigma}(\St\sigma) = \dim \tilde{H}^{j-1}\left(\Dnk{n'}{d'}\right).
\end{equation*}
The (standard simplicial) reduced cohomology $\tilde{H}^i\left(\Dnk{n'}{d'}\right)$ is trivial for all $i\neq d'$, and for $i=d'$, we compute the dimension from the Euler characteristic:
\begin{multline*}
    \dim\tilde{H}^{d'}\left(\Dnk{n'}{d'}\right) = 
    (-1)^{d'} \tilde{\chi} \left(\Dnk{n'}{d'}\right) = 
    (-1)^{d'} \left( 1 + \sum_{i=0}^{d'} \binom{n'+1}{i+1} (-1)^i \right) \\ =
    (-1)^{d'-1} \left( \sum_{i=0}^{d'+1} \binom{n'+1}{i} (-1)^i \right) =
    (-1)^{d'-1} \cdot (-1)^{d'+1} \binom{n'}{d'+1} =
    \binom{n'}{d'+1}.
\end{multline*}
Therefore,
\begin{align*}
   m^j(\sigma) = 
    \begin{cases}
        \binom{n-\dim\sigma-1}{d-\dim\sigma} &\text{ if $j = d'+1 = d-\dim\sigma$,} \\
        0 &\text{ otherwise.}
    \end{cases}
\end{align*}
Finally, we compute $m^j(\St v)$ for a vertex $v$:\[
    m^j(\St v)
    = \sum_{\sigma\in\St v} m^j(\sigma)
    = \sum_{\substack{\sigma\in\St v \\ \dim\sigma=d-j}} \binom{n-d+j-1}{j}
    = \binom{n}{d-j}\cdot\binom{n-d+j-1}{j}
\]

We rearrange this as follows
\begin{align*}
    m^j(\St v)
    &= \frac{n!}{(n-d+j)!\ (d-j)!}\cdot \frac{(n-d+j-1)!}{(n-d-1)!\ j!}
    \\&= \frac{n!}{(n-d)!\ d!}\cdot \frac{n-d}{n-d+j}\cdot  \frac{d!}{(d-j)!\ j!}
    = \binom{n}{d} \binom{d}{j} \frac{n-d}{n-d-1}.
\end{align*}
Now we can easily compare this with $\# \St v = \sum_{i=0}^d \binom{n}{i}$. When we fix $d$ and $j$, we get \[
    \lim_{n\rightarrow\infty} \stcplx^j(v) = \lim_{n\rightarrow\infty} \frac{m^j(\St v)}{\# \St v} = \binom{d}{j}.
\]
\end{proof}

\begin{corollary}\label{cor:complexity}
    For a fixed dimension $h$, the Algorithm~\ref{algo:injective_resolution_tail} computes the minimal injective resolution of the constant sheaf on an $h$-dimensional simplicial complex $\Sigma$ in time $\mathcal{O}(n\cdot s^3)$, where $n$ is the cardinality of $\Sigma$ (as an abstract simplicial complex), and $s$ is the cardinality of the largest star in $\Sigma$.
\end{corollary}

\section{Derived Functors}
\label{sec:derived-functors}

A fundamental application of derived sheaf theory is the study of continuous maps. To illustrate these applications, we will recall several functors between categories of sheaves which are induced by continuous maps between topological spaces. In Section \ref{sec:derived-functors:definitions} we define the functors necessary to state the main results of this paper (in Section \ref{sec:MorseTheory}). In Section \ref{sec:algorithms_for_computing_derived_functors} we give detailed algorithms for computing these functors, using the terminology and notation of Section \ref{sec:labeled_matrices}.

We define pushforward and pullback for any order preserving map on posets, and proper pushforward and proper pullback for an inclusion of a \emph{locally closed} subset of a poset.

\begin{definition}\label{def:locally-closed}
    A subset $Z$ of a poset $\Pi$ is \emph{locally closed} if it is an intersection of a closed subset (a down set) and an open set (an upper set). Equivalently, $Z$ is convex: if $a,c\in Z$ and $a\leq b\leq c$, then $b\in Z$.
\end{definition}

\subsection{Definitions of Derived Functors}\label{sec:derived-functors:definitions}

\subsubsection{Pushforward}
\begin{definition}
Suppose $f:\Pi\rightarrow \Lambda $ is an order preserving map of posets, and $F$ is a sheaf on $\Pi$. Then the \emph{pushforward} of $F$ by $f$ is defined by
\begin{align*}
f_\ast F (\lambda)&\coloneqq \varprojlim_{\pi\in f^{-1}(\St\lambda)} F(\pi)\\
&= \bigg\{v\in \bigoplus_{\pi\in f^{-1}(\St\lambda)}F(\pi):\\
& \hspace{.5in} F(\gamma\le \tau)\left(\Proj_{F(\gamma)} (v)\right)=\Proj_{F(\tau)}(v)\text{ for all }(\gamma\le\tau)\in f^{-1}(\St\lambda)\bigg\}, 
\end{align*}
where $\Proj_{F(\gamma)}:\bigoplus_{\pi\in f^{-1}(\St\lambda)}F(\pi)\rightarrow F(\gamma)$ denotes projection onto $F(\gamma)$, and the linear maps $f_\ast F (\kappa\le \lambda)$ are restrictions of the projections $\bigoplus_{\pi\in f^{-1}(\St\kappa)}F(\pi) \rightarrow \bigoplus_{\pi\in f^{-1}(\St\lambda)}F(\pi)$. 

The pushforward is functorial; that is, if $\eta: F\rightarrow G$ is a natural transformation between sheaves $F$ and $G$ on $\Pi$, then
\[
f_\ast \eta: f_\ast F\rightarrow f_\ast G,
\]
is a natural transformation, where $f_\ast\eta(\lambda)$ is obtained by restricting the domain of the sum of linear maps 
\[\sum_{\pi\in f^{-1}(\St\lambda)} \eta(\pi):\bigoplus_{\pi\in f^{-1}(\St\lambda)}F(\pi)\rightarrow \bigoplus_{\pi\in f^{-1}(\St\lambda)}G(\pi).\]
\end{definition}
Notice that if $\eta:F\rightarrow G$ is injective, then $f_\ast \eta$ is injective. However, the same is not necessarily true for surjectivity. In the language of category theory, $f_\ast $ is a left exact functor. It is useful to observe that in our finite setting, pushforwards of injective complexes have a straightforward structure.

\begin{lemma}\label{lem:pushforward_of_injective_sheaf}
    If $I\cong \bigoplus_{0\leq j\leq n} [\pi_j]$ is an injective sheaf on $\Pi$, then $f_\ast I$ is an injective sheaf on~$\Lambda$ given by \[f_\ast I \cong \bigoplus_{0\leq j\leq n} [f(\pi_j)].\]
\end{lemma}
\begin{proof}
    Because $f_\ast\left(\bigoplus_j F_j \right) = \bigoplus_j f_\ast F_j$, it is enough to observe that $f_\ast[\pi] \cong [f(\pi)]$ for any $\pi\in\Pi$. Indeed, by definition \[
        f_\ast[\pi](\sigma) = \begin{cases}
            k \text{ if $\pi\in f^{-1}(\St\sigma)$} \\
            0 \text{ otherwise,}
        \end{cases}
    \]
    and $\pi\in f^{-1}(\St\sigma)$ if and only if $f(\pi)\in \St\sigma$.
\end{proof}

\subsubsection{Pullback}\label{sec:pullback}

\begin{definition}\label{def:pullback}
Suppose $f:\Pi\rightarrow \Lambda $ is an order preserving map of posets, and $F$ is a sheaf on $\Lambda$. Then the \emph{pullback} of $F$ by $f$ is defined by
\begin{align*}
f^\ast F (\pi)&\coloneqq F(f(\pi)), 
\end{align*}
with linear maps $f^\ast F (\pi\le \tau) = F(f(\pi)\le f(\tau))$.

As with the pushforward, the pullback is functorial: if $\eta: F\rightarrow G$ is a natural transformation between sheaves $F$ and $G$ on $\Lambda$, then $\eta$ defines a natural transformation 
\[
f^\ast \eta: f^\ast F\rightarrow f^\ast G,
\]
where $f^\ast\eta(\pi)\coloneqq \eta(f(\pi))$. 
\end{definition}
Unlike the pushforward, the pullback preserves both the injectivity and surjectivity of natural transformations, and defines an exact functor. However, (again in contrast to the pushforward) the pullback does not necessarily preserve injectivity of sheaves. 
\begin{remark}
A classic result in sheaf theory, which continues to hold in the setting of finite posets, is that the functors $f^\ast$ and $f_\ast$ form an adjoint pair:
\[
\Hom_{\Pi}(f^\ast F, G) \cong \Hom_{\Lambda}(F,f_\ast G).
\]
See \cite[Theorem 3.14]{Curry2018} for a detailed proof and \cite[Theorem 3.4.7]{Shepard1985} for an analogous theorem of the adjointness of the corresponding derived functors. 
\end{remark}
\begin{example}\label{ex:pullback_of_injective_not_injective}
Let $\Pi=(a\leq b,c)$, $\Lambda=(a'\leq b',c'\leq d')$ and $f:\Pi\rightarrow\Lambda$ sends $a,b,c$ to $a',b',c'$, respectively. Then the constant sheaf on $\Lambda$, is injective: $k_\Lambda=[d']$. By the definition, its pullback is $f^\ast k_\Lambda = k_\Pi$, which is not an injective sheaf on $\Pi$.

\begin{minipage}{.95\textwidth}
    \centering
    \vspace{10pt}
    \begin{tikzcd}[column sep = 7pt]
        & \phantom{k} & \\[-10pt]
        k & & k \\[-10pt]
        & k \ar{lu} \ar{ru} &
    \end{tikzcd}%
    \hspace{15pt}$\xleftarrow{\qquad f^\ast \qquad}$\hspace{15pt}
    \begin{tikzcd}[column sep = 7pt]
        & k & \\[-10pt]
        k \ar{ru} & & k \ar{lu} \\[-10pt]
        & k \ar{lu} \ar{ru} &
    \end{tikzcd}%
    \vspace{10pt}
\end{minipage}
\end{example}

In Section~\ref{sec:algorithms_for_computing_derived_functors} we describe a pushforward of injective sheaves as a restriction of an injective sheaf over a modified poset.

\subsubsection{Proper Pushforward}

Now we define the proper pushforward. Unlike for the previous functors, here we assume $f\coloneqq i_Z$ to be an inclusion of a locally closed subset $Z$.

\begin{definition}
For a locally closed subset, $Z$, of a poset $\Pi$, and a sheaf $F$ on $Z$, the \emph{proper pushforward} of $F$ by the inclusion $i:Z\hookrightarrow \Pi$ is
\[
i_!F (\sigma )\coloneqq\begin{cases}
F(\sigma)\text{ if }\sigma\in Z,\\
0\text{ else.} 
\end{cases}
\]
As above, it is straightforward to show that $i_!$ is functorial, and that $i_!(\eta)$ preserves injectivity and surjectivity of $\eta$. In other words, $i_!$ is exact.
\end{definition}

Unless $Z$ is closed, a proper pushforward, $(i_Z)_! I$, of an injective sheaf is not injective, because there will necessarily exist two poset elements $a \le b$ with $(i_Z)_! I (a) = 0$ and $(i_Z)_! I (b) \neq 0$. However, it is a subsheaf of an injective sheaf. Indeed, if $I\cong \bigoplus_{0\leq j\leq n} [\pi_j]$ is an injective sheaf on $Z$, and $\hat I\coloneqq (i_Z)_\ast I \cong \bigoplus_{0\leq j\leq n} [\pi_j]$ is an injective sheaf on $\Pi$, then $(i_Z)_! I (\pi) = \hat I (\pi)$ for every $\pi\in Z$.

\subsubsection{Proper Pullback}\label{sec:proper_pullback}
We define the proper pullback, again only for inclusions of locally closed subsets.

\begin{definition}
For a locally closed subset, $Z$, of a poset $\Pi$, and a sheaf $F$ on $Z$, the \emph{proper pullback} of $F$ by the inclusion $i:Z\hookrightarrow\Pi$ is defined by
\[
i^!F(\zeta) := \bigg\{v\in F(\zeta) : F(\zeta\le\tau)(v) = 0\text{ for all }\tau\in \St_\Pi \zeta-Z\bigg\}, 
\]
where $\zeta\in Z$ and $\St_\Pi \zeta$ is the star of $\zeta$ in $\Pi$. 
\end{definition}

As above, it is straightforward to show that $i^!$ is functorial, and that $i^!(\eta)$ preserves the injectivity of $\eta$ but not necessarily surjectivity. As the following lemma shows, computing proper pullbacks of injective sheaves is straightforward. 

\begin{lemma}\label{lem:proper_pullback_of_injective_sheaf}
Let $Z$ be a locally closed subset of $\Pi$. If $I\cong \bigoplus_{\pi\in\Pi} [\pi]^{p_\pi}$ is an injective sheaf on $\Pi$, then $i_Z^! I$ is an injective sheaf on $Z$ with a decomposition \[
    i_Z^! I \cong \bigoplus_{\pi\in Z} [\pi]^{p_\pi}.
\] 
Moreover, if $I^\bullet$ is a minimal injective complex, then so is $i_Z^! I^\bullet$.
\end{lemma}
\begin{proof}
The functor $i^!_Z$ is additive, so it is enough to consider $i^!_Z[\pi]$ for an indecomposable injective sheaf~$[\pi]$.

Assume first that $\pi\notin Z$. Let $\zeta\in Z$. Either $\zeta\le\pi$, and then $i^!_Z[\pi](\zeta) = 0$, because $[\pi](\zeta\le \pi) = \id_k$; or $\zeta\not\le\pi$, and then also $i^!_Z[\pi](\zeta) = 0$, because $[\pi](\zeta)=0$. Therefore, $i^!_Z[\pi]= 0$ for $\pi\notin Z$.

Now we consider the case when $\pi\in Z$. If $\zeta\in Z$ with $\zeta\not\le \pi$, then $i^!_Z[\pi](\zeta) = 0$ because $[\pi](\zeta)=0$. If $\zeta\le \pi$, then the intersection 
\[
\{\gamma\in\Pi:\gamma \le \pi\}\cap \St_\Pi \zeta
\]
is a subset of $Z$ (because $Z$ is locally closed). Therefore, if $\tau\in \St_\Pi\zeta - Z$ then $\tau\not\le \pi$, and hence $[\pi](\zeta\leq\tau)=0$. It follows that $i^!_Z[\pi](\zeta) = k$. Altogether, we have shown that $i^!_Z([\pi]_\Pi)$ is isomorphic to $[\pi]_Z$ (we add the subscripts $\Pi$ and $Z$ to emphasize that $[\pi]_\Pi$ is a sheaf on $\Pi$ and $[\pi]_Z$ is a sheaf on $Z$). 

The minimality of $i_Z^! I^\bullet$, provided that $I^\bullet$ is minimal, then follows from condition~(\ref{thm:minimal_injective_complex_maximal_vectors}) of Theorem~\ref{thm:minimal_injective_complex}: for a natural transformation $\eta^d:I^d\rightarrow I^{d+1}$ and $\pi\in Z$, the map $i_Z^! \eta^d(\pi)$ is just a restriction of $\eta^d(\pi)$. The maximal vectors in $i_Z^! I (\pi)$ are the same as maximal vectors in $I (\pi)$, so if they are sent to zero by $\eta^d$, they are also sent to zero by $i_Z^! \eta^d(\pi)$, verifying the condition~(\ref{thm:minimal_injective_complex_maximal_vectors}).
\end{proof}

Combining this with Lemma~\ref{lem:pushforward_of_injective_sheaf}, we see that $(i_Z)_\ast i_Z^! I$ is the subsheaf of $I$ containing only the indecomposable injective summands $[\pi]$ on $\Pi$ for $\pi \in Z$.

\subsubsection{Summary}\label{sec:derived-functors:summary}

For a quick reference, we provide a summary of how the four introduced functors, $(i_Z)_\ast$, $(i_Z)_!$, $(i_Z)^\ast$, $(i_Z)^!$, act on injective sheaves in case of an inclusion $i_Z : Z\hookrightarrow\Pi$. The cones in the schema represent indecomposable injectives $[\pi]$.

\noindent\begin{minipage}{\textwidth}
    \centering
    \vspace{3mm}
    \includegraphics[width=65mm]{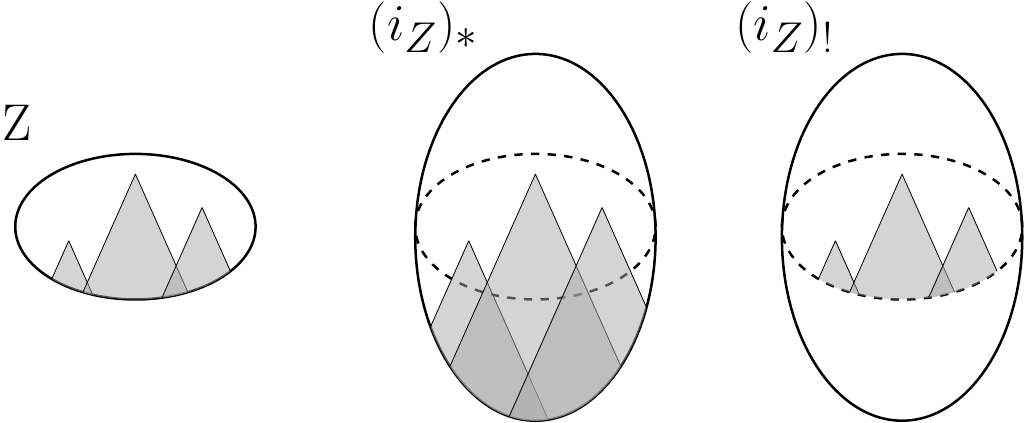}
    \hfill
    \includegraphics[width=65mm]{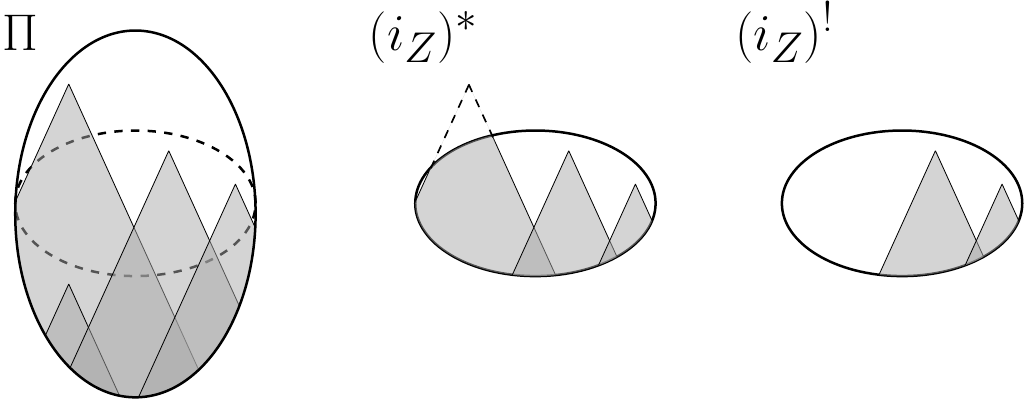}
    \vspace{3mm}
\end{minipage}

We see that while $(i_Z)_\ast$ and $(i_Z)^!$ are built of full indecomposable injectives, this is not the case for the other two---$(i_Z)^\ast$ contains parts of indecomposable injectives clipped from above, and $(i_Z)_!$ contains parts clipped from below. In Section~\ref{sec:algorithms_for_computing_derived_functors} we discuss how to compute injective resolutions of images of injective complexes. While for $(i_Z)_\ast$ and $(i_Z)^!$ we already have injective complexes, we need to carry out extra computations in case of $(i_Z)^\ast$ and $(i_Z)_!$, which we describe as Algorithm~\ref{algo:pullback_of_complex_of_injectives} and Algorithm~\ref{algo:proper_pushforward_of_complex_of_injectives}, respectively.

\subsection{Derived Functors and Hypercohomology} \label{sec:hypercohomology}
We now apply the above operations to complexes of injective sheaves. 

\begin{definition}\label{def:derived-functor}
Let $f:\Sigma\rightarrow\Lambda$ be an order preserving map of posets and $(I^\bullet,\eta^\bullet)\in D^b(\Sigma)$ be a minimal injective complex on $\Sigma$. Let $\varphi$ be one of the functors $f_\ast$, $f^\ast$, $f_!$, or $f^!$, defined above (we assume $f$ is an inclusion in the proper cases). Define $\Rfunc\varphi (I^\bullet)\in D^b(\Lambda)$ to be the minimal injective resolution of the complex 
\[
\cdots\rightarrow \varphi(I^d)\xrightarrow{\varphi\eta^d}\varphi(I^{d+1})\xrightarrow{\varphi\eta^{d+1}}\varphi(I^{d+2})\rightarrow \cdots.
\]
As with other functors, we usually write $\Rfunc\varphi I^\bullet=\Rfunc\varphi(I^\bullet)$ for brevity.
\end{definition} 

\begin{remark}
    For a functor $\varphi$ on sheaves, the above defined functor $\Rfunc\varphi$ plays the role of the right derived functor $R\varphi$ from the classical theory of derived categories, and we therefore refer to $\Rfunc\varphi$ as the \emph{derived functor} of $\varphi$.
\end{remark}
\begin{remark}
    While most authors would not use $R\varphi$ for an exact functor $\varphi$, we still write $\Rfunc\varphi$ to ensure that the result is a minimal injective chain.
    For example, since $f^\ast$ is an exact functor it naturally extends to a functor on the derived category, and all of the higher cohomology sheaves vanish for the right derived functor $Rf^\ast$. Most authors then use $f^\ast$ to denote the functor on the category of sheaves as well as on the derived category of sheaves. However, we note that while $f^\ast$ is exact, it does not preserve the injectivity of sheaves, and therefore does not immediately define a functor on (the chosen skeleton of) the derived category of sheaves as we have defined it. To define such a functor, we must post-compose $f^\ast$ with a minimal injective resolution.
    Therefore, we write $\Rfunc f^\ast$ and highlight the algorithmic difference from simply applying $f^\ast$ to a complex of sheaves. The same remarks apply for $f_!$.
\end{remark}
\begin{remark}
    Secondly, in the most classical settings (such as when studying smooth maps between manifolds), $f^!$ does not exist as a functor on the category of sheaves, but only as a functor on the derived category (this is indeed one of the motivations for defining the derived category). In these settings, $f^!$ is often introduced as the left adjoint of $Rf_!$. However, for simplicity, when discussing $f^!$ in this paper we restrict our attention to only the poset maps $f:Z\hookrightarrow \Pi$ which are inclusions of locally closed subposets. In this restricted setting, the functors $f^!$ are well-defined as functors on the category of sheaves, and justify the additional notation $R f^!$ for the corresponding right derived functor on the derived category of sheaves, and our corresponding functor $\Rfunc f^!$. It would be interesting and worthwhile to extend our results to more general poset maps, and to give an algorithm for computing $f^!$ in a more general setting when it is classically defined as the left adjoint of $Rf_!$.
\end{remark}

\begin{definition}\label{def:hypercohomology}
Let $p:\Sigma \rightarrow \{\star\}$ be the projection map to the single point poset. Let $I^\bullet$ be a complex of injective sheaves on $\Sigma$. Define the \emph{hypercohomology} of $I^\bullet$ to be
\[
\mathbb{H}^d( I^\bullet) \coloneqq H^d(\Rfunc p_\ast I^\bullet).
\]
Notice that $\Rfunc p_\ast(I^\bullet)$ is simply a complex of vector spaces. As such, the minimality condition forces all maps to be zero maps, which implies that $H^d(\Rfunc p_\ast I^\bullet) = (\Rfunc p_\ast I^\bullet)^d$ when we view the right-hand side as the vector space over the single element.
\end{definition}

\begin{corollary}\label{cor:proper-pull-back-multiplicity}
Assume $I^\bullet\in D^b(\Sigma)$ is a minimal injective complex. Let $m_{I^\bullet}^d(\pi)$ be the multiplicity of the indecomposable summand $[\pi]$ in $I^d$. If $i_\pi:\{\pi\}\hookrightarrow \Sigma$, then
\[
\dim \mathbb{H}^d\Rfunc i_\pi^! I^\bullet = m_{I^\bullet}^d(\pi).
\]
\end{corollary}
\begin{proof}
This follows directly from Lemma \ref{lem:proper_pullback_of_injective_sheaf}: we have $(\Rfunc i_\pi^! I^\bullet)^d = i_\pi^! I^d \cong [\pi]^{m_{I^\bullet}^d(\pi)}$, and since all vectors on a singleton are maximal, all the complex maps are zero by the minimality and condition~(\ref{thm:minimal_injective_complex_maximal_vectors}) of Theorem~\ref{thm:minimal_injective_complex}.
\end{proof}

In the setting of simplicial complexes (or, more generally, cellular complexes) the cohomology sheaves of the derived pushforward of a constant sheaf are closely related to singular cohomology groups of level-sets. This allows us to give some intuition behind derived pushforwards with the proposition below. Before formulating it, we recall two classical notions of cohomology  for topological spaces (see, e.g., \cite[2.6.8]{KashiwaraSchapira1994}).

\begin{definition}\label{def:cohomology_functorial}
     Let $X$ be a topological space, $p:\Pi\rightarrow \{\mathrm{pt}\}$ a constant map to a point, and $F$ a sheaf on $X$ with injective resolution $I^\bullet$. Then
     \begin{align*}
        H^d(X;F) &\coloneqq H^d(\Rfunc p_\ast F) = H^d(\Rfunc p_\ast I^\bullet) , \\
        H_c^d(X;F) &\coloneqq H^d(\Rfunc p_! F) = H^d(\Rfunc p_! I^\bullet).
     \end{align*}
     By $H^d(X;k)$ and $H_c^d(X;k)$ we denote the above with the choice $F=k_X$, the constant sheaf on~$X$. Those definitions yield the standard and compactly supported cohomology (see, e.g., \cite{hatcher}), respectively.
\end{definition}

We use these notions mainly for geometric realizations of simplicial complexes and their subsets. Note that we cannot just blindly replace $X$ by a poset---for one, we did not give a definition for $p_!$, and, more importantly, if $\{v\}$ is a vertex and $\{e\}$ an open edge, then they are equal as posets, but $H_c^1(|\{v\}|;k)=0$ while $H_c^1(|\{e\}|;k)\cong k$. We can, however, replace $X$ by a face poset of a simplicial complex in $H^d(X; F)$.

\begin{proposition}\label{prop:persistent-homology}
Suppose $f:\Sigma\rightarrow \Lambda$ is a simplicial map. As sheaves on $\Lambda$, 
\begin{align*}
    H^d\Rfunc f_\ast k_\Sigma  &\cong H^d(|f^{-1}(\St\blank)|;k)
\end{align*}
where $H^d(|f^{-1}(\St\blank)|;k)$ is the sheaf defined by associating the simplex $\lambda$ to the singular cohomology of the geometric realization of $f^{-1}(\St\lambda)$ (with linear maps induced by inclusion). 
\end{proposition}
\begin{proof}
We provide a proof at the end of \ref{app:injective_resolution_order_complex}.

\end{proof}

\subsection{Recollement for Posets}
Suppose $\Pi=C\sqcup U$ is a partition of the finite poset $\Pi$ into a closed set $C$ and an open set $U$ (relative to the Alexandrov topology). Let $i_C:C\hookrightarrow \Pi$, $j_U:U\hookrightarrow \Pi$, and $p,q,r$ denote projections to the single element poset $\{\star\}$:
\begin{center}
\begin{tikzcd}
C \ar{r}{i_C}\ar[swap]{dr}{p}
& \Pi\ar{d}{q}& U\ar[swap]{l}{j_U}\ar{dl}{r} \\
 &\{\star\}&
\end{tikzcd}
\end{center}
\begin{lemma}\label{lemma:exact-sequence}
If $I^\bullet\in D^b(\Pi)$, then there are long exact sequences of hypercohomology groups (recalling Definitions \ref{def:derived-functor} and \ref{def:hypercohomology})
\begin{align}
\cdots\rightarrow& \mathbb{H}^d(\Rfunc i_C^! I^\bullet)\rightarrow \mathbb{H}^d( I^\bullet)\rightarrow
\mathbb{H}^d(\Rfunc j_U^\ast I^\bullet)\\
&\rightarrow \mathbb{H}^{d+1}(\Rfunc i_C^!I^\bullet)\rightarrow\cdots\nonumber
\end{align}
and 
\begin{align}
\cdots\rightarrow& \mathbb{H}^d(\Rfunc j_{U!}\Rfunc j_U^\ast I^\bullet)\rightarrow \mathbb{H}^d( I^\bullet)\rightarrow
\mathbb{H}^d(\Rfunc i_C^\ast I^\bullet)\\
&\rightarrow \mathbb{H}^{d+1}(\Rfunc j_{U!}\Rfunc j_U^\ast I^\bullet)\rightarrow\cdots\nonumber
\end{align}
\end{lemma}
\begin{proof}
For any injective sheaf $I$ on $\Pi$, there are short exact sequence of sheaves: 
\[
0\rightarrow i_{C\ast}i_{C}^! I\rightarrow I\rightarrow j_{U\ast}j_U^\ast I\rightarrow 0, 
\]
and
\[
0\rightarrow j_{U!}j^\ast_U I\rightarrow I\rightarrow i_{C\ast}i^\ast_C I\rightarrow 0. 
\]
These exact sequences define distinguished triangles in $D^b(\Pi)$, which then yield the desired long exact sequence of hypercohomology (see \cite[Propositions 1.7.5 and 1.8.8, Remark 2.6.10]{KashiwaraSchapira1994}). 
\end{proof}

\subsection{Algorithms for Computing Derived Functors}\label{sec:algorithms_for_computing_derived_functors}

In this section we describe how to compute minimal injective resolutions, $\Rfunc f_\ast$, $\Rfunc f^\ast$, $\Rfunc f_!$, $\Rfunc f^!$, of the introduced functors, $f_\ast$, $f^\ast$, $f_!$, $f^!$, applied to injective complexes. We first introduce a notion of a poset mapping cylinder, and then describe algorithms to compute each of the four functors. Concrete examples of the computations are presented in Section~\ref{sec:examples}. We assume $f:\Pi\rightarrow\Lambda$ is an order preserving map of posets and $(I^\bullet,\eta^\bullet)$ is a complex of injective sheaves represented by labeled matrices $\eta^\bullet$. We will assume each $I^\bullet$ to start in degree $0$. If it is not the case, we can shift the degrees in the obvious manner---if $n$ is the smallest degree with non-zero $I^n$, we consider new matrices $\hat\eta^d = \eta^{d+n}$. Note that this means that we start with a zero matrix $\eta^{-1}$ with rows labeled according to a decomposition of $I^0$, and no columns, as discussed after Definition~\ref{def:complex_of_labeled_matrices}.

For the discussions we also fix $I=\bigoplus_i[\pi_i]$, $J=\bigoplus_j[\sigma_j]$ and $\eta:I\rightarrow J$ a natural transformation represented by a labeled matrix $\eta$.

\subsubsection{Poset Mapping Cylinder of an Order Preserving Map}\label{sec:poset_mapping_cylinder}

The strength of our approach is that we can compute all the introduced derived functors using labeled matrices, despite the fact that pullback and proper pushforward do not preserve injectivity. For that purpose, we use poset mapping cylinders, and consider injective complexes on those. The construction is also known as non-Hausdorff mapping cylinder in the study of finite $T_0$ topological spaces, see, e.g., \cite[Definition~2.8.1]{BarmakBook} or before in \cite{BarmakMinian2008}.

\begin{definition}\label{def:mapping_cylinder_for_posts}
    Given an order preserving map between two posets, $f:\Pi\rightarrow\Lambda$, its \emph{mapping cylinder}, $\Pi\sqcup_f\Lambda$, is the poset with the underlying set the disjoint union $\Pi\sqcup\Lambda$, and the order defined as follows: $\pi\leq \tau$ in $\Pi\sqcup_f\Lambda$ if and only if one of the following conditions is satisfied 
    \begin{itemize}
        \item $\pi,\tau\in\Pi$ and $\pi\leq \tau$ in $\Pi$,
        \item $\pi,\tau\in\Lambda$ and $\pi\leq \tau$ in $\Lambda$, or
        \item $\pi\in\Pi$, $\tau\in\Lambda$ and $f(\pi)\leq \tau$ in $\Lambda$.
    \end{itemize}
    We also call it a \emph{poset cylinder} to more clearly distinguish it from the mapping cone of a complex morphism.
\end{definition}

In other words, the order relation in $\Pi\sqcup_f\Lambda$ is generated by new relations $\pi\leq f(\pi)$ for every $\pi\in\Pi$ together with the order relations on $\Pi$ and $\Lambda$. Note that every chain in $\Pi\sqcup_f\Lambda$ decomposes into a prefix of elements from $\Pi$ and suffix of elements from $\Lambda$ (either possibly empty). The following is an important property of this construction.

\begin{lemma}\label{lem:star_in_mapping_cylinder}
    For $\pi\in\Pi$, the star of $\pi$ in $\Pi\sqcup_f\Lambda$ is the union of stars of $\pi$ in $\Pi$ and $f(\pi)$ in $\Lambda$: \[ \St_{\Pi\sqcup_f\Lambda}\pi =  \St_{\Pi}\pi \cup \St_{\Lambda} f(\pi).
    \]
\end{lemma}
\begin{proof}
The inclusion ``$\supseteq$'' is immediate since $\pi\leq f(\pi)$ in $\Pi\cup_f\Lambda$. For ``$\subseteq$'', let $\pi\leq\tau$ in $\Pi\cup_f\Lambda$. Then by definition either $\tau\in\Pi$ and $\tau\in\St_{\Pi}\pi$; or $\tau\in\Lambda$ and then $f(\pi)\leq \tau$ in $\Lambda$, i.e., $\tau\in\St_{\Lambda} f(\pi)$.
\end{proof}

In particular, the lemma implies that if $\eta$ is a labeled matrix representing a natural transformation between two injective sheaves on $\Pi\sqcup_f \Lambda$, and if all its labels are in $\Lambda$, then $\eta(\pi)=\eta(f(\pi))$ for all $\pi\in\Pi$.

\subsubsection{Computing Pushforward}\label{sec:pushforward_algo}
For labeled matrices, computing pushforwards is very simple. By Lemma~\ref{lem:pushforward_of_injective_sheaf}, we only need to replace every label $\pi$ by $f(\pi)$.
This yields a labeled matrix representation of an injective complex, but it might not be minimal. Indeed, when $f$ is not injective, $f_\ast\eta$ might send some maximal vectors to non-zero vectors even if $\eta$ did not. Therefore, we compute $\Rfunc f_\ast I^\bullet$ by relabeling and then minimizing.
\begin{proposition}\label{lem:computation_pushforward}
    Given an injective complex $(I^\bullet, \eta^\bullet)$ on $\Pi$ represented by labeled matrices, and an order preserving map $f: \Pi\rightarrow \Lambda$, we obtain $\Rfunc f_\ast I^\bullet$ by performing two steps:
\begin{enumerate}
    \item change all labels $\pi$ to $f(\pi)$ in all matrices $\eta^d$,
    \item minimize the complex as described in Section~\ref{sec:derived_categories:algorithm_minimize}.
\end{enumerate}
We obtain $\Rfunc f_\ast I^\bullet$ regardless of whether $I^\bullet$ is minimal to start with. 
\end{proposition}

\subsubsection{Computing Pullback}\label{sec:pullback_algo}

The pullback $f^\ast I$ of an injective sheaf $I=\bigoplus_j[\lambda_j]$ on $\Lambda$ might not be injective, but we can describe it as a restriction of an injective sheaf on the poset cylinder $\Pi\sqcup_f\Lambda$ defined above. This injective sheaf is $\hat I\coloneqq \bigoplus_j[\lambda_j]$, where we now use the same symbol $[\lambda_j]$ for two different sheaves: once on $\Lambda$, once on $\Pi\sqcup_f\Lambda$. We claim that $\hat I$ restricted to $\Pi$ is $f^\ast I$.
Indeed, for any $\pi\in\Pi$ we have\[
    f^\ast I(\pi) \overset{Def~\ref{def:pullback}}{=}
    I(f(\pi)) =
    \bigoplus_{\substack{j \\ \lambda_j \in \St_{\Lambda}(f(\pi))}} [\lambda_j](f(\pi)) \overset{Lem~\ref{lem:star_in_mapping_cylinder}}{=}
    \bigoplus_{\substack{j \\ \lambda_j \in \St_{\Pi\sqcup_f\Lambda}(\pi)}} [\lambda_j](\pi) = \hat{I}(\pi).
\]
The same construction works for the pullback of a natural transformation $\eta$. This allows us to describe $f^\ast \eta$ by the same labeled matrices as $\eta$.

\begin{lemma}\label{lem:pullback_cone_representation}
    Given an injective complex $(I^\bullet, \eta^\bullet)$ on $\Lambda$, the pullback $f^\bullet I^\bullet$ is the restriction of $(\hat{I}^\bullet, \hat{\eta}^\bullet)$ from $\Pi\sqcup_f\Lambda$ to $\Pi$. The same complex of labeled matrices represents $I^\bullet$ and $\hat{I}^\bullet$.
\end{lemma} 

We describe how to compute a labeled matrix representation of $\Rfunc f^\ast I^\bullet$, given a labeled matrix representation of $(I^\bullet,\eta^\bullet)$. The key insight is a relation between a complex on the poset cylinder and a mapping cone of a complex morphism.

Consider now a complex $(J^\bullet, \delta^\bullet)$ on $\Pi$---it can also be viewed as a complex on $\Pi\sqcup_f\Lambda$ by setting $J^\bullet(\lambda) = 0$ for all $\lambda\in\Lambda$---and a morphism $\hat{\alpha}^\bullet:\hat{I}^\bullet \rightarrow J^\bullet$. We construct a complex $(C^\bullet, \gamma^\bullet)$ on $\Pi\sqcup_f\Lambda$ by setting\[
    C^d \coloneqq \hat I^{d+1} \oplus J^d,
    \hspace{20mm}
    \gamma^d \coloneqq
    \begin{pmatrix}
        -\hat{\eta}^{d+1} & 0 \\
        \hat{\alpha}^{d+1} & \delta^{d}
    \end{pmatrix}.
\]
Recalling Definition~\ref{def:mapping_cone_comlex} together with Lemma~\ref{lem:pullback_cone_representation} above gives a meaning to $C^\bullet$.

\begin{lemma}\label{lem:poset_cone_and_mapping_cone}
    The restriction of $(C^\bullet, \gamma^\bullet)$ from $\Pi\sqcup_f\Lambda$ to $\Pi$ is the mapping cone of the restriction of $\hat{\alpha}^\bullet$ to $\Pi$, which is a morphism $\alpha^\bullet: f^\ast I^\bullet \rightarrow J^\bullet$.
\end{lemma}

This gives us a recipe to compute an injective resolution of $f^\ast I^\bullet$. We construct $J^\bullet$ such that $C^\bullet$ restricted to $\Pi$ is exact. By Lemma~\ref{lem:mapping_cone_of_quasi-isomorphism}, we then have $\alpha^\bullet$ a quasi-isomorphism, which means that $J^\bullet$ is the resolution we are looking for. The construction itself is analogous to computing injective resolution of a sheaf in Section~\ref{sec:derived_categories:constant_sheaf}. We start with $J^\bullet$ empty, and add indecomposable injective sheaves one by one to ensure exactness at each $\pi\in\Pi$ and for each degree $d$. The procedure is described as Algorithm~\ref{algo:pullback_of_complex_of_injectives}. We remark that on line~\ref{algoline:pullback_of_complex_of_injectives_add_row} we can be adding $r$ instead of $-r$, since complexes given by $\eta^\bullet$ and $-\eta^\bullet$ are quasi-isomorphic---we switch the sign to be consistent with our choice of signs in the definition of the mapping cone.

\begin{algorithm}[htb]
\caption{Computing $\Rfunc  f^\ast I^\bullet$}\label{algo:pullback_of_complex_of_injectives}

\textbf{Inputs:}\vspace{-2mm}
\begin{itemize}
    \item $\eta^{-1},\dots,\eta^n$ labeled matrices representing a complex of injective sheaves $(I^\bullet,\eta^\bullet)$ on $\Lambda$
    \item $f:\Pi\rightarrow\Lambda$ order preserving map, with $\Pi$ and $\Lambda$ disjoint
\end{itemize}
\textbf{Output:} $\delta^{-1},\dots,\delta^m$ labeled matrices representing complex $\Rfunc  f^\ast (I^\bullet)$ on $\Pi$ \\
\textbf{Notation:}\vspace{-2mm}
\begin{itemize}
    \item $I^d$ denotes the tuple of labels of columns of $\eta^d$, or equivalently rows of $\eta^{d-1}$.
    \item $C^d$ denotes the tuple of labels of columns of $\gamma^d$, or equivalently rows of $\gamma^{d-1}$.
\end{itemize}
\begin{algorithmic}[1]
    \State $\gamma^{-2} \gets$ labeled matrix with no columns, and row labels $I^0$ 
    
    \Comment{We are building complex $(C^\bullet,\gamma^\bullet)$ on $\Pi\sqcup_f \Lambda$}
    \State $k\gets -1$
    \While{$C^d\neq 0$ or $d < n$}
        \State $\gamma^{d} \gets$ labeled matrix with columns labeled by $C^d$
        \ForEach{$r$ row in $\eta^{d+1}$}
            \State \label{algoline:pullback_of_complex_of_injectives_add_row} add a row $-r$ to $\gamma^d$
            \Comment{Including the label; with $0$ in columns not labeled by $\Lambda$}
        \EndFor

        \For{$\pi\in\Pi$ in non-increasing order}
            \State $\gamma^d \gets$\makeexact$(\gamma^{d-1},\gamma^d,\pi)$
            \Comment{Algorithm~\ref{algo:makeeact}}
        \EndFor
        \State $d \gets d+1$
    \EndWhile
    \State $m\gets \max\left\{ m \, \middle| \, C^m\neq 0 \right\}$
    \State \Return $\gamma^{-1}[\Pi,\Pi], \dots, \gamma^m[\Pi,\Pi]$
    \Comment{Note that $\gamma^{-2}[\Pi,\Pi]$ has no columns and no rows.}
\end{algorithmic}
\end{algorithm}

\begin{proposition}\label{lem:correctness_pullback_algorithm}
    Given an injective complex $I^\bullet$ on $\Lambda$ and an order preserving map $f: \Pi\rightarrow \Lambda$, Algorithm~\ref{algo:pullback_of_complex_of_injectives} outputs a complex of labeled matrices representing $\Rfunc f^\ast I^\bullet$. The output is minimal regardless of whether $I^\bullet$ is minimal.
\end{proposition}
\begin{proof}
    We use symbols as defined above the proposition. Since $\hat{I}^\bullet$ only has indecomposables generated be elements in $\Lambda$, and all the newly added rows are labeled by $\Pi$, the outputted complex of labeled matrices, $\gamma^\bullet[\Pi, \Pi]$, represents the injective complex $(J^\bullet, \delta^\bullet)$ from Lemma~\ref{lem:poset_cone_and_mapping_cone}. By its claim and Lemma~\ref{lem:mapping_cone_of_quasi-isomorphism}, it is quasi-isomorphic to $f^\ast I^\bullet$ iff $C^\bullet$ is exact for all $\pi\in\Pi$. This is true by construction, which we argue as in Lemma~\ref{lem:correctness_injective_resolution_sheaf}: by Lemma~\ref{lem:makeexact}, we have exactness right after the application of \makeexact, and due to the induction order this does not change later in the algorithm.

    We claim that the resulting complex is minimal (no mater whether $I^\bullet$ was minimal). According to Section~\ref{sec:derived_categories:algorithm_minimize}, we need $\delta^d[\pi,\pi]$ to be zero (or empty) for every element $\pi\in\Pi$ and degree $d$. Since $\delta^d[\pi,\pi]=\gamma^d[\pi,\pi]$, we can prove this for $\gamma^d$ and $\pi\in\Pi$. All rows labeled by $\pi$ were added by calling \makeexact. By Lemma~\ref{lem:makeexact}, if we added a row labeled by $\pi$ in $\gamma^{d-1}$ as the $j$-th row of $\gamma^{d-1}(\pi)$, then the canonical vector $e_j$ is in the image of $\gamma^{d-1}(\pi)$. Since we have a complex, this implies $\gamma^d(\pi)\cdot e_j=0$, which means that the $j$-th column of $\gamma^d(\pi)$ is $0$. This argument is true for all columns labeled by $\pi$ in $\gamma^d$, so we have $\gamma^d[\St\pi,\pi]=0$, and in particular $\gamma^d[\pi,\pi]=0$.
\end{proof}

\paragraph{Alternative minimization algorithm} As an alternative to Algorithm~\ref{algo:peeling} for turning a complex of injective sheaves into the minimal one, we can perform the minimization by computing $\Rfunc\,\id^*I^\bullet$. This complex is minimal, because Algorithm~\ref{algo:pullback_of_complex_of_injectives} always returns a minimal complex, and it is quasi-isomorphic to $I^\bullet$ by the functoriality of (derived) pullback.

\subsubsection{Computing Proper Pushforward} \label{sec:proper_pushforward_algo}

Fix a locally closed subset $Z\xhookrightarrow{i_Z} \Pi$ and a complex $(I^\bullet,\eta^\bullet)$ of injective sheaves on $Z$, and recall that $\Cl Z$ denotes the downwards closure of $Z$ in $\Pi$. In Algorithm~\ref{algo:proper_pushforward_of_complex_of_injectives} we give a construction of $\Rfunc (i_Z)_!I^\bullet$. The construction is very similar to Algorithm~\ref{algo:pullback_of_complex_of_injectives}---we again start with some prefilled labeled matrices, and add new rows to force exactness of the complex on a subset of $\Pi$. However, the interpretation of the constructed matrices is different. Whereas the constructed matrices in Algorithm~\ref{algo:pullback_of_complex_of_injectives} describe a mapping cone, in Algorithm~\ref{algo:proper_pushforward_of_complex_of_injectives} they describe directly the complex of injective sheaves we are looking for.

As noted above in Section~\ref{sec:proper_pullback}, $(i_Z)_! I^\bullet$ is a subcomplex of $(i_Z)_\ast I^\bullet$, which is the complex represented by the same labeled matrices as $I^\bullet$, viewed as a complex over $\Pi$ rather than $Z$. They differ on $\Cl Z \setminus Z$, where the former is $0$ while the latter is non-zero. The idea is to extend $(i_Z)_\ast I^\bullet$ to $J^\bullet$ such that the inclusion $\alpha^\bullet: (i_Z)_! I^\bullet \hookrightarrow J^\bullet$ becomes a quasi-isomorphism. Consider $(C^\bullet, \gamma^\bullet)$, the mapping cone of $\alpha^\bullet$, and recall that by Lemma~\ref{lem:mapping_cone_of_quasi-isomorphism} we want it to be an exact complex. For $\pi\in Z$ we can leave $J^\bullet(\pi) = (i_Z)_\ast I^\bullet(\pi)$, and $C^\bullet(\pi)$ is exact, since $\alpha^\bullet(\pi)$ is the identity. Same for $\pi\in\Pi\setminus\Cl Z$ where all considered sheaves are zero. We need to fix the chains for $\pi\in\Cl Z \setminus Z$, which we do by adding indecomposables $[\pi]$ via the \makeexact\ algorithm. The procedure is described as Algorithm~\ref{algo:proper_pushforward_of_complex_of_injectives}.

\begin{algorithm}[htb]
    \caption{Computing $\Rfunc ({i_Z})_! I^\bullet$}\label{algo:proper_pushforward_of_complex_of_injectives}
    
    \textbf{Inputs:}\vspace{-2mm}
    \begin{itemize}
        \item $\eta^{-1},\dots,\eta^n$ labeled matrices representing a complex of injective sheaves $(I^\bullet,\eta^\bullet)$ on $Z$
        \item $\Pi$ a poset, and $Z\subseteq\Pi$ locally closed
    \end{itemize}
    \textbf{Output:} $\delta^{-1},\dots,\delta^m$ labeled matrices representing complex $\Rfunc ({i_Z})_! I^\bullet$ on $\Pi$ \\
    \textbf{Notation:}\vspace{-2mm}
    \begin{itemize}
        \item $I^d$ denotes the tuple of labels of columns of $\eta^d$, or equivalently rows of $\eta^{d-1}$.
        \item $J^d$ denotes the tuple of labels of columns of $\delta^d$, or equivalently rows of $\delta^{d-1}$.
    \end{itemize}
    \begin{algorithmic}[1]
        \State $\delta^{-1}\gets$ labeled matrix with no columns and row labels $I^0$
        \State $d\gets 0$
        \While{$J^d\neq 0$ or $d \leq n$}
            \State $\delta^{d} \gets$ labeled matrix with columns labeled by $J^d$
            \ForEach{$r$ row in $\eta^{d}$}
                \State add a row $r$ to $\delta^d$
                \Comment{Including the label; with $0$ in columns not labeled by $Z$}
            \EndFor

            \For{$\pi\in \Cl Z \setminus Z$ in non-increasing order} \label{algoline:proper_pushforward_of_complex_of_injectives_makeexact_loop}
                \State $\delta^d \gets$\makeexact$(\delta^{d-1},\delta^d,\pi)$
                \Comment{Algorithm~\ref{algo:makeeact}}
            \EndFor
            \State $d \gets d+1$
        \EndWhile
        \State $m\gets \max\left\{ m \, \middle| \, J^m\neq 0 \right\}$
        \State \Return $\delta^{-1}, \dots, \delta^{m}$
    \end{algorithmic}
\end{algorithm}

\begin{proposition}\label{lem:correctness_proper_pushforward_algorithm}
    Given a locally closed subset $i_Z: Z \hookrightarrow \Pi$, and a minimal injective complex $I^\bullet$ on $Z$, Algorithm~\ref{algo:proper_pushforward_of_complex_of_injectives} outputs a complex of labeled matrices representing $\Rfunc (i_Z)_! I^\bullet$.
\end{proposition}
\begin{proof}
    The complex returned by the algorithm is $(J^\bullet, \delta^\bullet)$. We argued above the proposition that $\alpha^\bullet: (i_Z)_! I^\bullet \hookrightarrow J^\bullet$ is an injection and its mapping cone, $(C^\bullet, \gamma^\bullet)$, is exact for all $\pi\in\Pi\setminus \Cl Z$, as for those $\pi$ we have $(i_Z)_! I^\bullet(\pi) = J^\bullet (\pi)$. The remaining elements are $\pi\in \Cl Z \setminus Z$, which we loop through on line~\ref{algoline:proper_pushforward_of_complex_of_injectives_makeexact_loop}. Let $\pi\in \Cl Z \setminus Z$. Then $(i_Z)_! I^\bullet (\pi)$ is the zero complex, and $\gamma^d(\pi) = \delta^d(\pi)$ for all $d$.
    The exactness of $\delta^\bullet (\pi)$ in degree $d$ is forced by the construction after processing it---see Lemma~\ref{lem:makeexact}---and it does not change afterwards due to the order in which we process the elements.

    To show minimality of $J^\bullet$, we argue that $\delta^d[\pi, \pi]=0$ for each $\pi\in\Pi$ and degree $d$. This is trivial for $\pi\in\Pi\setminus\Cl Z$, follows from the minimality of $I^\bullet$ for $\pi\in Z$, and for $\pi\in\Cl Z\setminus Z$ we follow the same argument as for Algorithm~\ref{algo:pullback_of_complex_of_injectives} in Proposition~\ref{lem:correctness_pullback_algorithm}.
\end{proof}

\subsubsection{Computing Proper Pullback} \label{sec:proper_pullback_algo}

By Lemma~\ref{lem:proper_pullback_of_injective_sheaf}, to compute the proper pullback by an inclusion of a locally closed subset $Z$, we just need to throw away the indecomposables $[\pi]$ for $\pi\not\in Z$.

\begin{proposition}\label{lem:computation_proper_pullback}
     Given a locally closed subset $i_Z: Z \hookrightarrow \Pi$, and a minimal injective complex $(I^\bullet, \eta^\bullet)$ on $Z$ represented by labeled matrices, the proper pullback $\Rfunc i_Z^! I^\bullet$ is represented by matrices $\eta^d[Z, Z]$.
\end{proposition}

\paragraph{An interesting formula} Illustrating the utility of our algorithmic treatment of derived categories of sheaves, we make the following observation, which leads to an interesting formula for finite posets. We noted that Algorithm~\ref{algo:pullback_of_complex_of_injectives} and Algorithm~\ref{algo:proper_pushforward_of_complex_of_injectives} are essentially performing the same procedure. One difference is that in the former we start the construction of the $d$-th matrix by adding rows from $\eta^{d+1}$, while in the latter we add rows from $\eta^d$. The second difference is that in Algorithm~\ref{algo:pullback_of_complex_of_injectives} we take submatrices at the end---this can be interpreted as taking a proper pullback of the complex. Together, this yields an interesting relation between derived functors on posets.

\begin{proposition}\label{prop:pullback_via_proper}
    Let $f:\Pi\rightarrow \Lambda$ be an order preserving map, $p:\Pi\hookrightarrow \Pi\sqcup_f\Lambda$ and $\ell:\Lambda\hookrightarrow \Pi\sqcup_f\Lambda$ be the natural inclusions to the mapping cylinder of $f$, and let $I^\bullet$ be a complex of injective sheaves on $\Lambda$. Then \[
        \Rfunc f^* I^\bullet = \Rfunc p^! \Rfunc \ell_! I^{\bullet+1},
    \]
    where $I^{\bullet+1}$ denotes the complex with $(I^{\bullet+1})^d = I^{d+1}$.
\end{proposition}

\section{Microlocal Sheaf Theory and Discrete Morse Theory}\label{sec:MorseTheory}
In this section, we establish a microlocal generalization of two classical homological results of discrete Morse theory (Theorem \ref{thm:microlocal-morse} and Theorem \ref{thm:micro-morse-ineq}). These generalizations show that the fibers---preimages of singletons---of an order preserving map between posets which lie in the discrete microsupport of a given sheaf correspond to non-isomorphisms between sheaf hypercohomology groups of adjacent sub-level or super-level sets. We begin by introducing the discrete microsupport of a complex of sheaves and illustrating some of its basic properties. 
\subsection{Discrete Microsupport}
Microsupport was introduced and studied by Sato in the context of systems of linear differential equations and smooth manifolds \cite{SatoKawaiKashiwara, Sato, KashiwaraSchapira1994}. In this setting, the microsupport of a sheaf on a manifold is a conic subset of the cotangent bundle which describes `directions' in which certain homological properties of the sheaf will propagate. This is a natural and powerful invariant to consider, because, among other things, it naturally separates a topological space into strata in which the local properties of a given sheaf are `locally constant'. Our goal here is to translate this theory to finite posets, where vectors in the cotangent bundle are replaced by poset relations $\tau<\sigma$. We think of this analogy geometrically by imagining that our poset is the face poset of a triangulation of a smooth manifold. Here the face relation $\sigma < \tau$ corresponds to a family of cotangent vectors, based at a point $x$ in the geometric realization of $\tau$, and locally pointing in the `direction' of the geometric realization of $\tau$. In fact, we hope to make this analogy concrete in future work, by explicitly describing the correspondence between the microsupport of a sheaf on a manifold and the discrete microsupport of that sheaf on the face poset of an appropriate triangulation of the manifold. 

Recall that $D^b(\Pi)$ denotes the (skeleton of) derived category whose objects are minimal injective complexes over $\Pi$---and that for any complex there exists a unique quasi-isomorphic minimal injective complex. We first define the (non-microlocal) support of a sheaf on a poset, using the tools introduced in Section~\ref{sec:pullback}, \ref{sec:proper_pullback} and \ref{sec:hypercohomology}---the pullback, proper pullback, and hypercohomology.
\begin{definition}
Let $\Pi$ be a finite poset, $I^\bullet\in D^b(\Pi)$, and $i_\tau:\{\tau\}\hookrightarrow \Pi $ be the inclusion map. The \emph{$\ast$-support} of $I^\bullet\in D^b(\Pi)$ is defined as 
\[
\supp^\ast I^\bullet \coloneqq  \setdef{\tau\in \Pi}{\mathbb{H}^j \Rfunc i_\tau^\ast I^\bullet\neq0\text{ for some }j\in\mathbb{Z}}.
\]
 The \emph{$!$-support} of $I^\bullet\in D^b(\Pi)$ is defined as
\[
\supp^! I^\bullet \coloneqq  \setdef{\tau\in \Pi}{\mathbb{H}^j \Rfunc i_\tau^! I^\bullet\neq0\text{ for some }j\in\mathbb{Z}}.
\]
\end{definition}

\begin{proposition}\label{prop:mult-in-support}
Let $I^\bullet\in D^b(\Pi)$ be a minimal injective complex, and $m_{I^\bullet}^d(\pi)$ be the multiplicity of the indecomposable injective sheaf $[\pi]$ in $I^d$. 
Then
\[
    \pi\notin \supp^!I^\bullet \text{ if and only if } m_{I^\bullet}^d(\pi)=0\text{ for all }d\in\mathbb{Z}.
\]
Moreover,
\[
    \Cl\supp^!I^\bullet =  \Cl\supp^\ast I^\bullet, 
\]
where $\Cl\supp^\ast I^\bullet$, $\Cl\supp^! I^\bullet$ denote the downwards closure of $\supp^\ast I^\bullet$, $\supp^! I^\bullet$, respectively, corresponding to closure with respect to the Alexandrov topology (Definition \ref{def:alexandrov}).
\end{proposition}
\begin{proof}

The first statement is a direct consequence of Corollary \ref{cor:proper-pull-back-multiplicity}. We begin the proof of the second statement by showing that $\supp^!I^\bullet \subseteq \Cl\supp^\ast I^\bullet$. 
Suppose that $m^d_{I^\bullet}(\pi)\neq 0$ for some $d\in \mathbb{Z}$. If $\pi\in\supp^\ast I^\bullet$, then there is nothing to prove. Assume that $\pi\notin\supp^\ast I^\bullet$. Because $m^d_{I^\bullet}(\pi)\neq 0$, $I^d(\pi)$ contains a nonzero maximal vector, which we denote by $x$. Because $\pi\notin\supp^\ast I^\bullet$, $\mathbb{H}^d(I^\bullet(\pi))=0$. Therefore, there exists a vector $v\in I^{d-1}(\pi)$ which maps onto $x$. Because $I^\bullet$ is assumed to be minimal, $v$ is not maximal. Choose $\tau\ge \pi$ such that $I^{d-1}(\pi\le\tau)(v)$ is a maximal vector. Then $n^{d-1}_{I^\bullet}(\tau)\neq 0$. If $\tau\in \supp^\ast I^\bullet$, then $\pi\in\Cl\supp^\ast I^\bullet$. Otherwise, we can apply the same argument as above to produce a new simplex $\sigma\ge \tau$ such that $n^{d-2}_{I^\bullet}(\sigma)\neq 0$. Because the complex is bounded (and the poset $\Pi$ is finite), this process must terminate with an element $\omega\ge \pi$ such that $n^{d}_{I^\bullet}(\omega)\neq 0$ and $\omega\in\supp^\ast I^\bullet$, which proves that $\pi\in\Cl\supp^\ast I^\bullet$. 

Second, we prove the inclusion $\supp^\ast I^\bullet \subseteq   \Cl\supp^! I^\bullet$. Suppose $\tau\in \supp^\ast I^\bullet$. By the definition of $i_\tau^\ast$, if $\mathbb{H}^d \Rfunc i_\tau^\ast I^\bullet\neq 0$ for some $d\in \mathbb{Z}$, then there exists $\sigma\ge \tau$ such that $m^d_{I^\bullet}(\sigma)\neq 0$. Therefore, $\sigma\in \supp^! I^\bullet$, and $\tau\in\Cl\supp^! I^\bullet$. 
\end{proof}
\begin{remark}
The support is traditionally defined to be $\Cl\supp^\ast I^\bullet$, and is therefore, by definition, closed. As a consequence, there is usually no distinction between the $!$-support and the $\ast$-support. However, for the Morse theoretic results below, we find it useful to keep track of these differences. 
\end{remark}

We now introduce a novel definition of discrete microsupport, which is a finer invariant than the support(s) defined above. The elements of a microsupport are locally closed subsets of $\Pi$ rather than its elements.
\begin{definition}\label{def:microsupport}
    Let $\Pi$ be a finite poset and $I^\bullet\in D^b(\Pi)$. For a locally closed subset $Z$ of $\Pi$, denote the inclusion map by $i_Z:Z\hookrightarrow\Pi$. The \emph{discrete $\ast$-microsupport} of $I^\bullet$ is
    \begin{align*}
        \mu\supp^\ast I^\bullet \coloneqq&\setdef{ Z\text{ locally closed in }\Pi}{ \mathbb{H}^d \left(\Rfunc {i_Z}_!\Rfunc  i_Z^\ast I^\bullet\right)\neq 0\text{ for some }d\in\mathbb{Z}},
    \end{align*}
    and the \emph{discrete $!$-microsupport} of $I^\bullet$ is
        \begin{align*}
        \mu\supp^! I^\bullet \coloneqq&\setdef{Z\text{ locally closed in }\Pi}{\mathbb{H}^d\left(\Rfunc i_Z^! I^\bullet\right)\neq 0\text{ for some }d\in\mathbb{Z}}.
        \hphantom{\Rfunc i_Z^\ast} 
        \end{align*}
\end{definition}
By Lemma \ref{lemma:exact-sequence}, we can equivalently define the discrete microsupport by
\begin{align*}
    \mu\supp^\ast I^\bullet \coloneqq&\setdef{Z\text{ loc.\! closed in }\Pi}{\mathbb{H}^d\left(\Rfunc i_{\Cl Z}^\ast I^\bullet\right)\xrightarrow{\not{\sim}}\mathbb{H}^d\left(\Rfunc i_{\Cl Z-Z}^\ast I^\bullet\right)\,\text{for some }d\in\mathbb{Z}}\text{ and } \\
    \mu\supp^! I^\bullet \coloneqq&\setdef{ Z\text{ loc.\! closed in }\Pi}{\mathbb{H}^d \left(\Rfunc {i^\ast_{\St Z}} I^\bullet\right)\xrightarrow{\not{\sim}} \mathbb{H}^d \left(\Rfunc {i^\ast_{\St Z-Z}} I^\bullet\right)\,\text{for some }d\in\mathbb{Z}}.
\end{align*}

We use the functors $\Rfunc i^\ast$, $\Rfunc i_!$, and $\Rfunc i^!$ in the above definition to highlight its computability---indeed, we can easily compute the discrete microsupports defined here with our implementation of the algorithms \cite{BrownDraganov}, as showcased in Section~\ref{sec:examples}. For an arbitrary object in the bounded derived category $D^\textrm{b}(k_\Pi)$ as defined, e.g., in \cite{KashiwaraSchapira1994}, one can define the discrete microsupport by replacing $\Rfunc i^\ast$ with $i^\ast$, $\Rfunc i_!$ with the right derived functor $R i_!$, and $\Rfunc i^!$ with the left adjoint of $R i_!$.

In a recent work~\cite{SchapiraNote}, Schapira introduces the propagation-set of a derived sheaf as a means of generalizing the microsupport to presheaves on presites and ind-sheaves. In our setting, the propagation-set admits a simple definition
\begin{align*}
    \text{Prop }I^\bullet =\{\sigma\le\tau\in \Pi \,\,|\,\,
    \text{for all } \sigma\le \gamma\le\tau,\  
    & \St\sigma - \St\gamma \not\in \mu\supp^!I^\bullet
    \text{ and } \\
    & \St\gamma - \St \tau \not\in \mu\supp^!I^\bullet
    \}.
\end{align*}
This definition suggests that a more restrictive version of the discrete microsupport might have more desirable theoretical properties (at the expense of increased computational complexity). The more restrictive definition would consider a locally closed set $Z\not\in\mu\supp^! I^\bullet$ if and only if $\mathbb{H}^d \left(\Rfunc {i^!_{W}} I^\bullet\right)=0$ for each locally closed subset $W$ of $Z$ (similarly for $\mu\supp^\ast I^\bullet$). However, we choose to continue with the less restrictive definition of discrete microsupport, which is sufficient for our results on computational sheaf theory and discrete microlocal Morse theory. 
\begin{remark}
 It is immediate that $\pi\in\supp^! I^\bullet$ if and only if $\{\pi\}\in\mu\supp^! I^\bullet$. Therefore, the discrete $!$-microsupport generalizes the $!$-support in an obvious way. However, the discrete $\ast$-microsupport is strictly a generalization of the $\ast$-support only if $\Pi$ has the following property: for each $\pi\in \Pi$, if the boundary, $\Cl\pi-\pi$, of $\pi$ is non-empty, then the hypercohomology of the injective resolution of the constant sheaf on the boundary is non-zero. This happens, for example, whenever $\Pi$ is a simplicial complex. In general, $ \{\pi\in\Pi:\{\pi\}\in\mu\supp^\ast I^\bullet\}\subset \supp^\ast I^\bullet$ (for an example, consider the constant sheaf on a totally ordered poset with two elements). 
\end{remark}

\medskip
The first result we present ties together $!$-microsupport with linear maps in cohomology sheaves (recall Definition~\ref{def:cohomology_sheaf}).

\begin{theorem}
If $\St\sigma - \St\tau\notin \mu\supp^! I^\bullet$, for $\sigma\le\tau\in \Pi$, then the linear maps $H^d(I^\bullet)
(\sigma\le\tau)$ are isomorphisms for each $d\in \mathbb{Z}$. 
\end{theorem}
\begin{proof}
Let $\sigma\le\tau\in\Pi$.
 The linear map $H^d(I^\bullet)(\sigma\le\tau)$ is, by definition,
 \[
 H^d(I^\bullet)(\sigma\le\tau):\mathbb{H}^d(i_\sigma^*I^\bullet)\rightarrow  \mathbb{H}^d(i_\tau^*I^\bullet), 
 \]
 where $i_\gamma:\St\gamma\hookrightarrow \Pi$. Because $\St\tau$ is an open subset of $\St\sigma$ for $\sigma\le \tau$, we can apply Lemma \ref{lemma:exact-sequence} to get a long exact sequence
 \begin{align*}
\cdots\rightarrow& \mathbb{H}^d(\Rfunc i_Z^! I^\bullet)\rightarrow \mathbb{H}^d( i_\sigma^\ast I^\bullet)\rightarrow
\mathbb{H}^d(\Rfunc i_\tau^\ast  I^\bullet)\\
&\rightarrow \mathbb{H}^{d+1}(\Rfunc i_Z^!I^\bullet)\rightarrow\cdots
\end{align*}
where $Z=\St\sigma-\St\tau$. Because $Z\notin \mu\supp^!I^\bullet$, $\mathbb{H}^d(\Rfunc i_Z^! I^\bullet)=0$ for all $d\in\mathbb{Z}$. Therefore, $ H^d(I^\bullet)(\sigma\le\tau)$ is an isomorphism for each $d\in\mathbb{Z}$. 
\end{proof}
An immediate consequence of the above theorem is the following:
\begin{corollary}\label{cor:locally_constant_no_st_diff_in_microsupport}
    If $\mu\supp^! I^\bullet$ contains no sets of the form $\St\sigma-\St\tau$, then the cohomology sheaves $H^d(I^\bullet)$ are locally constant.
\end{corollary}
This shows that the discrete mircosupport can be used to separate our poset into `strata' over which cohomology sheaves are locally constant. We think a natural follow-up to this paper could use this fact to study sheaf theoretic stratification algorithms from the derived perspective, generalizing/building on prior work of Nanda and Brown-Wang \cite{Nanda2019, BrownWang}. 

\subsection{The Discrete Microlocal Morse Theorem and Inequalities}

Discrete Morse theory identifies topologically relevant parts of discrete objects and lets us skip parts that are ``essentially contractible''. A filtration function on a simplicial complex, $f:\Sigma\rightarrow\R$, is called \emph{generalized discrete Morse} if it is order preserving and every non-empty level set, $f^{-1}(r)$, is an interval---a set of the form $[\mu, \tau] = \setdef{\sigma\in\Sigma}{\mu \leq \sigma \leq \tau}$; see, e.g., \cite{Freij2009, BauerEdelsbrunner2016}. The singletons in the level set decomposition of $\Sigma$ are called \emph{critical simplices}, and their images \emph{critical values}. The critical values identify changes in cohomology.
Namely, comparing growing sublevel sets $f^{-1}((-\infty,s])$ and $f^{-1}((-\infty,r])$, their cohomology (even homotopy type) is the same if the interval $(s,r]$ contains no critical values. Conversely, it always differs if there is exactly one critical value in $(s, r]$.
The theorems below generalize the homological claim. Both $\Sigma$ and $\R$ are replaced by general (finite) posets. Intervals are replaced by locally closed sets, which allows us to consider any order preserving map $f$. This relaxation is counterbalanced by a more complicated structure of critical elements.

The results of this section are discrete analogues of work by Kashiwara~\cite{Kashiwara1983}, Schapira--Tose~\cite{SchapiraTose}, and Witten~\cite{Witten}. In fact, the proofs of Theorem \ref{thm:microlocal-morse} and Theorem \ref{thm:micro-morse-ineq} closely follow the main ideas presented in \cite{Kashiwara1983} and \cite{SchapiraTose}.

\begin{definition}
Suppose that $f:\Lambda\rightarrow\Pi$
is an order preserving map of finite posets. Let $I^\bullet\in D^b(\Lambda)$. An element $\pi\in\Pi$ is
\begin{itemize}
    \item \emph{$(I^\bullet, f,\ast)$-critical} if $f^{-1}(\pi)\in \mu\supp^\ast I^\bullet$,
    \item \emph{$(I^\bullet, f,!)$-critical} if $f^{-1}(\pi)\in \mu\supp^! I^\bullet$. 
\end{itemize}
\end{definition}

\begin{theorem}[Discrete Microlocal Morse Theorem]\label{thm:microlocal-morse}
Let $f:\Lambda\rightarrow\Pi$ be an order preserving map of finite posets. 
For $\pi\in \Pi$, let \[i_{\le \pi}:\{\sigma\in\Lambda:f(\sigma)\le \pi\}\hookrightarrow \Lambda, \quad i_{\ge \pi}:\{\sigma\in\Lambda:f(\sigma)\ge \pi\}\hookrightarrow \Lambda\] be inclusion maps of sublevel-sets and superlevel-sets of $f$. For $a\le b\in\Pi$, let \[(a,b]=\{z\in\Pi:a< z\le b\}.\]
For $I^\bullet\in D^b(\Lambda)$, if $(a,b]$ contains no $(I^\bullet,f,!)$-critical elements, then the morphisms (induced by inclusions)
\begin{align}
\mathbb{H}^d\left( \Rfunc i_{\le b}^!I^\bullet \right)&\rightarrow \mathbb{H}^d \left(\Rfunc  i_{\le a}^!I^\bullet \right)\text{ and }\label{eqn:first-iso}\\
\mathbb{H}^d\left(\Rfunc i_{\ge a}^\ast I^\bullet \right)&\rightarrow \mathbb{H}^d \left(\Rfunc i_{\ge b}^\ast I^\bullet \right)\label{eqn:second-iso}
\end{align}
are isomorphisms for each $d\in\mathbb{Z}$.  If $(a,b]$ contains no $(I^\bullet,f,\ast)$-critical elements, then the morphisms
\begin{align}
\mathbb{H}^d\left( \Rfunc i_{\le b}^\ast I^\bullet \right)&\rightarrow \mathbb{H}^d \left(\Rfunc i_{\le a}^\ast I^\bullet \right)\label{eqn:third-iso}
\end{align}
are isomorphisms for each $d\in\mathbb{Z}$.
\end{theorem}
\begin{proof}
We begin by describing each of the morphisms of the theorem in more detail. By Proposition~\ref{lem:computation_proper_pullback}, computing $\mathbb{H}^d(\Rfunc i_{\le b}^! I^\bullet)$ is quite simple. First, we consider the complex of labeled submatrices of $I^\bullet$ with labels in the set $\{\sigma\in\Lambda:f(\sigma)\le b\}$. Then we forget the labels and take the $d$-cohomology group of the complex of vector spaces these matrices represent. In this basis, morphism~\eqref{eqn:first-iso} is the map induced on cohomology groups by the projection map from the complex of injective subsheaves of $I^\bullet$ generated by maximal vectors which are supported in the set $\{\sigma\in\Lambda:f(\sigma)\le b\}$ to those with maximal vectors supported in the set $\{\sigma\in\Lambda:f(\sigma)\le a\}$. Morphism~\eqref{eqn:second-iso} is similarly induced by projection maps. Morphism~\eqref{eqn:third-iso} is slightly more involved. First, let $I^\bullet\vert_\pi$ be the restriction of $I^\bullet$ to $\{\sigma\in\Lambda:f(\sigma)\le \pi\}$. (Note that $I^\bullet\vert_\pi$ is not necessarily a complex of injective sheaves). There is a morphism of complexes of sheaves from $I^\bullet \vert_b$ to $I^\bullet \vert_a$ (where both complexes are viewed as complexes of sheaves on $\Lambda$ by extending sheaves by zero outside of their support) which is the identity map on the set $\{\sigma\in\Lambda:f(\sigma)\le a\}$ and zero elsewhere. This morphism of complexes of sheaves induces a morphism between the corresponding injective resolutions, which in turn induces morphism~\eqref{eqn:third-iso}. 

We prove that \eqref{eqn:first-iso}, \eqref{eqn:second-iso}, and \eqref{eqn:third-iso} are isomorphisms by realizing each as a composition isomorphisms. Specifically, there exists $\{a_i\}_{i=0}^k\subset \Pi$ such that $a=a_0 < a_1 < a_2 < \cdots < a_k = b$ and $(a_i,a_{i+1}]=\{a_{i+1}\}$. The morphism~\eqref{eqn:first-iso} is equal to the composition of maps of the form
\[
\mathbb{H}^d\left( \Rfunc i_{\le a_{i+1}}^!I^\bullet \right)\rightarrow \mathbb{H}^d \left(\Rfunc  i_{\le a_i}^!I^\bullet \right),
\]
because each is induced by a projection map described in the paragraph above. (A similar argument applies to \eqref{eqn:second-iso} and \eqref{eqn:third-iso}). We proceed by proving that for each pair $a_i< a_{i+1}$, the maps described above are isomorphisms. To make the notation easier to follow, we simply let $a=a_i$ and $b = a_{i+1}$ for the remainder of the proof.

We begin with equation \eqref{eqn:first-iso}. Because $f^{-1}_{\le a}$ is a closed subset of $f^{-1}_{\le b}$, by Lemma \ref{lemma:exact-sequence}, and the fact that $\Rfunc i^\ast_{\le a}\Rfunc i^\ast_{\le b}I^\bullet = \Rfunc i^\ast_{\le a}I^\bullet$, we have a long exact sequence of cohomology groups 
\begin{align*}
\cdots\rightarrow& \mathbb{H}^d\left(\Rfunc i_{\le a}^!I^\bullet\right)\rightarrow \mathbb{H}^d\left(\Rfunc i_{\le b}^!I^\bullet\right)\rightarrow
\mathbb{H}^d\left(\Rfunc i_{f^{-1}(b)}^! I^\bullet\right)\\
&\rightarrow \mathbb{H}^{d+1}\left(\Rfunc i_{\le a}^!I^\bullet\right)\rightarrow\cdots
\end{align*}
By the definition of $(I^\bullet,f,!)$-critical elements and $\mu\supp^!I^\bullet$, $\mathbb{H}^d\left(\Rfunc i_{f^{-1}(b)}^!I^\bullet\right)=0$. Therefore, 
\[\mathbb{H}^d\left(\Rfunc i_{\le a}^!I^\bullet\right)\rightarrow \mathbb{H}^d\left(\Rfunc i_{\le b}^!I^\bullet\right)\]
is an isomorphism for each $d\in\mathbb{Z}$. 

We now turn to equation \eqref{eqn:second-iso}. Because $f^{-1}_{\ge a}$ is an open subset of $\Lambda$, 
\[\Rfunc i_{f^{-1}(b)}^!\Rfunc i_{\ge a}^\ast I^\bullet = \Rfunc i_{f^{-1}(b)}^!\Rfunc i_{\ge a}^! I^\bullet =\Rfunc i_{f^{-1}(b)}^!I^\bullet, \] 
where $i_{f^{-1}(b)}:f^{-1}(b)\hookrightarrow f^{-1}_{\ge a}$. Moreover, because $f^{-1}_{\ge b}$ is an open subset of $f^{-1}_{\ge a}$, Lemma \ref{lemma:exact-sequence} yields 
\begin{align*}
\cdots\rightarrow& \mathbb{H}^d\left(\Rfunc i_{f^{-1}(b)}^!I^\bullet \right)\rightarrow \mathbb{H}^d\left(\Rfunc i_{\ge a}^\ast I^\bullet\right)\rightarrow
\mathbb{H}^d\left(\Rfunc i_{\ge b}^\ast I^\bullet\right)\\
&\rightarrow \mathbb{H}^{d+1}\left(\Rfunc i_{f^{-1}(b)}^!I^\bullet\right)\rightarrow\cdots
\end{align*}
Again, by the definition of $(I^\bullet,f,!)$-critical elements and $\mu\supp^!I^\bullet$, $\mathbb{H}^d\left(\Rfunc i_{f^{-1}(b)}^!I^\bullet\right)=0$, and equation \eqref{eqn:second-iso} follows. 

Finally, we turn to equation \eqref{eqn:third-iso}. Because $f^{-1}(b)$ is open in $f^{-1}_{\le b}$, Lemma \ref{lemma:exact-sequence} gives us the long exact sequence
\begin{align}
\cdots\rightarrow& \mathbb{H}^d(\Rfunc i_{f^{-1}(b)!}\Rfunc i_{f^{-1}(b)}^\ast I^\bullet)\rightarrow \mathbb{H}^d(\Rfunc i^\ast_{\le b} I^\bullet)\rightarrow
\mathbb{H}^d(\Rfunc i_{\le a}^\ast I^\bullet)\\
&\rightarrow \mathbb{H}^{d+1}(\Rfunc i_{f^{-1}(b)!}\Rfunc i_{f^{-1}(b)}^\ast I^\bullet)\rightarrow\cdots\nonumber
\end{align}
where $i_{f^{-1}(b)}:f^{-1}(b)\hookrightarrow f^{-1}_{\le b}$. By the definition of $(I^\bullet,f,\ast)$-critical elements and $\mu\supp^\ast I^\bullet$, \[\mathbb{H}^d(\Rfunc j_{f^{-1}(b)!}\Rfunc j_{f^{-1}(b)}^\ast I^\bullet)=0,\] and equation \eqref{eqn:third-iso} follows. 
\end{proof}

\begin{theorem}[Discrete Microlocal Morse Inequality]\label{thm:micro-morse-ineq}
    Let $\kappa^!\subset \Pi$ (and $\kappa^\ast \subset \Pi$) denote the set of $(I^\bullet,f,!)$-critical elements in $\Pi$ (and $(I^\bullet,f,\ast)$-critical elements, respectively). For each $\ell\in\mathbb{Z}$,
    \begin{align}
        (-1)^\ell \sum_{j\le\ell} (-1)^j\dim \mathbb{H}^j(I^\bullet)&\le (-1)^\ell \sum_{\pi\in \kappa^!}\sum_{j\le\ell} (-1)^j\dim \mathbb{H}^j\left(\Rfunc i_{f^{-1}(\pi)}^! I^\bullet\right),\text{ and}\\
          (-1)^\ell \sum_{j\le\ell} (-1)^j\dim \mathbb{H}^j(I^\bullet)&\le (-1)^\ell \sum_{\pi\in \kappa^\ast}\sum_{j\le\ell} (-1)^j\dim \mathbb{H}^j\left(\Rfunc i_{f^{-1}(\pi)!}\Rfunc i_{f^{-1}(\pi)}^\ast I^\bullet\right).
    \end{align}
    Moreover, the inequalities become equalities when we some over all $j\in\Z$: 
    \begin{align}
         \chi I^\bullet&= \sum_{\pi\in\kappa^!} \chi \Rfunc i^!_{f^{-1}(\pi)} I^\bullet= \sum_{\pi\in\kappa^\ast} \chi \Rfunc i_{f^{-1}(\pi)!}\Rfunc i^\ast_{f^{-1}(\pi)} I^\bullet.
    \end{align}
\end{theorem}
\begin{proof}
    Suppose 
    \[
    \cdots \rightarrow V_C^{d-1}\rightarrow V_A^{d}\rightarrow V^d_B\rightarrow V^d_C\rightarrow V_A^{d+1}\rightarrow V_B^{d+1}\rightarrow V_C^{d+1}\rightarrow V_A^{d+2}\rightarrow  \cdots
    \]
    is a long exact sequence of vector spaces. Because the sequence is exact, 
    \begin{align*}
        \sum_{j\in\mathbb{Z}}(-1)^j (\dim V_A^j-\dim V_B^j+\dim V_C^j)=0.
    \end{align*}
    Moreover, the sequence 
    \[
    \cdots \rightarrow V_C^{\ell-1}\rightarrow V_A^{\ell}\rightarrow V^\ell_B\xrightarrow{\delta} V^\ell_C\rightarrow \coker\delta \rightarrow 0\rightarrow  \cdots
    \]
    is exact, and therefore
    \begin{align*}
        (-1)^{\ell}\sum_{j\le \ell}(-1)^j (\dim V_A^j-  \dim V_B^j+  \dim V_C^j)= \dim \coker\delta\ge 0. 
    \end{align*}
    Rearranging the above equations gives 
    \begin{align*}
        (-1)^\ell \sum_{j\le \ell}(-1)^j \dim V_B^j&\le (-1)^\ell  \sum_{j\le \ell}(-1)^j (\dim V_A^j+  \dim V_C^j), \text{ and }\\
        \sum_{j\in\mathbb{Z}}(-1)^j \dim V_B^j&=  \sum_{j\in\mathbb{Z}}(-1)^j (\dim V_A^j+  \dim V_C^j).
    \end{align*}
    Substituting the long exact sequences of Lemma \ref{lemma:exact-sequence} to the above inequalities,  we have, for each $a\le b \in \Pi$ such that $(a,b]=\{b\}$,
    \begin{align*}
        (-1)^\ell \sum_{j\le \ell }&(-1)^j \dim \mathbb{H}^j(\Rfunc i^!_{\le b} I^\bullet) \le(-1)^\ell
        \sum_{j\le \ell}(-1)^j (\dim \mathbb{H}^j(\Rfunc i^!_{\le a} I^\bullet)+  \dim \mathbb{H}^j(\Rfunc i^\ast_{f^{-1}( b)} I^\bullet) ),\\
        \text{ and } \sum_{j\in\mathbb{Z} }&(-1)^j \dim \mathbb{H}^j(\Rfunc i^!_{\le b} I^\bullet) = \sum_{j\in\mathbb{Z}}(-1)^j (\dim \mathbb{H}^j(\Rfunc i^!_{\le a} I^\bullet)+  \dim \mathbb{H}^j(\Rfunc i^\ast_{f^{-1}( b)} I^\bullet) ).
    \end{align*}
    If $b\in \Pi$ is not $(I^\bullet,f,!)$-critical, then $ \dim \mathbb{H}^j(\Rfunc i^\ast_{f^{-1}( b)} I^\bullet) )=0$, which implies 
    \begin{align*}
        (-1)^\ell \sum_{j\le \ell}(-1)^j \dim \mathbb{H}^j(\Rfunc i^!_{\le b} I^\bullet) &\le (-1)^\ell 
        \sum_{j\le \ell}(-1)^j \dim \mathbb{H}^j(\Rfunc i^!_{\le a} I^\bullet),\text{ and }\\
         \sum_{j\in\mathbb{Z}}(-1)^j \dim \mathbb{H}^j(\Rfunc i^!_{\le b} I^\bullet) &= 
        \sum_{j\in\mathbb{Z}}(-1)^j \dim \mathbb{H}^j(\Rfunc i^!_{\le a} I^\bullet).
    \end{align*}
    By choosing a total order on $\Pi$ which respects the given partial order, and inductively applying the previous inequalities, we get 
    \begin{align*}
        (-1)^\ell \sum_{j\le \ell}(-1)^j \dim \mathbb{H}^j( I^\bullet)& \le 
        \sum_{\pi \in \kappa^!} (-1)^\ell \sum_{j\le \ell}(-1)^j  \dim \mathbb{H}^j(\Rfunc i^\ast_{f^{-1}( \pi)} I^\bullet), \text{ and }\\
        \chi I^\bullet&= \sum_{\pi\in\kappa^!}\chi \Rfunc i^\ast_{f^{-1}(\pi)}I^\bullet.
    \end{align*}
    The remaining (in)equalities follow from similar arguments. 
\end{proof}

\subsection{Reduction to Classical Discrete Morse Theory}

We will conclude this section by illustrating how the prior results generalize classical discrete Morse theory. 
Let $\Sigma$ be a finite simplicial complex, and $I_{k_\Sigma}^\bullet\in D^b(\Sigma)$ the minimal injective resolution of the constant sheaf on $\Sigma$. Let $f:\Sigma\rightarrow \R$ be a generalized discrete Morse function, i.e., an order preserving map such that each level set, $f^{-1}(x)$, is an interval $[\mu, \tau] = \setdef{\sigma\in\Sigma}{\mu \leq \sigma \leq \tau}$.
For the codomain to be a finite poset, replace $\R$ by the image of $f$ denoted by $\Pi$. We have a map $f:\Sigma\rightarrow\Pi$. Let $f^{-1}_c \coloneqq \{\sigma\in\Sigma:f(\sigma)\leq c\}$ be the sublevel set of $f$ for any $c\in \Pi$.


\begin{proposition}
The element $c\in\Pi$ is $(I_{k_\Sigma}^\bullet,f,\ast)$-critical if and only if $f^{-1}(c)$ is a singleton. 
\end{proposition}
\begin{proof}
 By assumption $f^{-1}(c)=[\mu, \tau]$. Let $\Cl \tau$ denote the closure of $\tau$ in $\Sigma$. Because $I_{k_\Sigma}^\bullet$ is an injective resolution of the constant sheaf on $\Sigma$, the pullback $ \Rfunc i_{f^{-1}(c)}^\ast I_{k_\Sigma}^\bullet$ is an injective resolution of the constant sheaf on $f^{-1}(c)$ and we have the following isomorphisms \[
    \mathbb{H}^d\Rfunc (i_{f^{-1}(c)})_!\Rfunc i_{f^{-1}(c)}^\ast I_{k_\Sigma}^\bullet
    \cong H_c^d([\mu, \tau];k)
    \cong H_d(\Cl \tau,\Cl \tau-\St \mu;k)\]
between hypercohomology, compactly supported cohomology, and relative homology for each~$d$. Because $\Cl \tau$ is a closed simplex, the hypercohomology groups vanish if and only if $\mu\neq \tau$.
\end{proof}

The right derived pullback $\Rfunc i^\ast_{{c}}I_{k_\Sigma}^\bullet$ of the injective resolution of the constant sheaf on $\Sigma$ by the inclusion map $i_{{c}}:f^{-1}_{{c}}\hookrightarrow \Sigma$, is an injective resolution of the constant sheaf on $f^{-1}_{c}$, and the hypercohomology groups are isomorphic to the usual simplicial cohomology groups:
\[
\mathbb{H}^d\left(\Rfunc i^\ast_{{ c}}I_{k_\Sigma}^\bullet\right)
=
H^d\left(\Rfunc p_\ast \Rfunc i^\ast_{{ c}}I_{k_\Sigma}^\bullet\right)
\cong
H^d\left(f^{-1}_{c};k\right).
\]

Therefore, in this context, Theorem \ref{thm:microlocal-morse} states that if $ \#f^{-1}(b)>1$ for each element $b\in \Pi$ with $a < b \leq c$, then the natural morphisms between simplicial cohomology groups
\begin{align*}
H^d\left(f^{-1}_{c};k\right)&\rightarrow H^j \left(f^{-1}_{a};k \right)
\end{align*}
are isomorphisms for each $d\in\mathbb{Z}$.

\section{Examples}\label{sec:examples}

In this section we compute simple examples which illustrate the benefits of working with the derived and microlocal perspectives. We will compute three different complexes of sheaves, each obtained via a triangulated (simplicial) map, $g$, $h$, and $l$, from $S^2\lor S^1$ and $D^2$ to $S^2$. We will then illustrate the discrete microlocal Morse theorem for each of these complexes of sheaves, relative to a discrete Morse function, $f$, on the triangulation poset $\Lambda$ of $S^2$. All computations in this section are done over $k=\Z_2$.

We consider triangulations of two maps on the wedge sum of a sphere and a circle, with values in a sphere. Both maps send $S^2$ identically to $S^2$. The first map, $g$, sends the circle to the point where it touches the sphere, and $h$ sends it to the equator. We start by triangulating the spaces as shown in Figure~\ref{fig:ex_wedge-circle-sphere}. This gives us two posets, and two order preserving maps, $g,h: \Sigma \rightarrow \Lambda$. Both maps are determined by mapping vertices. Both send $0,1,2,3,4$ identically, and then $g$ sends $5,6 \overset{g}{\mapsto} 4$, whereas $h$ sends $5\overset{h}{\mapsto} 1$ and $6\overset{h}{\mapsto} 3$. The third map we consider, $l$, from the closed disk to the sphere, is obtained by identifying all points along the boundary of the closed disk. We triangulate the closed disk according to Figure~\ref{fig:ex_wedge-circle-sphere}. Then, the map $l$ is the identity map on the vertices 0, 1, 2, 3, and 4, and $5,6 \overset{l}{\mapsto} 4$. 

\begin{figure}[htb]
    \begin{subfigure}{.95\textwidth}
    \centering
    \includegraphics[width=.85\textwidth]{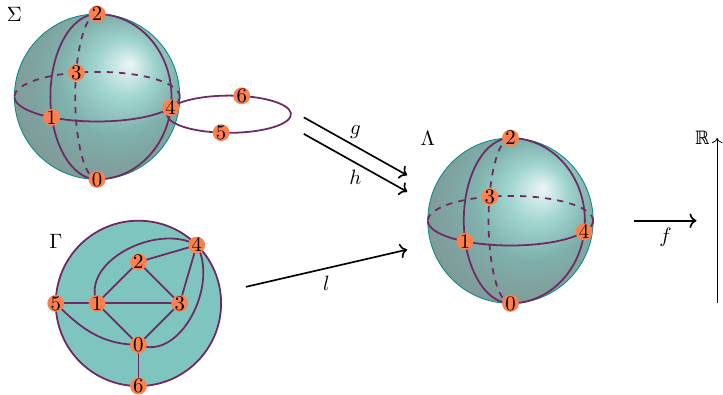}

    \caption{}
    \label{fig:ex_wedge-circle-sphere_geom}
    
    \end{subfigure}
    \vspace{3mm}
    
    \begin{subfigure}{.95\textwidth}
        \centering
        \footnotesize
        \includegraphics[width=.55\textwidth]{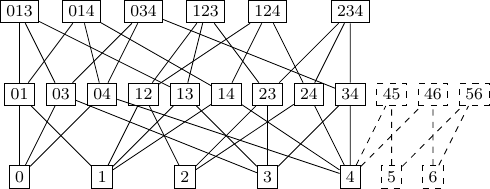}

        \caption{}
        \label{fig:ex_wedge-circle-sphere_poset}
    \end{subfigure}
    
        \caption{(a) Triangulations of $S^2\lor S^1$, the closed disk $D^2$, and $S^2$. The face relations give rise to the posets $\Sigma$, $\Gamma$, and $\Lambda$, respectively. (b) The poset $\Sigma$ when both solid and dashed lines are considered, the poset $\Lambda$ when only solid lines are considered.}
        \label{fig:ex_wedge-circle-sphere}
    \end{figure}
     
    Following the approach of traditional discrete Morse theory, we choose a partition of $\Lambda$ into intervals, and choose a total order (which we denote by $\preceq$) on the quotient poset $\Pi$ by taking the alphabetical order of the labelling in Figure~\ref{fig:Hasse-diagrams}.  When we view $\Pi$ as a totally ordered set, we interpret the corresponding quotient map $f:\Lambda\rightarrow \Pi$ as a discrete Morse function, i.e., an order preserving map from $\Lambda$ to the real line, and denote by $f^{-1}_{\preceq X}$ the sublevel set $\setdef{\lambda\in\Lambda}{f(\lambda) \preceq X}$.
    
    \begin{figure}[H]
    \centering
    \begin{subfigure}{.47\textwidth}
        \small
        \includegraphics[width=.85\textwidth]{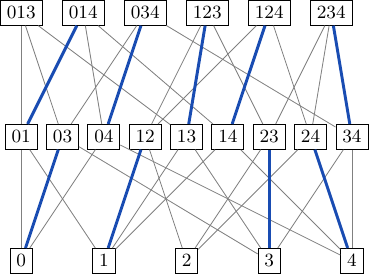}
        \label{fig:Sphere-Hasse}
        \caption{}
    \end{subfigure}\hfill
    \begin{subfigure}{.47\textwidth}
        \footnotesize
        \includegraphics[width=.85\textwidth]{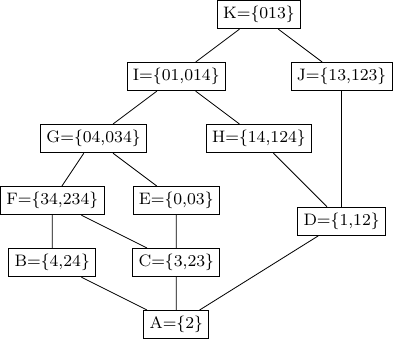}
        \label{fig:Quotient-Hasse}
        \caption{}
    \end{subfigure}
    \caption{(a) the Hasse diagram of $\Lambda$, with partition by locally closed sets. Each locally closed set has cardinality two, and is illustrated by a blue edge connecting its elements. (b) the Hasse diagram of the corresponding quotient poset $\Pi$. }
    \label{fig:Hasse-diagrams}
\end{figure}

\begin{figure}[H]
    \centering
    \begin{subfigure}{.24\textwidth}
        \centering
        \includegraphics[width=.8\textwidth]{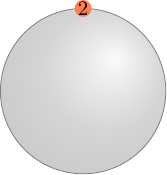}
            
        \caption{}
    \end{subfigure}
    \begin{subfigure}{.24\textwidth}
        \centering
        \includegraphics[width=.8\textwidth]{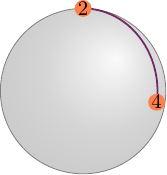}
        
        \caption{}
    \end{subfigure}
    \begin{subfigure}{.24\textwidth}
        \centering
        \includegraphics[width=.8\textwidth]{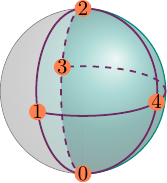}
        \caption{}
    \end{subfigure}
    \begin{subfigure}{.24\textwidth}
        \centering
        \includegraphics[width=.8\textwidth]{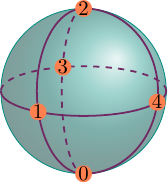}
        \caption{}
    \end{subfigure}
    \caption{(a) sublevel set $f^{-1}_{\preceq A}$. (b) sublevel set $f^{-1}_{\preceq B}$. (c) sublevel set $f^{-1}_{\preceq I}$. (d) sublevel set $f^{-1}_{\preceq K}$.}
    \label{fig:sublevelsets}
\end{figure}

From now on we work with the posets, and we only refer back to the original spaces when interpreting the results. We denote the elements in the posets by the string of vertices that define the given simplex, e.g., $014$ is the triangle with edges $01$, $04$, $14$ and vertices $0$, $1$, $4$. The notation does not distinguish whether we talk about elements in $\Sigma$, $\Gamma$, or $\Lambda$; in cases where the distinction is not clear from the context, we use subscripts, e.g., $014_\Sigma\in\Sigma$ and $014_\Lambda\in\Lambda$.

We start by computing injective resolutions of the constant sheaves $k_\Sigma$ and $k_\Gamma$, which, in an abuse of notation, we denote again by $k_\Sigma$ and $k_\Gamma$. Then we compute $\Rfunc g_\ast k_\Sigma$, $\Rfunc h_\ast k_\Sigma$, and $\Rfunc l_\ast k_\Gamma$. We give a visualization of several intervals in $\mu\supp^!\Rfunc h_\ast k_\Sigma$ and $\mu\supp^\ast \Rfunc h_\ast k_\Sigma$. We compute various sets of critical elements of the Morse function $f$ (relative to each of the above complexes of sheaves), and finally compute Betti numbers for hypercohomology groups corresonding to filtrations induced by $f$. 

\paragraph{The injective resolution of $k_\Sigma$} We compute the minimal injective resolution of $k_\Sigma$ iteratively using Algorithm~\ref{algo:injective_resolution_tail}, starting with a labeled matrix with one column labeled by a virtual element larger than everything in $\Sigma$, nine rows labeled by the maximal elements of $\Sigma$, and all entries $1$ (which will, however, \emph{not} be a part of the complex). Alternatively, we can view the same computation as Algorithm~\ref{algo:pullback_of_complex_of_injectives} computing the pullback $\Rfunc p^\ast(\dots\rightarrow 0 \rightarrow k \rightarrow 0 \rightarrow \dots)$ with respect to the constant map $p: \Sigma\rightarrow \mathrm{pt}$. Either way this yields a minimal complex of injective sheaves \[
   \cdots\rightarrow 0\rightarrow I^0 \xrightarrow{\eta^0} I^1 \xrightarrow{\eta^1} I^2\rightarrow 0 \rightarrow \cdots
\]
with the maps represented by the following two labeled matrices:

\noindent\begin{minipage}{\textwidth}
    \centering
    \includegraphics[width=140mm]{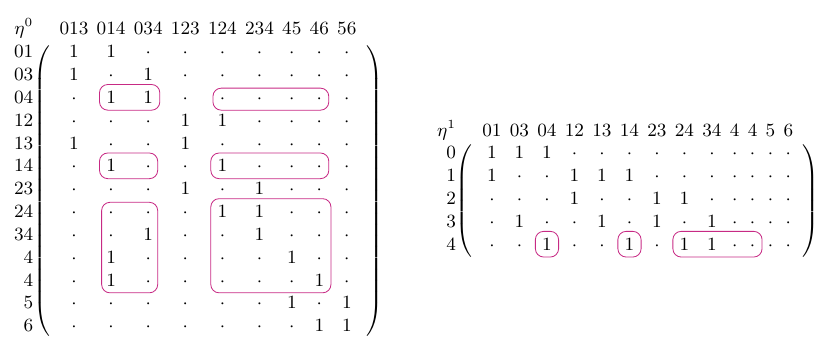}
\end{minipage}

Notice first that the complex is, indeed, minimal -- the only non-empty matrix of the form $\eta^d[\sigma,\sigma]$ is $\eta^1[4,4]=0$. We highlighted the submatrices $\eta^0(4)=\eta^0[\St 4, \St 4]$, as the vertex $4$ is a particularly interesting part of the example.

\paragraph{The derived pushforward $\Rfunc g_\ast k_\Sigma$} To compute $\Rfunc g_\ast k_\Sigma$, we only need to relabel the matrices and minimize (see Section~\ref{sec:pushforward_algo}). To relabel $\eta^0$, $\eta^1$ according to $g$, we need to perform the following changes:
\begin{itemize}
    \item The columns $45$, $46$, $56$ in $\eta^0$ are all relabeled to $4$,
    \item the rows $5$, $6$ in $\eta^0$ are both relabeled to $4$,
    \item the columns $5$, $6$ in $\eta^1$ are both relabeled to $4$.
\end{itemize}
For a moment, we denote the new labeled matrices (which do not necessarily represent a minimal complex) by $\tilde\eta^d$. We perform Algorithm~\ref{algo:peeling} on the matrices $\tilde\eta^d$ to minimize the complex. The above changes created a new non-empty $4$-block within $\tilde\eta^0$. We still have $\tilde{\eta}^1[4,4]=0$, but $\tilde\eta^0[4,4]$ is now a non-trivial matrix. We reduce $\tilde\eta^0[4,4]$, which is the $4\times 3$ submatrix in the bottom-right corner of $\tilde\eta^0$: \[
    \includegraphics[width=70mm]{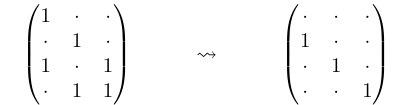}
\]
For the last three rows that are non-zero in the reduced submatrix, we perform further column operations to clear out all remaining places with $1$ in the whole matrix. Note that even though it is not necessary, the first of the four rows labeled by $4$ is now also zero, because the operations used to reduce the submatrix were performed on the whole matrix. We should now perform corresponding operations also on $\tilde\eta^{-1}$ and $\tilde\eta^{1}$. However, $\tilde\eta^{-1}=0$, and all column operations on $\tilde\eta^{1}$ only act on a zero submatrix, so nothing changes. Finally, we remove the last three columns and last three rows in $\tilde\eta^0$, and the last three columns in $\tilde\eta^1$. The derived pushforward $\Rfunc g_\ast k_\Sigma$ is represented by the resulting labelled matrices.\\

\begin{center}
\scalebox{.9}{
\noindent\begin{minipage}{.47\textwidth}
    \small
    {\setlength{\arraycolsep}{2pt}
        \[
        \begin{pNiceArray}[first-row,first-col, margin=.5em]{cccccc}
            \Rfunc g_\ast\eta^0 & 013   & 014   & 034   & 123   & 124   & 234   \\
            01            & 1     & 1     & \cdot & \cdot & \cdot & \cdot \\
            03            & 1     & \cdot & 1     & \cdot & \cdot & \cdot \\
            04            & \cdot & 1     & 1     & \cdot & \cdot & \cdot \\
            12            & \cdot & \cdot & \cdot & 1     & 1     & \cdot \\
            13            & 1     & \cdot & \cdot & 1     & \cdot & \cdot \\
            14            & \cdot & 1     & \cdot & \cdot & 1     & \cdot \\
            23            & \cdot & \cdot & \cdot & 1     & \cdot & 1     \\
            24            & \cdot & \cdot & \cdot & \cdot & 1     & 1     \\
            34            & \cdot & \cdot & 1     & \cdot & \cdot & 1     \\
            4             & \cdot & \cdot & \cdot & \cdot & \cdot & \cdot 
        \CodeAfter
            \begin{tikzpicture}
                \node () [fit=(3-2) (3-3), draw=magenta!80!black, rounded corners] {};
                \node () [fit=(3-5) (3-6), draw=magenta!80!black, rounded corners] {};
                \node () [fit=(6-2) (6-3), draw=magenta!80!black, rounded corners] {};
                \node () [fit=(6-5) (6-6), draw=magenta!80!black, rounded corners] {};
                \node () [fit=(8-2) (10-3), draw=magenta!80!black, rounded corners] {};
                \node () [fit=(8-5) (10-6), draw=magenta!80!black, rounded corners] {};
            \end{tikzpicture}
        \end{pNiceArray}
        \]
    }
    \hspace{3mm}
\end{minipage}\hfill%
\begin{minipage}{.47\textwidth}
    \small
    {\setlength{\arraycolsep}{2pt}
        \[
        \begin{pNiceArray}[first-row,first-col, margin=.5em]{cccccccccc}
            \Rfunc g_\ast\eta^1 & 01    & 03    & 04    & 12    & 13    & 14    & 23    & 24    & 34    & 4     \\
            0             & 1     & 1     & 1     & \cdot & \cdot & \cdot & \cdot & \cdot & \cdot & \cdot \\
            1             & 1     & \cdot & \cdot & 1     & 1     & 1     & \cdot & \cdot & \cdot & \cdot \\
            2             & \cdot & \cdot & \cdot & 1     & \cdot & \cdot & 1     & 1     & \cdot & \cdot \\
            3             & \cdot & 1     & \cdot & \cdot & 1     & \cdot & 1     & \cdot & 1     & \cdot \\
            4             & \cdot & \cdot & 1     & \cdot & \cdot & 1     & \cdot & 1     & 1     & \cdot 
        \CodeAfter
            \begin{tikzpicture}
                \node () [fit=(5-3) (5-3), draw=magenta!80!black, rounded corners] {};
                \node () [fit=(5-6) (5-6), draw=magenta!80!black, rounded corners] {};
                \node () [fit=(5-8) (5-10), draw=magenta!80!black, rounded corners] {};
            \end{tikzpicture}
        \end{pNiceArray}
        \]
    }
\end{minipage}
}
\end{center}

\paragraph{The derived pushforward $\Rfunc h_\ast k_\Sigma$} We compute $\Rfunc h_\ast k_\Sigma$ analogously to above. The relabeling is now as follows:
\begin{itemize}
    \item The columns $45$, $46$, and $56$ in $\eta^0$ are relabeled to $14$, $34$, and $13$, respectively,
    \item the rows $5$ and $6$ in $\eta^0$ are relabeled to $1$ and $3$, respectively,
    \item the columns $5$ and $6$ in $\eta^1$ are relabeled to $1$ and $3$, respectively.
\end{itemize}
There is now several new non-empty submatrices labeled by a single element, but all of them are zero matrices. This is guaranteed, because none of the faces that $h$ identifies are in a poset relation in $\Sigma$. The relabeled matrices, therefore, represent the minimal complex of injective sheaves $\Rfunc h_\ast k_\Sigma$:


\begin{center}
\scalebox{.9}{
\noindent\begin{minipage}{.49\textwidth}
    \small
    {\setlength{\arraycolsep}{2pt}
        \[
        \begin{pNiceArray}[first-row,first-col, margin=.5em]{ccccccccc}
            \Rfunc h_\ast\eta^0 & 013   & 014   & 034   & 123   & 124   & 234   & 14    & 34    & 13    \\
            01     & 1     & 1     & \cdot & \cdot & \cdot & \cdot & \cdot & \cdot & \cdot \\
            03     & 1     & \cdot & 1     & \cdot & \cdot & \cdot & \cdot & \cdot & \cdot \\
            04     & \cdot & 1     & 1     & \cdot & \cdot & \cdot & \cdot & \cdot & \cdot \\
            12     & \cdot & \cdot & \cdot & 1     & 1     & \cdot & \cdot & \cdot & \cdot \\
            13     & 1     & \cdot & \cdot & 1     & \cdot & \cdot & \cdot & \cdot & \cdot \\
            14     & \cdot & 1     & \cdot & \cdot & 1     & \cdot & \cdot & \cdot & \cdot \\
            23     & \cdot & \cdot & \cdot & 1     & \cdot & 1     & \cdot & \cdot & \cdot \\
            24     & \cdot & \cdot & \cdot & \cdot & 1     & 1     & \cdot & \cdot & \cdot \\
            34     & \cdot & \cdot & 1     & \cdot & \cdot & 1     & \cdot & \cdot & \cdot \\
            4      & \cdot & 1     & \cdot & \cdot & \cdot & \cdot & 1     & \cdot & \cdot \\
            4      & \cdot & 1     & \cdot & \cdot & \cdot & \cdot & \cdot & 1     & \cdot \\
            1      & \cdot & \cdot & \cdot & \cdot & \cdot & \cdot & 1     & \cdot & 1     \\
            3      & \cdot & \cdot & \cdot & \cdot & \cdot & \cdot & \cdot & 1     & 1   
        \CodeAfter
            \begin{tikzpicture}
                \node () [fit=(3-2) (3-3), draw=magenta!80!black, rounded corners] {};
                \node () [fit=(3-5) (3-8), draw=magenta!80!black, rounded corners] {};
                \node () [fit=(6-2) (6-3), draw=magenta!80!black, rounded corners] {};
                \node () [fit=(6-5) (6-8), draw=magenta!80!black, rounded corners] {};
                \node () [fit=(8-2) (11-3), draw=magenta!80!black, rounded corners] {};
                \node () [fit=(8-5) (11-8), draw=magenta!80!black, rounded corners] {};
            \end{tikzpicture}
        \end{pNiceArray}
        \]
    }
\end{minipage}\hspace{5mm}%
\begin{minipage}{.49\textwidth}
    \small
    {\setlength{\arraycolsep}{2pt}
        \[
        \begin{pNiceArray}[first-row,first-col, margin=.5em]{ccccccccccccc}
            \Rfunc h_\ast\eta^1 & 01    & 03    & 04    & 12    & 13    & 14    & 23    & 24    & 34    & 4     & 4     & 1     & 3     \\
            0      & 1     & 1     & 1     & \cdot & \cdot & \cdot & \cdot & \cdot & \cdot & \cdot & \cdot & \cdot & \cdot \\
            1      & 1     & \cdot & \cdot & 1     & 1     & 1     & \cdot & \cdot & \cdot & \cdot & \cdot & \cdot & \cdot \\
            2      & \cdot & \cdot & \cdot & 1     & \cdot & \cdot & 1     & 1     & \cdot & \cdot & \cdot & \cdot & \cdot \\
            3      & \cdot & 1     & \cdot & \cdot & 1     & \cdot & 1     & \cdot & 1     & \cdot & \cdot & \cdot & \cdot \\
            4      & \cdot & \cdot & 1     & \cdot & \cdot & 1     & \cdot & 1     & 1     & \cdot & \cdot & \cdot & \cdot 
        \CodeAfter
            \begin{tikzpicture}
                \node () [fit=(5-3) (5-3), draw=magenta!80!black, rounded corners] {};
                \node () [fit=(5-6) (5-6), draw=magenta!80!black, rounded corners] {};
                \node () [fit=(5-8) (5-11), draw=magenta!80!black, rounded corners] {};
            \end{tikzpicture}
        \end{pNiceArray}
        \]
    }
\end{minipage}
}
\end{center}

\paragraph{The injective resolution of $k_\Gamma$ and the derived pushforward $\Rfunc l_\ast k_\Gamma$} Here we provide the labeled matrices describing the injective resolution of $k_\Gamma$, and its pushforward $\Rfunc l_\ast k_\Gamma$. Note that no rows labeled by the simplices on the boundary -- $45$, $46$, $56$, $4$, $5$, $6$ -- appear in the injective resolution of $k_\Gamma$.

    \noindent\begin{minipage}{\textwidth}
            \centering
        \includegraphics[width=160mm]{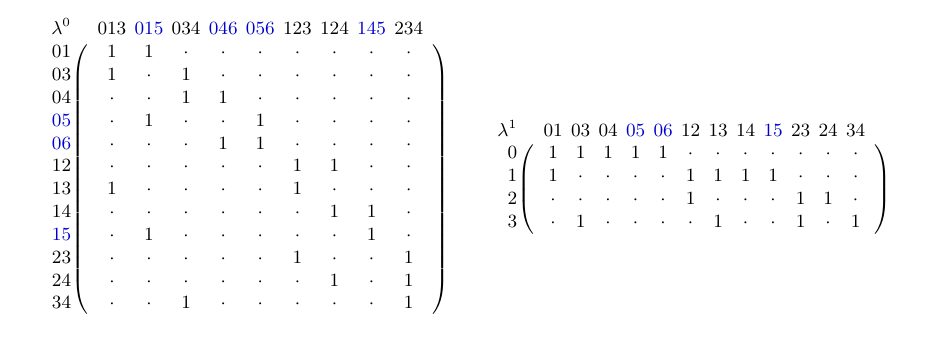}
    \end{minipage}

    To get the pushforward $\Rfunc l_\ast k_\Gamma$, we need to relabel the matrices according to $l$, and minimize. We highlighted in blue the labels that change. Recall that the map $l$ is defined on vertices by mapping $5,6\mapsto 4$ and other vertices identically.

    \noindent\begin{minipage}{\textwidth}
        \centering
        \includegraphics[width=160mm]{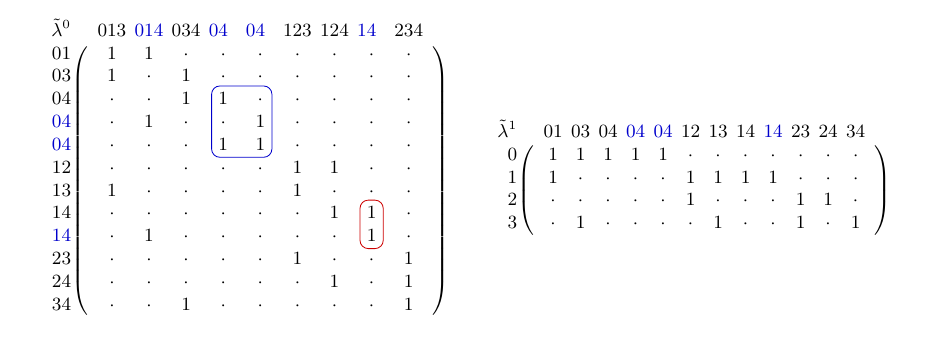}
    \end{minipage}

    We highlight the non-trivial 04- and 14-block in blue and red, respectively. Due to their ranks, we see we need to remove 2 columns and 2 rows from the blue 04-block, and 1 column and 1 row from the red 14-block. After the Gaussian elimination and deletion as described in Algorithm~\ref{algo:peeling}, we get the following matrices.

    \begin{center}
    \scalebox{.9}{
        \noindent
        \begin{minipage}{.47\textwidth}
            \vspace{3mm}
            \small
            {\setlength{\arraycolsep}{2pt}
                \[
                \begin{pNiceArray}[first-row,first-col, margin=.5em]{cccccc}
                    \Rfunc l_\ast\lambda^0   & 013   & 014   & 034   & 123   & 124   & 234   \\
                    01                 & 1     & 1     & \cdot & \cdot & \cdot & \cdot \\
                    03                 & 1     & \cdot & 1     & \cdot & \cdot & \cdot \\
                    04                 & \cdot & 1     & 1     & \cdot & \cdot & \cdot \\
                    12                 & \cdot & \cdot & \cdot & 1     & 1     & \cdot \\
                    13                 & 1     & \cdot & \cdot & 1     & \cdot & \cdot \\
                    14                 & \cdot & 1     & \cdot & \cdot & 1     & \cdot \\
                    23                 & \cdot & \cdot & \cdot & 1     & \cdot & 1     \\
                    24                 & \cdot & \cdot & \cdot & \cdot & 1     & 1     \\
                    34                 & \cdot & \cdot & 1     & \cdot & \cdot & 1     
                \CodeAfter
                \end{pNiceArray}
                \]
            }
            \hspace{3mm}
        \end{minipage}\hfill%
        \begin{minipage}{.47\textwidth}
            \small
            {\setlength{\arraycolsep}{2pt}
                \[
                \begin{pNiceArray}[first-row,first-col, margin=.5em]{ccccccccc}
                    \Rfunc l_\ast\lambda^1 & 01    & 03    & 04    & 12    & 13    & 14    & 23    & 24    & 34    \\
                    0                & 1     & 1     & 1     & \cdot & \cdot & \cdot & \cdot & \cdot & \cdot \\
                    1                & 1     & \cdot & \cdot & 1     & 1     & 1     & \cdot & \cdot & \cdot \\
                    2                & \cdot & \cdot & \cdot & 1     & \cdot & \cdot & 1     & 1     & \cdot \\
                    3                & \cdot & 1     & \cdot & \cdot & 1     & \cdot & 1     & \cdot & 1     
                \end{pNiceArray}
                \]
            }
        \end{minipage}
    }
    \end{center}

\vspace{5mm} 
Now that we computed the derived push-forward of several maps, we have three examples of (complexes of) sheaves on the triangulation of the sphere. We will now use microlocal Morse theory to analyze these examples. First, we choose our analogue of a discrete Morse function: an order preserving map, $f$ (see Figure~\ref{fig:ex_wedge-circle-sphere}), from the face-relation poset of the triangulation of the sphere to the real line (or, more precisely, to a totally ordered set). We then compute which fibers $f^{-1}(X)$ lie in the discrete microsupport of each sheaf. For this example we choose $f$ to be an honest discrete Morse function in the sense of \cite{Forman1998,Forman2002} (see Figure \ref{fig:Hasse-diagrams}). However, Theorem \ref{thm:microlocal-morse} and Theorem \ref{thm:micro-morse-ineq} make no assumptions on $f$ other than the requirement that it is order preserving. 

\paragraph{Intervals in the discrete microsupport}
Giving a full list of all of the locally closed subsets of $\Lambda$ in the microsupport of $\Rfunc g_\ast k_\Sigma$ (or $\Rfunc h_\ast k_\Sigma$) is cumbersome and mostly unnecessary. In practice, we are mostly interested in the locally closed subsets which arise as fibers of a discrete Morse function (for example, the fibers of $f$). Below we will compute and visualize the intersection of $\mu\supp^! \Rfunc h_\ast k_\Sigma$ with several intervals of interest.

\begin{figure}[H]
    \centering
    \footnotesize
    \renewcommand{\arraystretch}{1.3}
    \begin{NiceTabular}{|p{.15cm}||*{5}{p{.2cm}|}*{9}{p{.3cm}|}*{6}{p{.45cm}|}}
     \hline
     \multicolumn{21}{c}{Intervals in $\mu\supp^! \Rfunc h_\ast k_\Sigma$} \\
     \hline
    &0&1&2&3&4&01&03&04&12&13&14&23&24&34&013&014&034&123&124&234\\
     \hline
     $0$&\checkmark& & & & &\ding{55} &\ding{55}&\ding{55}&&&&&&&\ding{55}&\ding{55}&\ding{55}&&& \\
     \hline
     $1$&& \checkmark&&&&\checkmark&&&\checkmark&\ding{55}&\ding{55}&&&&\ding{55}&\ding{55}&&\ding{55}&\ding{55}&\\
     \hline
     $4$&&&& &\checkmark & & &\checkmark& & &\checkmark &&\checkmark &\checkmark&&\checkmark&\checkmark&&\checkmark&\checkmark\\
     \hline
    \end{NiceTabular}\vspace{2mm}
    
    \begin{NiceTabular}{|p{.15cm}||*{5}{p{.2cm}|}*{9}{p{.3cm}|}*{6}{p{.45cm}|}}
    \hline
     \multicolumn{21}{c}{Intervals in $\mu\supp^\ast \Rfunc h_\ast k_\Sigma$} \\
     \hline
    &0&1&2&3&4&01&03&04&12&13&14&23& 24 &34&013&014&034&123&124&234\\
     \hline
     $0$&\checkmark& & & & &\ding{55} &\ding{55}&\ding{55}&&&&&&&\ding{55}&\ding{55}&\ding{55}&&& \\
     \hline
     $1$&& \checkmark&&&&\checkmark&&&\checkmark&\ding{55}&\ding{55}&&&&\ding{55}&\ding{55}&&\ding{55}&\ding{55}&\\
     \hline
     $4$&&&& &\checkmark & & &\ding{55}& & &\checkmark &&\ding{55} &\checkmark&&\checkmark&\checkmark&&\checkmark&\checkmark\\
     \hline
    \end{NiceTabular}
    
    \vspace{7mm}
    \centering
    \begin{subfigure}{.25\textwidth}
        \centering
        \includegraphics[height=40mm]{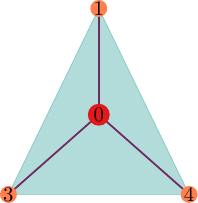}
        \caption{}
    \end{subfigure}
    \begin{subfigure}{.25\textwidth}
        \centering
        \includegraphics[height=40mm]{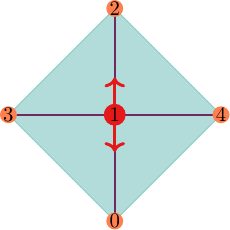}
        \caption{}
    \end{subfigure}
    \begin{subfigure}{.3\textwidth}
        \centering
        \includegraphics[height=40mm]{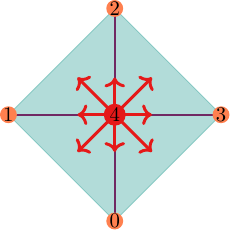}
        \caption{}
    \end{subfigure}

    \vspace{5mm}
    \begin{subfigure}{.25\textwidth}
        \centering
        \includegraphics[height=40mm]{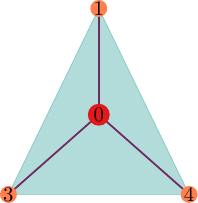}
        \caption{}  
    \end{subfigure}
    \begin{subfigure}{.25\textwidth}
        \centering
        \includegraphics[height=40mm]{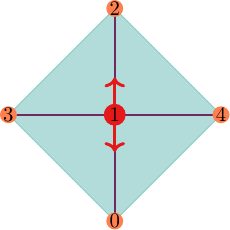}
        \caption{}
    \end{subfigure}
    \begin{subfigure}{.3\textwidth}
        \centering
        \includegraphics[height=40mm]{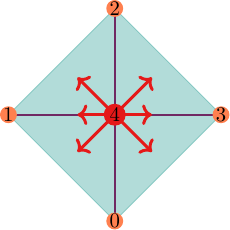}
        \caption{}
    \end{subfigure}
    
    \caption{Entries in the tables mark (an incomplete list of) intervals in the microsupport. The rows correspond to three chosen simplicies, $\sigma$, and the columns correspond to each of the simplicies, $\tau$, in $\Lambda$. Each entry of the first table is given $\checkmark$ if $[\sigma,\tau]\in\mu\supp^! \Rfunc h_\ast k_\Sigma$, \ding{55} if $[\sigma,\tau]\not\in\mu\supp^! \Rfunc h_\ast k_\Sigma$, and are left empty if $\sigma\not\le \tau$. Similarly, the second table gives the corresponding intervals for $\mu\supp^\ast \Rfunc h_\ast k_\Sigma$. Below the tables we visualize the intervals contained in the discrete microsupport. In the first row ((a), (b), and (c)) red arrows correspond to intervals $[\sigma,\tau]\coloneqq \{\gamma\in\Lambda:\sigma\le\gamma\le\tau\}$ (where $\sigma$ is the base of the arrow and $\tau$ is the head) contained in $\mu\supp^!\Rfunc h_\ast k_\Sigma$. In the second row ((d), (e), and (f)) red arrows correspond to intervals contained in $\mu\supp^\ast \Rfunc h_\ast k_\Sigma$. The first column ((a) and (d)) visualize intervals in the star of $0$, in the second column ((b) and (e)) the star of $1$, and in the third column ((c) and (f)) the star of $4$.}
        \label{fig:microsupport}
\end{figure}

\paragraph{Critical elements}
Here we provide a table of $(\Rfunc g_\ast k_\Sigma,f,\ast)$, $(\Rfunc g_\ast k_\Sigma,f,!)$, $(\Rfunc h_\ast k_\Sigma,f,\ast)$, $(\Rfunc h_\ast k_\Sigma,f,!)$, $(\Rfunc l_\ast k_\Gamma,f,\ast)$, and $(\Rfunc l_\ast k_\Gamma,f,!)$-critical elements of $\Pi$. Below we give more details of the four computations for the pair $B=\{4,24\}$. We fix $j:B \hookrightarrow \Lambda$ to be the inclusion map.
\begin{figure}[H]
    \centering
    \footnotesize
    \renewcommand{\arraystretch}{1.3}
    \begin{NiceTabular}{ |p{2cm}||*{11}{p{.5cm}|}}
        \hline
        \multicolumn{12}{c}{Critical Elements} \\
        \hline
        &A&B&C&D&E&F&G&H&I&J&K\\
        \hline
        $(\Rfunc g_\ast  k_\Sigma, f,!)$   & \checkmark& \checkmark& & & & & & & & &\checkmark\\
        $(\Rfunc g_\ast  k_\Sigma, f,\ast)$  & \checkmark& \checkmark& & & & & & & & & \checkmark\\
        \hline
        $(\Rfunc h_\ast k_\Sigma, f,!)$ & \checkmark& \checkmark& \checkmark&\checkmark & &\checkmark & & \checkmark& &\checkmark & \checkmark\\
        $(\Rfunc h_\ast k_\Sigma, f,\ast)$ & \checkmark& & \checkmark& \checkmark& &\checkmark & &\checkmark & &\checkmark &\checkmark \\
        \hline
        $(\Rfunc l_\ast k_\Gamma, f,!)$   & \checkmark& \checkmark& & & & & & & & &\checkmark\\
        $(\Rfunc l_\ast k_\Gamma, f,\ast)$  & \checkmark& \checkmark& & & & & & & & & \checkmark\\
        \hline
    \end{NiceTabular}
    \caption{Entries in the table mark critical elements of the quotient poset $\Pi$, for various notions of critical corresponding to row labels. }
    \label{fig:critical-elements}
\end{figure}

\paragraph{Computing critical elements: $(\Rfunc g_\ast k_\Sigma,f,\ast)$} Here we check whether $B=\{4,24\}$ is in $\mu\supp^\ast \Rfunc g_\ast k_\Sigma$. By definition, this happens iff $\mathbb{H}^d \left(\Rfunc j_!\Rfunc j^\ast \Rfunc g_\ast k_\Sigma\right)\neq 0$, where $j:B \hookrightarrow \Lambda$.

We first need to compute the pullback $\Rfunc j^\ast (\Rfunc g_\ast k_\Sigma)$ using Algorithm~\ref{algo:pullback_of_complex_of_injectives} as described in Section~\ref{sec:pullback_algo}. According to the algorithm, we copy $\Rfunc g_\ast\eta^0$ and $\Rfunc g_\ast\eta^1$ as $\gamma^{-1}$ and $\gamma^0$, viewing them as matrices labeled by the poset cylinder $\Lambda\sqcup_j B$, and then run the updating step from Algorithm~\ref{algo:injective_resolution_tail} for elements in $B$. It is, however, enough to only copy the parts above $B$ in $\Lambda\sqcup_j B$, i.e., the submatrices labeled by $\St 4$, which we highlighted in the matrices above. Altogether we get the following matrices $\gamma^d$, where one row was added to each of them in the updating step:

\noindent
\begin{minipage}{\textwidth}
    \centering
    \includegraphics[width=120mm]{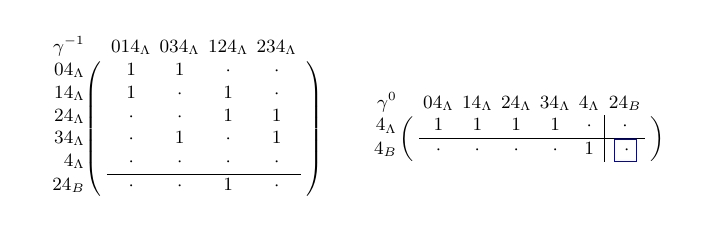}
\end{minipage}

We only keep the highlighted submatrix $\gamma^0[B, B]$. It already represents a minimal complex of injective sheaves, and no further minimization is needed. Altogether, we have \[
    \Rfunc j^\ast (\Rfunc g_\ast k_\Sigma) = \cdots \rightarrow 0 \rightarrow [24] \xrightarrow{0} [4] \rightarrow 0 \rightarrow \cdots.
\]

Next, we need to compute proper pushforward $\Rfunc j_!(\Rfunc j^\ast \Rfunc g_\ast k_\Sigma)$. The general construction is described by Algorithm~\ref{algo:proper_pushforward_of_complex_of_injectives} in Section~\ref{sec:proper_pushforward_algo}. We first copy the $1\times 1$ matrix from above, and then run the updating step for all elements in $\Cl B\setminus B = \{2\}$. This forces us to add one row in the first matrix, but nothing else in the following. We get \[
     \Rfunc j_!(\Rfunc j^\ast \Rfunc g_\ast k_\Sigma) = \cdots \rightarrow 0 \rightarrow [24] \xrightarrow{\text{\footnotesize $\begin{pmatrix}
         0 \\
         1
     \end{pmatrix}$}} [4]\oplus [2] \rightarrow 0 \rightarrow \cdots.
\]
Finally,\[
    \mathbb{H}^d \left(\Rfunc j_!\Rfunc j^\ast \Rfunc g_\ast k_\Sigma\right) = \begin{cases}
        k \text{, if $d=1$} \\
        0 \text{ otherwise.}
    \end{cases}
\]
The hypercohomology is non-trivial, which means that $B \in \mu\supp^\ast \Rfunc g_\ast k_\Sigma$.

\paragraph{Computing critical elements: $(\Rfunc g_\ast k_\Sigma,f,!)$} Next, we check whether $B=\{4,24\}$ is in $\mu\supp^! \Rfunc  g_\ast k_\Sigma$, i.e., whether $\mathbb{H}^d \left(\Rfunc j^! \Rfunc g_\ast k_\Sigma\right)\neq 0$. As discussed in Section~\ref{sec:proper_pullback_algo}, computing the proper pullback is easy -- we only need to consider the submatrices $\Rfunc g_\ast\eta^d[B,B]$. We get \[
\Rfunc j^! \Rfunc g_\ast k_\Sigma = \cdots \rightarrow 0 \rightarrow 0 \rightarrow [24]\oplus[4]\xrightarrow{\text{\footnotesize $\begin{pmatrix}
    1 & 0
\end{pmatrix}$}} [4] \rightarrow 0 \rightarrow \cdots,\]
the hypercohomology of which is \[
    \mathbb{H}^d \left(\Rfunc j^! \Rfunc g_\ast k_\Sigma\right) = \begin{cases}
        k \text{, if $d=1$} \\
        0 \text{ otherwise.}
    \end{cases}
\]
Therefore, $B\in \mu\supp^! \Rfunc g_\ast k_\Sigma$.

\paragraph{Computing critical elements: $(\Rfunc h_\ast k_\Sigma,f,\ast)$} We follow the same computation as above, but for $h$. For the pullback, we get matrices

\noindent
\begin{minipage}{\textwidth}
    \centering
    \includegraphics[width=140mm]{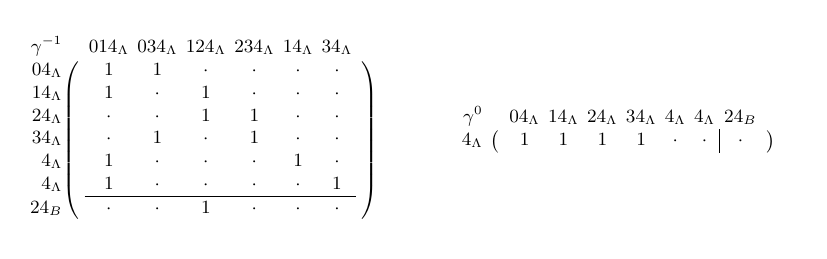}
\end{minipage}

\noindent The submatrices $\gamma^d[B,B]$ are all empty with one label in degree $0$, which yields \[
    \Rfunc j^\ast (\Rfunc h_\ast k_\Sigma) = \cdots \rightarrow 0 \rightarrow [24] \rightarrow 0 \rightarrow 0 \rightarrow \cdots.
\]
As above, to compute the proper pushforward, we start with the matrices we just computed, and perform the update step for $2$. We need to add a row in the matrix going from degree~$0$, which had one column labeled by $24$ and no rows. Hence, \[
    \Rfunc j_! (\Rfunc j^\ast \Rfunc h_\ast k_\Sigma) = \cdots \rightarrow 0 \rightarrow [24] \xrightarrow{(1)} [2] \rightarrow 0 \rightarrow \cdots.
\]
The hypercohomology of this complex is trivial, so $B\notin \mu\supp^\ast \Rfunc h_\ast k_\Sigma$.

\paragraph{Computing critical elements: $(\Rfunc h_\ast k_\Sigma,f,!)$} The matrices $\Rfunc h_\ast\eta_0[B,B]$ and $\Rfunc h_\ast\eta_1[B,B]$ give \[
    \Rfunc j^! \Rfunc h_\ast k_\Sigma = \cdots \rightarrow 0 \rightarrow 0 \rightarrow [24]\oplus[4]\oplus[4]\xrightarrow{\text{\footnotesize $\begin{pmatrix}
    1 & 0 & 0
\end{pmatrix}$}} [4] \rightarrow 0 \rightarrow \cdots,
\]
with hypercohomology \[
    \mathbb{H}^d \left(\Rfunc j^! \Rfunc g_\ast k_\Sigma\right) = \begin{cases}
        k^2 \text{, if $d=1$} \\
        0 \text{ otherwise,}
    \end{cases}
\]
and hence $B\in \mu\supp^! \Rfunc h_\ast k_\Sigma$.

\paragraph{Filtrations and hypercohomology} In Figure \ref{fig:Betti-numbers} we list Betti numbers for hypercohomology groups associated to the filtration of $\Lambda$ by sublevel sets of $f$. As Theorem~\ref{thm:microlocal-morse} implies, the value in column $X$ for the sublevel filtration differs from the value in its left neighbor only if $X$ is a critical element for the corresponding notion of microlocal support as shown in the table above. Similarly, for the superlevel filtration, non-critical columns $X$ have the same value as their right neighbors.

\begin{figure}[htb]
    \centering
    \footnotesize
    \renewcommand{\arraystretch}{1.3}
    \begin{NiceTabular}{ |p{2.3cm}||*{11}{p{.7cm}|}}
         \hline
         \multicolumn{12}{c}{Rank of hypercohomology groups in degree 0, 1, and 2} \\
         \hline
         $\dim\mathbb{H}^d\left(-\right)$ & A&B&C&D&E&F&G&H&I&J&K\\
         \hline
          $\Rfunc i_{Z}^! \Rfunc g_\ast k_\Sigma$  & 0,0,1& 0,1,0& & & & & & & & &1,0,0\\
        $\Rfunc i_{\preceq x}^! \Rfunc g_\ast k_\Sigma$  & 0,0,1& 0,1,1&0,1,1 & 0,1,1& 0,1,1& 0,1,1& 0,1,1&0,1,1 &0,1,1 &0,1,1 &1,1,1\\
         $\Rfunc i_{\succeq x}^\ast \Rfunc g_\ast k_\Sigma$ & 1,1,1&1,1,0&1,0,0 &1,0,0 &1,0,0 &1,0,0 &1,0,0 & 1,0,0&1,0,0 & 1,0,0&1,0,0\\
         \hline
         $\Rfunc i_{Z!}\Rfunc i_{Z}^\ast \Rfunc g_\ast k_\Sigma$  & 1,0,0& 0,1,0& & & & & & & & & 0,0,1\\
          $\Rfunc i_{\preceq x}^\ast \Rfunc g_\ast k_\Sigma$ & 1,0,0 & 1,1,0 & 1,1,0& 1,1,0& 1,1,0& 1,1,0& 1,1,0& 1,1,0& 1,1,0& 1,1,0&1,1,1\\
         \hline
        \hline

         $\Rfunc i_{Z}^! \Rfunc h_\ast k_\Sigma$ & 0,0,1& 0,2,0& 0,1,0& 0,1,0 & &1,0,0 & & 1,0,0& &1,0,0 & 1,0,0\\
          $\Rfunc i_{\preceq x}^! \Rfunc h_\ast k_\Sigma$ & 0,0,1& 0,2,1& 0,3,1&0,4,1 & 0,4,1&0,3,1 &0,3,1 & 0,2,1& 0,2,1 &0,1,1 & 1,1,1\\
          $\Rfunc i_{\succeq x}^\ast \Rfunc h_\ast k_\Sigma$ & 1,1,1& 1,1,0& 2,0,0&3,0,0 & 4,0,0&4,0,0 & 3,0,0& 3,0,0& 2,0,0&2,0,0 & 1,0,0\\
         \hline
         $\Rfunc i_{Z!}\Rfunc i_{Z}^\ast \Rfunc h_\ast k_\Sigma$& 1,0,0& & 1,0,0& 1,0,0& &0,1,0 & &0,1,0 & &0,1,0 &0,0,1 \\
         $\Rfunc i_{\preceq x}^\ast \Rfunc h_\ast k_\Sigma$ & 1,0,0&1,0,0 & 2,0,0& 3,0,0& 3,0,0&2,0,0 &2,0,0 &1,0,0 &1,0,0 &1,1,0 &1,1,1 \\
        
         \hline
         \hline
         
         $\Rfunc i_{Z}^! \Rfunc l_\ast k_\Gamma$ & 0,0,1& 0,1,0& &  & & & & & & & 1,0,0\\
          $\Rfunc i_{\preceq x}^! \Rfunc l_\ast k_\Gamma$ & 0,0,1&& && & & & &  && 1,0,0\\
          $\Rfunc i_{\succeq x}^\ast \Rfunc l_\ast k_\Gamma$ & 1,0,0& 1,1,0& 1,0,0&1,0,0 & 1,0,0&1,0,0 & 1,0,0& 1,0,0& 1,0,0&1,0,0 & 1,0,0\\
         \hline
         $\Rfunc i_{Z!}\Rfunc i_{Z}^\ast \Rfunc l_\ast k_\Gamma$& 1,0,0& 0,1,0 & & & & & & & & &0,0,1 \\
         $\Rfunc i_{\preceq x}^\ast \Rfunc l_\ast k_\Gamma$ & 1,0,0&1,1,0 & 1,1,0& 1,1,0& 1,1,0&1,1,0 &1,1,0 &1,1,0 &1,1,0 &1,1,0 &1,0,0 \\

        \hline
    \end{NiceTabular}
  
    \caption{Each entry in the table lists the dimension of hypercohomology groups in degree 0, 1, and~2 (empty entries have trivial hypercohomology). Here, $Z=f^{-1}(x)$ and $i_{\preceq x}$ ($i_{\succeq x}$) is the inclusion map of sublevel-sets (superlevel-sets) of $f$ (relative to the total order $\preceq$ on $\Pi$) into $\Lambda$, for $x$ an element of $\Pi$ corresponding to the column labels. }
    \label{fig:Betti-numbers}
\end{figure}

\newpage
\bibliographystyle{alpha}
\bibliography{bibliography}


\newpage
\appendix

\section{Background: Injective Sheaves}

\subsection{Duality between injective and projective resolutions}
\label{appendix:dualities}
We should comment on a matter of perspective and terminology. Sheaves over finite posets are closely related to several other mathematical objects studied by various research communities. Specifically, there is a great deal of work within the field of commutative algebra on minimal projective and free resolutions of modules over various kinds of algebras. Here, we choose to focus on the perspective and terminology which most closely aligns with classical sheaf theory in order to preserve intuition from that discipline. Below we briefly comment on the connections between sheaves, cosheaves, injective resolutions, and projective resolutions, with implications for computing the minimal injective hull of a sheaf.

Given a sheaf $F$ on a poset $\Pi$, let $\hat{F}$ denote a sheaf on $\Pi^\op$ defined by 
\[
\hat{F}(\pi)\coloneqq \Hom_k(F(\pi),k), 
\]
the $k$-linear functionals on $F(\pi)$, where $\hat{F}(\pi\le\tau)$ is the linear map from $\Hom_k(F(\pi),k)$ to $\Hom_k(F(\tau),k)$ defined by precomposing linear functionals on $F(\pi)$ with $F(\tau\le \pi)$. This defines an exact contravariant functor from the category of sheaves on $\Pi$ to the category of sheaves on $\Pi^\op$.  
In fact, this functor can be extended to an equivalence (of triangulated categories) between derived categories $D^b(\Pi)$ and $D^b(\Pi^\op)^\op$~\cite{Ladkani,Curry2018}. 
It is easy to show that $I$ is an injective sheaf on $\Pi$ if and only if $\hat{I}$ is a projective sheaf on $\Pi^\op$. Moreover, this contravariant functor interchanges injective and projective complexes of sheaves. If $I^\bullet\in D^b(\Pi)$, let $\hat{I}^{-\bullet}$ be the corresponding complex of projective sheaves (where the integer indices are given the opposite sign). 
It is again straightforward to show that $I^\bullet$ is an injective resolution of a sheaf $F$ on $\Pi$ if and only if $\hat{I}^{-\bullet}$ is a projective resolution of $\hat{F}$ on $\Pi^\op$. Moreover, this functor preserves minimality of the resolutions. This is to say that computing injective resolutions of sheaves on $\Pi$ is equivalent to computing projective resolutions of sheaves on $\Pi^\op$. We see two applications of this observation. First, we can compute left derived functors by taking a sheaf $F$ and simply computing the injective resolution of $\hat{F}$ on $\Pi^\op$, taking the transpose of each labeled matrix (Definition~\ref{def:labeled-matrix}) to obtain a projective resolution of $F$ on $\Pi$, and applying the desired right exact functor. Secondly, this reformulation highlights close connections between computational derived sheaf theory and commutative algebra. Particularly, theoretical results and implementations for computing minimal projective resolutions of modules over the incidence algebra of a finite poset (see~\cite{GreenSolbergZacharia}, for example) can be exploited for computations in derived sheaf theory, providing alternatives to Algorithm~\ref{algo:injective_hull_alpha} and Algorithm~\ref{algo:injective_resolution_tail} of the present paper.

\subsection{Examples of non-injective sheaves}

A sheaf $I$ is injective if any natural transformation from $A$ into it can be extended to any super-sheaf $B\hookleftarrow A$.
\begin{center}
\includegraphics[width=40mm]{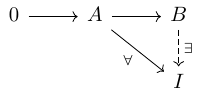}
\end{center}

Below are two simple examples of sheaves that are not injective.

\begin{example}\label{ex:non-injective_sheaf}
We describe a sheaf $F$ on a poset with two elements $\sigma\leq\tau$ which is not injective. Fix any vector space $W\neq 0$, and let $F(\sigma) = 0$ and $F(\tau) = W$. This sheaf injects in a sheaf $G$ given by $G(\sigma) = W = G(\tau)$ with $\smap G{\sigma}{\tau} = \id$. We claim that $F\overset{\id}{\rightarrow} F$ can not be extended to $G\rightarrow F$.\\
    \begin{minipage}{\textwidth}
    \centering
    \vspace{10pt}
    \includegraphics[width=100mm]{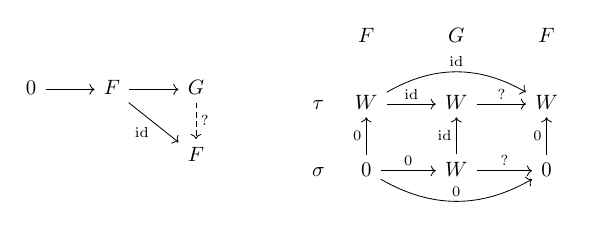}
    \vspace{10pt}
    \end{minipage}
Indeed, the triangle on the left diagram commutes iff the right diagram commutes. The only way to make the right square commute is to set both horizontal maps to $0$. But then the top triangle does not commute.
\end{example}
The same reasoning applies for any sheaf on any poset with a non-zero vector space one step above a zero vector space. Below we demonstrate one other obstruction to injectivity.

\begin{example}\label{ex:non-injective_sheaf_2}
We consider a three-element ``V'' shaped poset, a vector space $W\neq 0$, and two different endomorphisms $f,g: W\rightarrow W$. We define sheaves $A$, $B$ and $F$ as follows:

    \begin{minipage}{.95\textwidth}
    \centering
    \vspace{10pt}
    \includegraphics[width=100mm]{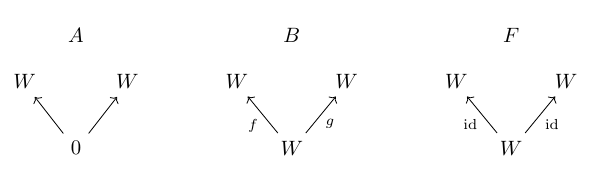}
    \vspace{10pt}
    \end{minipage}
    
    We claim that $F$ is not injective. The sheaf $A$ embeds into both $B$ and $F$ with identity maps. We ask whether we can extend the embedding $A\rightarrow F$ to a natural transformation $\beta: B\rightarrow F$, while respecting the embedding $A\rightarrow B$. The only choice we have is the bottom map $\beta_0: W\rightarrow W$, as the other two must be the identity. However, commutativity requires $f = \beta_0 = g$, which is impossible to satisfy, since $f\neq g$.
\end{example}

\subsection{Proofs of standard facts about injective sheaves}\label{app:sec:proofs_injective_sheaves}

\newtheorem*{lem:elementary_injective_sheaf}{Lemma~\ref{lem:elementary_injective_sheaf}}
\begin{lem:elementary_injective_sheaf}[{cf.\,\cite[Lemma 7.1.5]{Curry2014}}]
    Indecomposable injective sheaves are injective.
\end{lem:elementary_injective_sheaf}
\begin{proof}
    We show that, for a fixed poset $\Pi$ and $\pi\in \Pi$, $I=[\pi]$ satisfies Definition~\ref{def:injective_sheaf_2}. Given an inclusion $A\xhookrightarrow{f} B$ and a natural transformation $\alpha: A\rightarrow I$, we need to find an extension $\beta: B\rightarrow I$. For the linear map $A(\pi)\xhookrightarrow{f(\pi)} B(\pi)$, there is a projection $A(\pi)\xleftarrow{g}B(\pi)$ such that $g f(\pi)=\id_{A(\pi)}$.
    We define
    \begin{align*}
        \beta(\sigma) \coloneqq
        \begin{cases}
            \alpha(\pi)\circ g\circ B(\sigma\leq \pi) &\text{ if $\sigma\leq \pi$,} \\
            0 &\text{ otherwise.}
        \end{cases}
    \end{align*}
    For every $\sigma$, this satisfies $\beta(\sigma) f(\sigma)=\alpha(\sigma)$, because if $\sigma\leq\pi$, then\[
    \beta(\sigma) f(\sigma) = \alpha(\pi)\, g\, B(\sigma\leq \pi) f(\sigma) = \alpha(\pi)\, g\, f(\pi) A(\sigma\leq \pi) = \alpha(\pi) A(\sigma\leq \pi) = \alpha(\sigma),
    \]
    and otherwise both sides are 0.
    For the commutativity conditions, consider $\sigma\leq\tau\leq\pi$. Then\[
    \beta(\tau) \smap B \sigma \tau = \alpha(\pi)\, g\, \smap B \tau \pi \smap B \sigma \tau = \alpha(\pi)\, g\, \smap B \sigma \pi = \beta(\sigma) = \smap I \sigma \tau \beta(\sigma).
    \]
    If $\sigma\leq\tau\not\leq\pi$, then both sides are $0$.

    Note that this argument works for arbitrary vector spaces\footnote{Assuming the Axiom of Choice.}, not only finite-dimensional ones. Indeed, we only need to verify the existence of the map $g$. Consider a basis $\{b_t\}_{t\in T}$ of $B(\pi)$ such that for some subset $S\subset T$, the subset $\{b_s\}_{s\in S}$ is a basis of $\im{f(\pi)}$. Then $g$ is a composition of the projection onto $\im f(\pi)$ given by $\sum_{t\in T} \gamma_t b_t \mapsto \sum_{s\in S} \gamma_s b_s$ followed by the inverse of $f(\pi):A(\pi)\rightarrow \im f(\pi)$. Clearly, $gf(\pi)$ is the identity. The existence of the basis is easily shown with Zorn's lemma by first obtaining a maximal element, $M$, of the collection of sets of linearly independent vectors in $\im f(\pi)$, and then a maximal element of the collection of sets of linearly independent vectors in $B(\pi)$ that contain $M$.
\end{proof}

\newtheorem*{lem:coker-injective}{Lemma~\ref{lem:coker-injective}}
\begin{lem:coker-injective}[{cf.\ \cite[Lemma 1.3.1]{Shepard1985}}]
    A direct sum of injective sheaves is injective. Additionally, if $I\xhookrightarrow{\alpha} J$ is an injective natural transformation with $I,J$ injective sheaves, then $J\cong I\oplus \coker\alpha  $, and $\coker\alpha$ is an injective sheaf. 
\end{lem:coker-injective}
\begin{proof}
This proof is standard for any abelian category, we include a sketch for completeness. Suppose $I=A\oplus B$, with $A,B$ injective sheaves. Suppose $F\hookrightarrow G$ and $F\rightarrow I$. Then composition with projection gives maps $F\rightarrow A$ and $F\rightarrow B$. By injectivity of $A$ and $B$, each map extends to $G\rightarrow A$ and $G\rightarrow B$, respectively. The sum of these maps defines an extension $G\rightarrow I$, proving that $I$ is injective. The second claim follows by extending the identity map $I\rightarrow I$ to $J\rightarrow I$ by $\alpha$ and the injectivity of $J$. Then, the sum of the extension and the quotient map define an isomorphism $J\rightarrow I\oplus\coker\alpha$. The final claim follows by composing a given map $F\rightarrow\coker\alpha$ with the extension by zero map, $\coker\alpha\hookrightarrow J$, to get $F\rightarrow J$. Then, for $F\hookrightarrow G$, we define (by the injectivity of $J$) an extension $G\rightarrow J$. By post-composing with the projection map, we get the desired extension $G\rightarrow\coker\alpha$.
\end{proof}

\newtheorem*{prop:decomposition_of_injective_sheaves}{Proposition~\ref{prop:decomposition_of_injective_sheaves}}
\begin{prop:decomposition_of_injective_sheaves}[{cf.\,\cite[Lemma 7.1.6]{Curry2014}, \cite[Theorem 1.3.2]{Shepard1985}}] 
  Every injective sheaf is isomorphic to a direct sum of indecomposable injective sheaves.
\end{prop:decomposition_of_injective_sheaves}
\begin{proof}
We adapt the proof of \cite[Theorem 1.3.2]{Shepard1985} to the setting of finite posets on $n$ elements (rather than cell complexes). We fix some linear extension of the partial order, $(\pi_1,\dots,\pi_n)$, and let $\Pi_d=\left\{\pi_j\,\middle|\, j\leq d\right\}$. We will proceed with the proof by working inductively through this filtration of $\Pi$. We define support of a sheaf $I$ as\[
    \supp I \coloneqq \setdef{\pi\in\Pi}{I(\pi)\neq 0}.
\]

Assume that the result holds for injective sheaves supported on $\Pi_{d-1}$. Suppose $I$ is an injective sheaf with support contained in $\Pi_{d}$. If $\supp I\subseteq \Pi_{d-1}$, then the inductive assumption implies the result. Therefore, we are left to prove the result for $I$ such that $I(\pi_d)\neq 0$. Set $F_{\pi_d}$ to be the functor which assigns $I(\pi_d)$ to $\pi_d$ and the zero vector space to each other poset element (and the zero linear map to each poset relation). Then the identity map induces injective natural transformations
\[
F_{\pi_d}\xhookrightarrow{\alpha} I\qquad\text{and}\qquad F_{\pi_d}\hookrightarrow \bigoplus_{v\in B}[\pi_d],
\]
where $B$ is some basis of $I(\pi_d)$. Because $I$ is injective, we can extend $\alpha$ to a natural transformation
$
 \beta:  \bigoplus_{v\in B}[\pi_d]\rightarrow I.
$
It is injective, because for every $\sigma\leq\pi_d$, the linear map $I(\sigma\leq\pi_d)\beta(\sigma)=\beta(\pi_d)=\alpha(\pi_d)$ is injective.
By Lemma~\ref{lem:coker-injective}, this implies that 
\[
I\cong \coker \beta\oplus\bigoplus_{v\in B}[\pi_d],
\]
and that $\coker \beta$ is injective. Because $\supp\coker \beta\subseteq \Pi_{d-1}
$, the inductive hypothesis completes the proof.
\end{proof}

\section{The correspondence between Definition \ref{def:derived-category} and the bounded derived category of sheaves on a finite poset with the Alexandrov topology}\label{app:section:correspondence}
Here we sketch a correspondence between the bounded derived category of sheaves on a finite poset endowed with the Alexandrov topology and the derived category studied in Section \ref{sec:derived_categories}. Here we assume the reader is familiar with sheaf theory (for example, \cite{KashiwaraSchapira1994}, \cite{Bredon1997}, \cite{Curry2014}, and \cite{Shepard1985}). 

Let $X$ be a finite poset endowed with the Alexandrov topology. Let $\Mod(k_X)$ be the category of sheaves of $k$-vector spaces on $X$ (as defined in \cite{KashiwaraSchapira1994}), and 
\[
\Mod_{\fin}(k_X) = \setdef{
        \mathcal{F}\in \Mod(k_X)
    }{
        \dim_k\mathcal{F}_x<\infty \quad \forall x\in X
    },
\]
be the subcategory of sheaves with finite-dimensional stalks.
\begin{proposition}\label{lem:app:sheaves-and-representations}
        $\Mod_\fin(k_X)$ is equivalent to the category of functors from $X$ (with morphisms given by poset relations) to the category of finite dimensional $k$-vector spaces, $\Fun(X,\vect_k)$. Similarly, $\Mod_(k_X)$ is equivalent to $\Fun(X,\Vect_k)$.
\end{proposition}
\begin{proof}
   See \cite[Theorem 4.2.10]{Curry2014}.
\end{proof}

\begin{lemma}\label{lem:app:serre-subcategory}
$\Mod_\fin(k_X)$ is a Serre\footnote{In particular, also weak Serre.} subcategory of $\Mod(k_X)$. 
\end{lemma}
\begin{proof}
Assume 
\[F_0\rightarrow F_1 \rightarrow F_2 \]
is an exact sequence of sheaves on $X$, with $F_i\in \Mod_\fin(k_X)$ for $i\in\{0,2\}$. For each $x\in X$ we can apply the pullback $i_x^\ast$ along the inclusion map $i_x:\{x\}\hookrightarrow X$ to get the exact sequence (recall that $i^\ast_{x}$ is an exact functor) of $k$-vector spaces 
\[F_0(x)\rightarrow F_1(x) \rightarrow F_2(x) \]
Because $F_i(x)$ is finite-dimensional for $i\in\{0,2\}$, $F_1(x)$ is finite-dimensional. 
\end{proof}

\begin{proposition}[Prop 1.7.11 and Remark 1.7.12 in \cite{KashiwaraSchapira1994}]
If $C'$ is a weak Serre subcategory of a category $C$ such that for every monomorphism $f:a\rightarrow b$, $a\in \Ob(C')$, $b\in \Ob(C)$, there exists $g: b\rightarrow c$ with $c\in \Ob(C')$ such that $g\circ f$ is a monomorphism, then $D^\textrm{b}(C') \simeq D^\textrm{b}_{C'}(C)$, where $D^\textrm{b}_{C'}(C)$ denotes the full triangulated subcategory of $D^\textrm{b}(C)$ with cohomology objects in $C'$.  
\end{proposition}

Let $D^\textrm{b}(k_X)$ be the bounded derived category of sheaves on $X$ (as defined in \cite{KashiwaraSchapira1994}) and $D_\fin^\textrm{b}(k_X)$ be the subcategory whose cohomology objects are in $\Mod_\fin(k_X)$.
\begin{theorem}
    $D_\fin^b(k_X)$ is equivalent to $D^b(X)$ (with $D_\fin^b(k_X)$ defined above and $D^b(X)$ in Definition \ref{def:derived-category} ). 
\end{theorem}
\begin{proof}
For any monomorphism $f:\mathcal{F}\rightarrow \mathcal{G}$ with $\mathcal{F}\in \Mod_\fin(k_X)$, there exists a morphism $g:\mathcal{G}\rightarrow\mathcal{H}$, with $\mathcal{H}\in \Mod_\fin(k_X)$, such that $g\circ f$ is a monomorphism. Namely, choose $\mathcal{H}$ to be the injective hull of $\mathcal{F}$ (Definition \ref{def:injective_hull}). Then, because $\mathcal{H}$ is also an injective object in $\Mod(k_X)$ (\cite[Lemma 7.1.5]{Curry2014}, see also Lemma~\ref{lem:elementary_injective_sheaf}), a morphism $g$ exists with the desired properties. Therefore, by \cite[Prop 1.7.11]{KashiwaraSchapira1994} and the previous lemma (\ref{lem:app:serre-subcategory}), $D_\fin^b(k_X)$ is equivalent to $D^b(\Mod_\fin(k_X))$. By the above proposition (\ref{lem:app:sheaves-and-representations}), $D^b(\Mod_\fin(k_X))$ is equivalent to $D^b(\Fun(X,\vect_k))$. The result then follows by applying the replacement functor which maps a given complex $\mathcal{C}\in D^b(\Fun(X,\vect_k))$ to its minimal injective resolution in $D^b(X)$. 
\end{proof}

\section{Derived Category and Injective Resolution of a Sheaf}

\subsection{Computing a Basis for the Space of Morphisms in a Derived Category}\label{app:basis_for_morphisms}

We describe the spaces of all morphisms and all null-homotopic morphisms between two complexes of injective sheaves as solutions to systems of linear equations. Those systems of equations are smaller compared to the case of general sheaves, because for injective sheaves, we can take advantage of the labeled matrix representations of natural transformations. This said, the size of the systems of linear equations which define these morphisms are impractically large; the goal of this exposition is to give an explicit description of the space of morphisms between two objects in the \emph{derived category of sheaves}, including a way (albeit impractical) to compute their basis.

We fix two complexes of injective sheaves, $(I^\bullet,\eta^\bullet)$, $(J^\bullet,\lambda^\bullet)$, and denote by $m(I^d)$ and $m(J^d)$ the number of indecomposable injective summands in $I^d$ and $J^d$, respectively. We assume that both complexes are represented as complexes of labeled matrices, and we denote the matrices by the same symbols as the natural transformations they represent, e.g., $\eta^d$ is a $m(I^{d+1})\times m(I^{d})$ labeled matrix.

We can describe every morphism $\alpha^\bullet:I^\bullet\rightarrow J^\bullet$ as another collection of labeled matrices. The dimensions are fixed: $\alpha^d$ is a $m(J^d)\times m(I^d)$ matrix with columns labeled as columns of $\eta^d$, and rows labeled as columns of $\lambda^d$. Therefore, we have $\sum_d m(J^d)\cdot m(I^d)$ variables which will be used to define $
\alpha^\bullet$. We observe that the definition for $\alpha^\bullet$ to represent a morphism yields two sets of linear constrains:
\begin{itemize}
    \item \emph{poset constrains:} $\alpha^d[i,j]=0$ whenever $\sigma\not\leq\pi$ for $\sigma$ the label of the $i$-th row and $\pi$ the label of the $j$-th column of $\alpha^d$,
    \item \emph{commutativity constrains:} $\alpha^{d+1}\cdot \eta^d - \lambda^d\cdot\alpha^d = 0$ yields $m(J^{d+1})$ linear constrains for each $d$.
\end{itemize}

The system of linear equations for null-homotopic morphisms is the one above together with new variables and new constrains. We have $\sum_d m(I^d)\cdot m(J^{d-1})$ new variables for the matrices $h^d$ representing the natural transformations $I^d\rightarrow J^{d-1}$ as in Definition~\ref{def:null_homotopic}. Again, we label each $h^d$ accordingly to fit with the labeling of $\eta^\bullet$, $\lambda^\bullet$. We get the following two new sets of linear constrains:

\begin{itemize}
    \item \emph{poset constrains:} $h^d[i,j]=0$ whenever $\sigma\not\leq\pi$ for $\sigma$ the label of the $i$-th row and $\pi$ the label of the $j$-th column of $h^d$,
    \item \emph{homotopy constrains:} $\alpha^{d} - h^{d+1}\cdot\eta^d - \lambda^d\cdot h^{d-1} = 0$ yields $m(J^d)$ linear constrains for each $d$.
\end{itemize}

Altogether, we showed the following.
\begin{proposition}
    Computing a basis for the space of morphisms in derived category, $\Hom_{D^b(\Pi)}(I^\bullet,J^\bullet)$ (see Definition \ref{def:derived-category}), reduces to computing a basis for the space of solutions to the above system of linear equations.
\end{proposition}

\subsection{Injective Resolution via Order Complex}\label{app:injective_resolution_order_complex}

We present a non-inductive construction of a (not necessarily minimal) injective resolution of a given sheaf $F$. This section generalizes, from the constant sheaf to general sheaves, Lemma 1.3.17 of~\cite{Ladkani2008}. On a practical level, this allows one to compute the cohomology of right derived functors without first computing the full injective resolution (see Section \ref{sec:derived-functors}).

\begin{definition}
    The \emph{order complex}, $K(\Pi)$, of a finite poset $\Pi$, is the poset of strictly increasing chains $\pi_\bullet = \pi_0<\pi_1<\cdots<\pi_d$ in $\Pi$. The order complex has the structure of an abstract simplicial complex. Let $K^d(\Pi)$ denote the $d$-simplices of $K(\Pi)$, i.e. the set of chains $\pi_0<\pi_1<\cdots<\pi_d$ of length $d+1$.  
\end{definition}

\paragraph*{The construction} Given a sheaf $F$ on $\Pi$, we define (recalling the notation of Definition \ref{def:indecomposable_injective_sheaf}) 
\[
I^d \coloneqq   \bigoplus_{\pi_\bullet\in K^d(\Pi)}[\pi_0]^{ F(\pi_d)}.
\]
Suppose $\pi_\bullet\in K^d(\Pi)$ and $\pi_\bullet <_1\tau_\bullet$ (i.e. the chain $\pi_\bullet$ is obtained from the chain $\tau_\bullet$ by removing one element). Then $\pi_0\ge \tau_0$ and $\pi_d\le\tau_{d+1}$. Therefore, 
\[
F(\pi_d\le\tau_{d+1})\in \Hom (F(\pi_d),F(\tau_{d+1}))\cong \Hom \left( [\pi_0]^{F(\pi_d)}, [\tau_0]^{ F(\tau_{d+1})}\right) . 
\]
\begin{definition}[{\cite[Definition 6.1.9]{Curry2014}}]\label{def:signed-incidence-relation}
    A \emph{signed incidence relation} on $K(\Pi)$ is an assignment to each pair of simplices $\sigma_\bullet,\gamma_\bullet\in K(\Pi)$ a number $[\sigma_\bullet:\gamma_\bullet]\in\{-1,0,1\}$, such that
    \begin{enumerate}
        \item if $[\sigma_\bullet:\gamma_\bullet]\neq 0$, then $\sigma_\bullet<_1\gamma_\bullet$, and 
        \item for each pair of simplices $(\sigma_\bullet,\gamma_\bullet)$,  \[\sum_{\tau_\bullet\in K(\Pi)}[\sigma_\bullet:\tau_\bullet][\tau_\bullet:\gamma_\bullet]=0.\] 
    \end{enumerate}
\end{definition}
Using this identification, we define the natural transformation
$
    \eta^d:I^d\rightarrow I^{d+1}
$
so that on the $\pi_\bullet$-summand $[\pi_0]^{F(\pi_d)}$ of $I^d$, 
\begin{align*}
    \eta^d\vert_{[\pi_0]^{ F(\pi_d)}}= \sum_{\pi_\bullet<_1\tau_\bullet}[\pi_\bullet:\tau_\bullet]F(\pi_d\le \tau_{d+1}),
\end{align*}
where $F(\pi_d\le \tau_{d+1})\in\Hom \left( [\pi_0]^{ F(\pi_d)}, [\tau_0]^{F(\tau_{d+1})}\right)$ is understood to have its codomain as the $\tau_\bullet$-summand of $I^{d+1}$. Let $\alpha:F\hookrightarrow I^0$ be the natural transformation given by the maps 
\[\alpha(\sigma)
\coloneqq\ 
\sum_{\sigma\le\gamma}F(\sigma\le\gamma)
\ :\ 
F(\sigma)\xhookrightarrow{\ \ \ \ }  \bigoplus_{\sigma\le\gamma}F(\gamma)
\ \reflectbox{ $\coloneqq$ } I^0(\sigma).\]

The following lemma is a generalization of \cite[Theorem 6.12]{Curry2018}. 
\begin{lemma}\label{lemma:subdivision}
Let $t:K(\Pi)\rightarrow \Pi$ denote the poset map which assigns each chain $\pi_0<\pi_1<\cdots<\pi_d$ to its terminal element $\pi_d$, and $p:\Pi\rightarrow\{\text{pt}\}$. Given a sheaf $F$ on a finite poset $\Pi$, we have the following isomorphism of sheaf cohomology
\[
H^j\left(\Pi;F\right) \cong H^j\left(K(\Pi);t^\ast F\right)\text{, for each $j$.}
\]
\end{lemma}
\begin{proof}
When $\Pi$ is a cell complex, the theorem follows, for example, from \cite[Corollary 2.7.7 (iv)]{KashiwaraSchapira1994}. We include a more general proof for arbitrary finite posets below. Let $I^\bullet$ be an injective resolution of $F$. Because $t^\ast$ defines an exact functor from the category of sheaves on $\Pi$ to the category of sheaves on $K(\Pi)$, we have that 
\[
0\rightarrow t^\ast F\rightarrow t^\ast I^0\rightarrow t^\ast I^1\rightarrow \cdots\rightarrow t^\ast I^n\rightarrow 0
\]
is an exact sequence. We claim that each $t^\ast I^d$ is acyclic:
\[
H^j(\Rfunc (p\circ t)_\ast \left( t^\ast I^d\right))=0, \text{ for }j>0.
\]
Indeed, it is enough to prove that for each indecomposable injective sheaf $[\pi]$ on $\Pi$, the sheaf $t^\ast [\pi]$ is acyclic. Because $t^\ast[\pi]$ is the constant sheaf on the order complex of the downward closure, $\Cl\pi\coloneqq  \{\tau\in\Pi:\tau\le\pi\}$, of $\pi$ in $\Pi$, we have 
\[
H^j(\Rfunc (p\circ t)_\ast \left(t^\ast[\pi]\right))\cong H^j(|K(\Cl\pi)|;k).
\]
Notice that $K(\Cl\pi)$, as a simplicial complex, is equal to the cone of the simplicial complex $K(\Cl\pi-\pi)$. Therefore, $|K(\Cl\pi)|$ is contractible, $H^j(|K(\Cl\pi)|;k)=0$ for $j>0$, and $t^\ast[\pi]$ is acyclic. The complex $t^\ast I^\bullet$ is therefore an acyclic resolution of $t^\ast F$, and by standard results of homological algebra (for example, \cite[Theorem 4.1]{Bredon1997}), we have
\[
H^j(K(\Pi); t^\ast F)\cong H^j\big( (p \circ t)_\ast (t^\ast I^\bullet)\big) \cong H^j(\Pi;F). 
\]
\end{proof}

\begin{theorem}\label{thm:non-inductive}
The complex $0\rightarrow F\xrightarrow{\alpha} I^0\xrightarrow{\eta^0}I^1\xrightarrow{\eta^1} \cdots$ defined above is an injective resolution of $F$. 
\end{theorem}
\begin{proof}
By construction, each sheaf $I^j$ is injective, and each map $\eta^j$ (as well as $\alpha$) is a natural transformation. It remains to show that the sequence is an exact complex. It is enough to show that for each $\pi\in\Pi$, the sequence $0\rightarrow F(\pi)\xrightarrow{\alpha(\pi)}I^0(\pi)\xrightarrow{\eta^0(\pi)}I^1(\pi)\xrightarrow{\eta^1(\pi)}\cdots$ is exact.

We first define a functor $t^\ast$ from the category of sheaves on $\Pi$ to the category of sheaves on $K(\Pi)$. To each sheaf $F$ on $\Pi$, let $t^\ast F$ be the sheaf on $K(\Pi)$ defined by associating to each chain $\tau_\bullet=\tau_0<\cdots<\tau_d$ the `terminal' vector space: $t^\ast F(\tau_\bullet) \coloneqq F(\tau_d)$ and $t^\ast F(\tau_\bullet\le\gamma_\bullet)=F(\tau_d\le \gamma_j)$ for $\tau_\bullet \in K^d(\Pi)$, $\gamma_\bullet\in K^j(\Pi)$. Because $t^\ast(\eta)(\tau_\bullet) \coloneqq \eta(\tau_d):t^\ast F(\tau_\bullet)\rightarrow t^\ast G(\tau_\bullet)$ for any natural transformation $\eta:F\rightarrow G$, it is clear that $t^\ast$ is an exact functor. 

Notice that $0\rightarrow F(\pi)\xrightarrow{\alpha(\pi)}I^0(\pi)\xrightarrow{\eta^0(\pi)}I^1(\pi)\xrightarrow{\eta^1(\pi)}\cdots$ is identical to the compactly supported cochain complex of the sheaf $t^\ast (F\vert_{\St\pi})$ on the simplicial complex $K(\St\pi)$ \cite[Definition 6.2.1 and Definition 6.2.3]{Curry2014}. Therefore, exactness in $I^0(\pi)$ follows from \[
    \ker\eta_0(\pi) 
    \cong \Gamma(t^\ast(F\vert_{\St\pi}))
    \cong F(\pi)
    \cong \im \alpha(\pi),
\] and it remains to prove a vanishing property for the cohomology of $t^\ast(F\vert_{\St\pi})$, namely that \[H^j(K(\St\pi);t^\ast (F\vert_{\St\pi}))=0\] for $j>0$. By Lemma \ref{lemma:subdivision}, $H^j(K(\St\pi);t^\ast (F\vert_{\St\pi}))\cong H^j(\St\pi;F\vert_{\St\pi})$. 
Let $J^\bullet$ be an injective resolution of the sheaf $F\vert_{\St\pi}$ on the poset $\St\pi$. Then $H^j(\St\pi;F\vert_{\St\pi})$ is, by definition, the $j$-th cohomology group of the complex of vector spaces $J^\bullet(\pi)$, which, by the exactness of $J^\bullet$, is zero for $j>0$. 

\end{proof} 

This construction also allows us to prove Proposition~\ref{prop:persistent-homology}

\newtheorem*{prop:persistent-homology}{Proposition~\ref{prop:persistent-homology}}
\begin{prop:persistent-homology}
Suppose $f:\Sigma\rightarrow \Lambda$ is a simplicial map. As sheaves on $\Lambda$, 
\begin{align*}
    H^d\Rfunc f_\ast k_\Sigma  &\cong H^d(|f^{-1}(\St\blank)|;k)
\end{align*}
where $H^d(|f^{-1}(\St\blank)|;k)$ is the sheaf defined by associating the simplex $\lambda$ to the singular cohomology of the geometric realization of $f^{-1}(\St\lambda)$ (with linear maps induced by inclusion). 
\end{prop:persistent-homology}
\begin{proof}
Let $I^\bullet$ be the injective resolution described above. Then $\Rfunc f_\ast I^\bullet (\lambda)$ is the complex consisting of only the linear combinations of generators for indecomposable sheaves $[\pi]\subset I^\bullet$ such that $\pi\in f^{-1}(\St\lambda)$ and chain maps $\eta^\bullet(f^{-1}(\St\lambda))$. After forgetting the matrix labels representing $\Rfunc f_\ast I^\bullet (\lambda)$, we obtain the simplicial cochain complex of $K(f^{-1}(\St\lambda))$. The cohomology groups of this cochain complex are isomorphic to the singular cohomology of the geomtric realization of $f^{-1}(\St\lambda)$: \[H^d\Rfunc f_\ast k_\Sigma(\lambda) \cong H^d(|f^{-1}(\St\lambda)|;k),\] and the linear maps $H^d\Rfunc f_\ast(\kappa\le\lambda)$ are the usual cohomology maps \[H^d(|K(f^{-1}(\St\kappa))|;k)\rightarrow H^d(|K(f^{-1}(\St\lambda))|;k)\] induced by inclusion (cf.\,\cite[Chapter II Proposition 5.11]{Iversen}). 
\end{proof}

\subsection{Injective Hull Computation}\label{app:injective_hull}

We describe an explicit construction of the minimal injective hull of a sheaf $F$ on a poset $\Pi$. An injective hull of $F$ consists of an injective sheaf $I$ and an inclusion map of $F$ into $I$. To construct the minimal injective hull of $F$, we first find $M_F$, the subsheaf of $F$ with $M_F(\pi)$ the space of maximal vectors in $F(\pi)$ (recall Definition~\ref{def:maximal-vectors}), and zero maps between the spaces.
Recalling the notation described below Definition \ref{def:indecomposable_injective_sheaf}, we define\[
I^0 = \bigoplus_{\pi\in\Pi}[\pi]^{M_F(\pi)} \cong \bigoplus_{\pi\in\Pi}[\pi]^{\dim M_F(\pi)},
\]
where $[\pi]^{M_F(\pi)}$ is the injective sheaf with $[\pi]^{M_F(\pi)}(\sigma)=M_F(\pi)$ if $\sigma\leq\pi$, and $0$ otherwise. We can naturally include $M_F \xrightarrow{\gamma} I^0$, and extend this inclusion to $F\xrightarrow{\alpha} I^0$, using the injectivity of $I^0$. We choose the extension $\alpha=\sum_{\pi\in\Pi}\alpha_{\pi}$, where $\alpha_\pi$ is given by
\begin{align*}
    \alpha_{\pi}(\sigma) \coloneqq
    \begin{cases}
        \gamma(\pi) \circ\, \Proj_{M_F(\pi)} \circ\, F(\sigma\leq \pi) &\text{ if $\sigma\leq \pi$,} \\
        0 &\text{ otherwise.}
    \end{cases}
\end{align*}

\begin{proposition}\label{prop:construction_of_minimal_injective_hull}
    The construction above yields the minimal injective hull $F\xrightarrow{\alpha} I^0$.
\end{proposition}
\begin{proof}
We claim that $\alpha$ is injective. Let $u\in \ker\alpha(\sigma)$. Then\[
   0 = \alpha(\sigma)(u) = \sum_{\pi\in\Pi}\alpha_\pi(\sigma)(u) = \sum_{\sigma\leq\pi} \gamma(\pi) \circ\, \Proj_{M_F(\pi)} F(\sigma\leq \pi)(u),
\]
which is equivalent to $\Proj_{M_F(\pi)} F(\sigma\leq \pi)(u)=0$ for every $\pi\geq\sigma$, since the images of different $\alpha_{\pi}$ have trivial intersections. But this means that $u=0$, because every non-zero vector is either maximal or maps onto some non-zero maximal vector via the sheaf maps.

An injective hull is minimal iff the minimal injective resolution starts with it. By Proposition~\ref{lem:minimal_injective_resolution_im_condition}, this minimality is equivalent to the condition that every maximal vector of $I^0$ is in $\im\alpha$. This is satisfied, as $M_F(\pi)$ are exactly the maximal vectors in $I^0(\pi)$.
\end{proof}

We give an explicit formulation of an algorithm computing $\alpha(\pi)$ as Algorithm~\ref{algo:injective_hull_alpha}. We first fix bases in $F$. For each $\pi\in\Pi$, we fix a basis $B(\pi)=(v_1,\dots,v_l,w_{l+1},\dots,w_{l+d})$, with $l,k$ dependent on $\pi$, such that $(w_{l+1},\dots,w_{l+d})$ is a basis of $M_F(\pi)$---this involves finding a basis of intersection of kernels and its complement. We use the same bases for $I^0(\pi)$. We assume that all maps $F(\pi\leq\sigma)$ are expressed with respect to those bases.

To express $\alpha(\pi)$ with respect to the fixed bases, we need to find the image of each $v_1,\dots,v_l, w_{1+l},\dots,w_{l+d}$. The maximal vectors $w_j$ are mapped identically to $M_F(\pi)\subseteq I_0(\pi)$. For the other vectors, $v_j$, we need to find $u_j\coloneqq \sum_{\pi<\sigma}\Proj_{M_F(\sigma)}F(\pi\leq\sigma)(v_j)$. The algorithm does that while avoiding redundant computations. Each $\sigma$ is added to $D$ at most once. If it is added, then $D[\sigma]=F(\pi\leq\sigma)(v_j)$, and $\Proj_{M_F(\pi)}F(\pi\leq\sigma)(v_j)$ is added to $u_j$. If $\sigma$ is never added to $D$, then $F(\pi\leq\sigma)(v_j)=0$. In the end, $u_j$ contains the desired sum.

\begin{algorithm}[htb]
    \caption{Minimal injective hull}
    \label{algo:injective_hull_alpha}
    \hspace*{\algorithmicindent} \textbf{Input:} $F$ with fixed bases as described above, $\pi\in\Pi$ \\
    \hspace*{\algorithmicindent} \textbf{Output:} $\alpha(\pi)$ as a $\left(\sum_{\pi\leq\sigma}\dim M_F(\sigma)\right)\times(\dim F(\pi))$ matrix
\begin{algorithmic}[1]
    \Procedure{Incl}{$\sigma$, $w\in M_F(\sigma)$}
        \State return inclusion of $w$ into $\bigoplus_{\pi<\tau}M_F(\tau)$ \Comment{just adding extra zeros}
    \EndProcedure
    \For{$v_j\in\{v_1,\dots,v_l\}$} \Comment{$(v_1,\dots,v_l,w_{l+1},\dots,w_{l+d})$ is the fixed basis of $F(\pi)$}
        \State $D\gets$ empty dictionary \Comment{keys: elemets $\sigma\in\Pi$, values: vectors in $F(\sigma)$}
        \State $D[\pi]\gets v_j$
        \State $u_j \gets 0$ \Comment{vector of length $\sum_{\pi<\tau}\dim M_F(\tau)$}
        \ForEach{$\sigma\geq\pi$ in non-decreasing order}
            \If{$\sigma\in\textrm{Keys}(D)$ and $D(\sigma)\neq 0$}
                \State $w\gets D[\sigma]$
                \Comment{$D[\sigma] = F(\pi\leq\sigma)(v_j)$}
                \ForEach{$\tau>_1\sigma$}
                    \If{$\tau\not\in\mathrm{Keys}(D)$}
                        \State $D[\tau]\gets F(\sigma\leq\tau)(w)$
                        \State $u_j\gets u_j + \textsc{Incl}(\tau,\, \Proj_{M_F(\tau)}(D[\tau]))$
                    \EndIf
                \EndFor
            \EndIf
            \State clear $D[\sigma]$ \Comment{optional, just to free up memory}
        \EndFor
        \State return a block matrix
            $\begin{pmatrix}%
                U & 0 \\%
                0 & I %
            \end{pmatrix}$,
            where $U=(u_1|\dots|u_l)$, \newline \hspace*{5.5cm} and $I$ is the identity matrix of order $d=\dim M_F(\pi)$
    \EndFor
\end{algorithmic}
\end{algorithm}

\subsection{Geometrical Meaning of Multiplicities in the Minimal Resolution} \label{app:thm:multiplicities}

\newtheorem*{thm:multiplicities}{Theorem~\ref{thm:multiplicities}}
\begin{thm:multiplicities}
Let $\Sigma$ be a finite simplicial complex, $k_\Sigma$ the constant sheaf on $\Sigma$ (viewed as a poset with the face relation), and $H_c^\bullet (|\St\sigma|;k)$ be the singular cohomology with compact support of the geometric realization of $\St\sigma$. Then
\[
m^d_{k_\Sigma}(\sigma)= \dim H_c^{d+\dim\sigma}(|\St\sigma|;k).
\]
\end{thm:multiplicities}
\begin{proof}
Let $I^\bullet$ be an injective resolution of the constant sheaf on $\Sigma$. By Corollary \ref{cor:proper-pull-back-multiplicity}, 
\[
m_{ k_\Sigma}^d(\sigma) = \dim H^d \Rfunc i_\sigma^! I^\bullet.\]
For $\sigma\in\Sigma$, let \[C(\sigma)\coloneqq \sigma^0\ast \left\{ \tau\setminus\sigma \,\middle|\, \tau\in\St\sigma\setminus\{\sigma\} \right\}\] be the cone of the link $\left\{ \tau\setminus\sigma \,\middle|\, \tau\in\St\sigma\setminus\{\sigma\} \right\}$ (where $\tau\setminus\sigma$ denotes set difference when $\tau$ and $\sigma$ are viewed as sets of vertices in $\Sigma$), with cone point given by the vertex $\sigma^0$. Then $\St_{C(\sigma)}\sigma^0 \cong \St_\Sigma\sigma$ as posets. Therefore, $I^\bullet\vert_{\St_\Sigma\sigma}$, viewed as a complex of sheaves over $\St_{C(\sigma)}\sigma^0$, is an injective resolution of the constant sheaf on $\St_{C(\sigma)}\sigma^0$. Let $J^\bullet$ be the injective resolution of the constant sheaf on the simplicial complex $C(\sigma)$, defined non-inductively via order complex in \ref{app:injective_resolution_order_complex}. Then $I^\bullet\vert_{\St_\Sigma\sigma}$ is quasi-isomorphic to $J^\bullet\vert_{\St_{C(\sigma)}\sigma^0}$.
By definition, $i^!_{\sigma^0}(J^\bullet)$ is exactly the complex of vector spaces used to compute compactly supported cohomology of the constant sheaf on $\St_{K(C(\sigma))}(\sigma^0)$, the star of $\sigma^0$ in the barycentric subdivision $ K(C(\sigma))$ of $C(\sigma)$ (see \cite[Definition 6.2.1]{Curry2014} and \cite{Shepard1985}). Because the geometric realisation $|\St_{K(C(\sigma))}(\sigma^0)|$ is homeomorphic to $|\St_{C(\sigma)}\sigma^0|$, the cohomology groups of this complex are isomorphic to $H_c^{d}(|\St_{C(\sigma)}(\sigma^0)|;k)$. Finally, because $|\St_\Sigma \sigma|$ is homeomorphic to $\mathbb{R}^{\dim\sigma}\times |\St_{C(\sigma)}\sigma^0|$, 
\[
m^d_{k_\Sigma}(\sigma) = \dim H_c^{d}(|\St_{C(\sigma)}(\sigma^0)|;k) = \dim H_c^{d+\dim\sigma}(|\St_\Sigma (\sigma)|;k).
\]
\end{proof}

\end{document}